\documentclass[12pt]{article} 

\usepackage[utf8]{inputenc}
\usepackage[english]{babel}
\usepackage{amsfonts}
\usepackage{amstext}
\usepackage{graphicx}
\usepackage{amsmath}
\usepackage{amssymb}
\usepackage{euscript}
\usepackage{amsthm}
\usepackage{enumerate}
\usepackage{appendix}
\usepackage{tocloft}
\usepackage{amsthm}
    
\renewcommand{\Re}{\mathop{\rm Re}\nolimits}
\renewcommand{\Im}{\mathop{\rm Im}\nolimits}

\numberwithin{equation}{section}

\newtheorem{theorem}{Theorem}
\newtheorem{proposition}{Proposition}
\newtheorem{remark}{Remark}

\newtheorem{innercustomthm}{Theorem}
\newenvironment{customtheorem}[1]
{\renewcommand\theinnercustomthm{#1}\innercustomthm}
{\endinnercustomthm}

\begin{document}

\title{Asymptotic stage of modulation instability for the nonlocal nonlinear Schr\"odinger equation}

\author{Yan Rybalko$^{\dag}$ and Dmitry Shepelsky$^{\dag,\ddag}$\\
	\small\em {}$^\dag$ B.Verkin Institute for Low Temperature Physics and Engineering\\
	\small\em {} of the National Academy of Sciences of Ukraine\\
	\small\em {}$^\ddag$ V.Karazin Kharkiv National University}
\date{}

\maketitle

\begin{abstract}
We study the initial value problem for 
the integrable nonlocal nonlinear Schr\"odinger (NNLS) equation
\[
iq_{t}(x,t)+q_{xx}(x,t)+2 q^{2}(x,t)\bar{q}(-x,t)=0
\]
with symmetric boundary conditions: $q(x,t)\to Ae^{2iA^2t}$ as $x\to\pm\infty$, where $A>0$ is an arbitrary constant.
We describe the asymptotic stage of modulation instability for the NNLS equation
by computing the large-time asymptotics of the solution $q(x,t)$ of this initial value problem.
We shown that it exhibits a non-universal, in a sense, behavior: the asymptotics of $|q(x,t)|$ depends on details of the initial data $q(x,0)$.
This is in a sharp contrast with the local classical NLS equation, where
the long-time asymptotics of the solution depends on the initial value through the phase parameters only.
The main tool used in this work is the inverse scattering transform method applied in the form of the matrix Riemann-Hilbert problem.
The  Riemann-Hilbert problem associated with the original initial value problem is   analyzed asymptotically by the nonlinear steepest decent method.
\end{abstract}



\section{Introduction}
In the present paper we consider the initial value (IV) problem for the so-called focusing nonlocal nonlinear Schr\"odinger (NNLS) equation (here and below $\bar{q}$ is the complex conjugate of $q$)
\begin{subequations}
\label{fsivp}
\begin{align}
\label{fsivp-a}
& iq_{t}(x,t)+q_{xx}(x,t)+2q^{2}(x,t)\bar{q}(-x,t)=0, & & x\in\mathbb{R},\,t>0,  \\
\label{fsivp-b}
& q(x,0)=q_0(x), & &  x\in\mathbb{R}, 
\end{align}
with symmetric nonzero boundary conditions:
\begin{equation}
\label{fsivp-c}
q(x,t) \to Ae^{2iA^2t}, \quad  x\to\pm\infty,\quad t\geq 0,
\end{equation}
with some $A>0$ (throughout the paper we assume that $q(x,t)$ approaches the boundary values sufficiently fast).
\end{subequations}

The NNLS equation was introduced by Ablowitz and Musslimani in 2013 \cite{AMP} as a \textit{PT-}symmetric reduction ($r(x,t)=\bar{q}(-x,t)$) 
of the well-known Ablowitz-Kaup-Newell-Segur (AKNS) system \cite{AKNS} (a.k.a. coupled Schr\"odinger equations):
\begin{subequations}
	\label{fscs}
	\begin{align}
	\label{fscs-a}
	iq_t(x,t)+q_{xx}(x,t)+2q^2(x,t)r(x,t)=0,\\
	\label{fscs-b}
	-ir_t(x,t)+r_{xx}(x,t)+2r^2(x,t)q(x,t)=0.
	\end{align}
\end{subequations}
This  reduction leads to (\ref{fsivp-a}) and is consistent with  the \textit{PT-}symmetry condition \cite{BB}: if $q(x,t)$ is a solution of (\ref{fsivp-a}), then $\bar{q}(-x,-t)$ is a solution as well.
Therefore, the NNLS equation is related to the $PT$-symmetric theory, which is a state-of-the-art area in modern physics (see, e.g., \cite{B16, KYZ16, MA16} and references therein).
Also this equation is connected to the unconventional system
of coupled Landau-Lifshitz equations \cite{GA, R21} and it can be obtained as a small amplitude quasi-monochromatic reduction
of the nonlinear Klein-Gordon equation \cite{AM19}. Moreover, it can be viewed as a particular case of  Alice-Bob nonlocal systems \cite{Lou18}.

The IV problem (\ref{fsivp-a}), (\ref{fsivp-b}) with 
the boundary conditions $q(x,t)\to q_{\pm}(t)=Ae^{i\theta_\pm(t)}$ as $x\to\pm\infty$  was considered in \cite{ALM18}, where it was shown that the
admissible
 boundary functions $q_{\pm}(t)$ are neither exponentially growing nor decaying only if $q_{\pm}(t)=Ae^{2iA^2t}$ or $q_{\pm}(t)=\pm Ae^{-2iA^2t}$. 
Notice that the spectral pictures for these two cases are significantly different. For $q_{\pm}(t)=Ae^{2iA^2t}$, the continuous spectrum consists, as in the case of the classical focusing nonlinear Schr\"odinger (NLS) equation \cite{BK14,BM17}, of the real axis $\mathbb{R}$ and the segment $[-iA, iA]$, whereas when $q_{\pm}(t)=\pm Ae^{-2iA^2t}$ the continuous spectrum is purely real and has a gap $(-A, A)$ and thus the spectral picture is similar to that for the classical defocusing NLS equation \cite{ZS73, IU88} (see Sections 3 and 4 in \cite{ALM18} for details).

The choice of the (symmetric) boundary conditions (\ref{fsivp-c}) is inspired by considerable interest in nonlinear dynamics of modulation instability (MI)
in recent years (see, e.g., \cite{BM16, ET20, GS18, KHB16, KSER19, ZG13} and references therein).
The MI (a.k.a. the Benjamin-Feir instability \cite{BF67} in the context of
water waves) is related to numerous important physical phenomena, such as envelope solitons (``bright solitons''), envelope shocks, freak (rogue) waves, effects of hydrodynamic instability, to name but a few (see, e.g., \cite{BLS21-cimp, DR09, DDEG14, KPS09, ZO09} and references therein).
For the focusing NLS equation (which is  a conventional model for studying the MI) the nonlinear stage of modulation instability was studied in \cite{BM17}, where the authors showed that in the solitonless case the large-time asymptotics of the modulus of the solution is formed solely by the boundary conditions of the problem
while only the phase parameters depend on the initial data.
Therefore, the solution exhibits the same large-time behavior for a large class of initial data and in this sense the asymptotic stage of MI is universal. 
Later, this result was numerically established for other NLS-type models \cite{BLMT18}.

For both focusing NNLS and NLS equations, the stationary wave $Ae^{2iA^2t}$ is unstable under  small perturbations.
In \cite{San18} Santini considered the periodic Cauchy problem for (\ref{fsivp-a}) and showed, using the perturbation approach, that Akhmediev-type rogue waves (see Section 2 in \cite{San18}; cf. \cite{YY} where the Peregrine-type rogue waves were obtained) are relevant for describing the evolution of the solution in an intermediate region, i.e., for $t$ such as $1\ll t\ll O(|\log\varepsilon|)$, where the perturbation of the stationary wave is $O(\varepsilon)$.
Particularly, it was rigorously demonstrated that 
since Akhmediev-type soliton solutions of the NNLS equation blow up,  in general, in finite time, the solution also blows up in the linear stage of MI.

In the present paper we study the nonlinear stage of MI for the nonlocal NLS equation. We do so by analyzing the solution $q(x,t)$ as $t\to\infty$, i.e., beyond the intermediate region considered in \cite{San18}. This can’t be done by linearizing equation (\ref{fsivp-a}), so we utilize the full force of integrability of the IV problem (\ref{fsivp}).
Applying the inverse scattering transform (IST) method in the form of the matrix Riemann-Hilbert (RH) factorization problem \cite{FT, NMPZ84}, we adapt the nonlinear steepest-descent method (Deift and Zhou method, see \cite{DZ, DIZ} and \cite{DVZ94, DVZ97, MM, MM2} for its extensions) to the associated oscillatory RH problem.
We found that the asymptotic stage of MI in the case of the NNLS equation essentially depends on the initial data. Particularly, the modulus of the leading order asymptotic term depends explicitly on $q_0(x)$ (via the associated spectral functions) in all asymptotic regions (see Theorems \ref{fsth1pw} and \ref{fsth2} below).
In this sense, the nonlinear stage of MI is \textit{non-universal} in the case of the nonlocal equation (\ref{fsivp-a}), which is in a sharp contrast with the local NLS equation, where only the phase parameters depend on the initial data \cite{BLMT18, BM17}.

To be more specific, we present here rough results on the asymptotic behavior of  $q(x,t)$ 
(see Theorems \ref{fsth1pw} and \ref{fsth2} below for the precise results).
\begin{customtheorem}{$\mathbf{1^\prime}$} (Plane wave region)
	\\
	Assuming that the initial data $q_0(x)$ is such that the solitons are absent (i.e., the associated spectral functions $a_j(k)$, $j=1,2$, see Section \ref{fssectspfunct}, have no zeros in the corresponding domains) and the winding of the argument of certain spectral function is less than $\pi$ (see Assumption (\ref{fsarg-ass})), the asymptotics of the solution $q(x,t)$ 
	of problem (\ref{fsivp})
	along the rays $\frac{x}{4t}=const$ with $|\frac{x}{4t}|>\sqrt{2}A$ has the form:
	\begin{subequations}\label{fspl1}
		\begin{align}
		&q(x,t)=Ae^{-2\Im F_{\infty}(k_1)}
		e^{2i(A^2t+\Re F_{\infty}(k_1))}
		+o(1),&& t\to\infty,\quad \frac{x}{4t}>\sqrt{2}A,\\
		&q(-x,t)=Ae^{2\Im F_{\infty}(k_1)}
		e^{2i(A^2t+\Re F_{\infty}(k_1))}
		+o(1),&& t\to\infty,\quad -\frac{x}{4t}<-\sqrt{2}A,
		\end{align}
	\end{subequations}
	where $k_{1}=\frac{1}{2}\left(-\xi-\sqrt{\xi^2-2A^2}\right)$ 
	with $\xi=\frac{x}{4t}$ for $x>0$
	and the \textit{complex} constant $F_{\infty}(k_1)$ (which depends on  $\xi$ through $k_1$) is given by (\ref{fsreimFinf}).
\end{customtheorem}

\begin{customtheorem}{$\mathbf{2^\prime}$} (Modulated elliptic wave region)\\
	Under the same assumptions as in Theorem $1^{\prime}$ (with only difference that instead of Assumption (\ref{fsarg-ass}) we make Assumption (\ref{fsDelta-argew})),
	the asymptotics of the solution $q(x,t)$ 
		of problem (\ref{fsivp})
	along the rays $\frac{x}{4t}=const$ with $0<|\frac{x}{4t}|<\sqrt{2}A$ has the form:
	\begin{subequations}\label{fsasellw1}
		\begin{align}
		\nonumber
		&q(x,t)=(A+\Im\alpha)
		e^{-2\Im G_{\infty}(k_0,\alpha)}
		\frac{\Theta(\frac{\Omega t}{2\pi}
			+\frac{\omega}{2\pi}-\frac{1}{4}-v_{\infty}+c)
			\Theta(v_{\infty}+c)}
		{\Theta(\frac{\Omega t}{2\pi}
			+\frac{\omega}{2\pi}-\frac{1}{4}+v_{\infty}+c)
			\Theta(-v_{\infty}+c)}\\
		&\qquad\qquad
		\times e^{2i(tH_{\infty}+\Re G_{\infty}(k_0,\alpha))}
		+o(1),\quad 0<\frac{x}{4t}<\sqrt{2}A,\\
		\nonumber
		&q(-x,t)=(A+\Im\alpha)
		e^{2\Im G_{\infty}(k_0,\alpha)}
		\frac{\Theta(\frac{\Omega t}{2\pi}
			+\frac{\overline{\omega}}{2\pi}-\frac{1}{4}
			+\overline{v_{\infty}}-\overline{c})
			\Theta(-\overline{v_{\infty}}-\overline{c})}
		{\Theta(\frac{\Omega t}{2\pi}
			+\frac{\overline{\omega}}{2\pi}-\frac{1}{4}
			-\overline{v_{\infty}}-\overline{c})
			\Theta(\overline{v_{\infty}}-\overline{c})}\\
		&\qquad\qquad
		\times e^{2i(tH_{\infty}+\Re G_{\infty}(k_0,\alpha))}
		+o(1),\quad 0>-\frac{x}{4t}>-\sqrt{2}A.
		\end{align}
	\end{subequations}
	Here the genus-1 theta function $\Theta$ is given by (\ref{fsg1thf}) and the constants $\alpha$, $\Omega$, $v_{\infty}$ and $c$, which  do not depend on the initial data $q_0(x)$, are given by (\ref{fsreaima}), (\ref{fsOmega}), (\ref{fsvinfty}) and (\ref{fsc}), respectively. 
	The complex constants $\omega$, $H_{\infty}$ and $G_{\infty}(k_0,\alpha)$, which depend on the initial data $q_0(x)$, are given by (\ref{fsomega}), (\ref{fsH0}) and (\ref{fsGinfty}), respectively.
	Moreover, $G_{\infty}(k_0,\alpha)$ depends on the ray $\xi$ through $k_0$ and $\alpha$, where the former is defined as a unique solution of  equation (\ref{fsintk0}).
\end{customtheorem}
\begin{remark}
	  Assumption (\ref{fsarg-ass})  is sufficient for establishing the asymptotics for the nonlocal NLS equation in the plane wave regions
		(Theorem $1^\prime$).
	Though the asymptotics in these regions has not been considered before 
	(to the best of our knowledge) for nonlocal equations, the analysis needed for obtaining this result is close to that used in the decaying 
	\cite{RS} and ``modulated constant'' zones \cite{RS20,RSs}.
	Remarkably, a similar assumption (see Assumption (\ref{fsDelta-argew})) on the winding of the argument of the spectral function $(1+r_1(k)r_2(k))$ turns out to be  sufficient for establishing the asymptotics in the modulated elliptic wave region as well (Theorem $2^\prime$).
	Here the reasoning  is more involving; particularly, we have to deal with a non-analytic phase function and unbounded entries of the jump matrix at a neighborhood of the stationary phase point (see Appendix C, item (ii)).
\end{remark}
\begin{remark}
In the case of the NNLS equation, the reflection coefficients $r_1(k)$ and $r_2(k)$ (see Section \ref{fssectspfunct}) are not connected (in contrast with the classical NLS equation);  this ``lack of symmetry'' implies, in particular, 
that $F_{\infty}(k_1)$ and $G_{\infty}(k_0,\alpha)$ in (\ref{fspl1}) and (\ref{fsasellw1}) can be
 complex-valued and thus the modulus of the asymptotics depends on the initial data.
A similar lack of symmetry holds for other integrable nonlocal equations (see, e.g., \cite{AM17, ALM20, F16, HFX19, Y18}), which allows us to conjecture that 
the modulation instability in other nonlocal models should also exhibit
 a kind of non-universal behavior.
\end{remark}
\begin{remark}
It is known that the solution of the problem (\ref{fsivp}) can blow up in finite time.
Particularly, the ``breathing'' two-soliton solution \cite{ALM18}, Peregrine-type \cite{YY} and Akhmediev-type \cite{San18} solutions can blow up at a discrete set of points in the $(x,t)$ plane, even in the cases when the initial profile (i.e., $q(x,t=0)$) is smooth.
On the other hand,  away from  point singularities, a solution can be smooth and satisfying the boundary conditions for all $t\geq0$.
With this respect,
the Riemann-Hilbert problem  provides a formalism for constructing global
solutions outside  the  domains where they may have irregular behavior.
Indeed, (i) the RH problem is intrinsically local and (ii) the jump matrix is usually analytic w.r.t. parameter(s) (say, $x$), which imply (in view of the analytic Fredholm alternative \cite{Zh89}) that the solution of the associated RH problem is meromorphic in $x$ (or, otherwise, the problem is solvable for no $x$; see, e.g., Chapter 3, Section 1 in \cite{FIKN}).
Particularly, in this paper we obtain the extension of the solution from a neighborhood of $t=\infty$ into the sectors $\frac{x}{4t}\in\mathbb{R}\setminus\{\pm\sqrt{2}A,0\}$.
\end{remark}

The article is organized as follows. In the Section 2 we develop the inverse scattering transform method in the form of the RH problem.
In Sections 3 and 4 we obtain the long-time asymptotics of the solution in the plane wave region and modulated elliptic wave region respectively, establishing the non-universality of the asymptotic stage of the modulation instability in these zones.

\section{Inverse scattering transform and the Riemann-Hilbert problem}\label{fsist}

In this section we reformulate the IST method for the IV problem (\ref{fsivp}) which was first developed in \cite{ALM18}, in the form suitable for asymptotic analysis, particularly, keeping the  determinant of a matrix constructed from  the Jost solutions to be equal to 1. We also point out that since the boundary conditions in (\ref{fsivp}) are symmetric, the implementation of the IST method is close to that in the case of the classical NLS equation \cite{BK14, BM17}; but the associated spectral functions satisfy different symmetry relations, which affects significantly  the resulting asymptotic formulas.

\subsection{Direct scattering}

The NNLS equation (\ref{fsivp-a}) is a compatibility condition of the following system of linear equations \cite{AMP} (the so-called Lax pair)
\begin{align}
\label{fsLP}
&\Phi_{x}+ik\sigma_{3}\Phi=U\Phi,\\
&\Phi_{t}+2ik^{2}\sigma_{3}\Phi=V\Phi,
\end{align}
where $\sigma_3=\left(\begin{smallmatrix} 1& 0\\ 0 & -1\end{smallmatrix}\right)$ is the third Pauli matrix, $\Phi(x,t,k)$ is a $2\times2$ matrix-valued function, $k\in\mathbb{C}$ is a spectral parameter, and the $2\times2$ matrix coefficients $U(x,t)$ and $V(x,t,k)$ are given in terms of  
$q(x,t)$ as follows:	
\begin{equation}
U(x,t)=\begin{pmatrix}
0& q(x,t)\\
-\bar{q}(-x,t)& 0\\
\end{pmatrix},\qquad 
V(x,t,k)=\begin{pmatrix}
V_{11}(x,t)& V_{12}(x,t,k)\\
V_{21}(x,t,k)& V_{22}(x,t)\\
\end{pmatrix},
\end{equation}
where $V_{11}=-V_{22}=iq(x,t)\bar{q}(-x,t)$, $V_{12}=2kq(x,t)+iq_{x}(x,t)$, and
$V_{21}=-2k\bar{q}(-x,t)+i(\bar{q}(-x,t))_{x}$. 

Assuming that $\int_{\mathbb{R}}|q(x,t)-Ae^{2iA^2t}|\,dx<\infty$ for all $t\geq0$, introduce the $2\times2$ matrix valued functions $\Psi_j(x,t,k)$, $j=1,2$ as the solutions of the following linear Volterra integral equations
\begin{align}\label{fsPsi}
\nonumber
\Psi_j(x,t,k)=&e^{iA^2t\sigma_3}\mathcal{E}(k)\\
&+\int\limits_{(-1)^j\infty}^{x}G(x,y,t,k)(U(y,t)-U_0(t))\Psi_j(y,t,k)e^{i(x-y)f(k)\sigma_3}\,dy,j=1,2,
\end{align}
where $U_{0}(t)$ is the limits of $U(x,t)$ as $x\to\pm\infty$:
\begin{equation}
U(x,t)\to U_{0}(t),\quad x\to\pm\infty;\quad
U_0(t)=
\begin{pmatrix}
0& Ae^{2iA^2t}\\
-Ae^{-2iA^2t} & 0
\end{pmatrix}.
\end{equation}
The kernel $G(x,y,t,k)$ is defined as follows:
\begin{equation}
G(x,y,t,k)=e^{iA^2t\sigma_3}\mathcal{E}(k)
e^{-i(x-y)f(k)\sigma_3}\mathcal{E}^{-1}(k)
e^{-iA^2t\sigma_3},
\end{equation}
where
\begin{equation}\label{fsK}
\mathcal{E}(k)=\frac{1}{2}
\begin{pmatrix}
w(k)+\frac{1}{w(k)} & w(k)-\frac{1}{w(k)}\\
w(k)-\frac{1}{w(k)} & w(k)+\frac{1}{w(k)}
\end{pmatrix},\quad w(k)=\left(\frac{k-iA}{k+iA}\right)^{\frac{1}{4}},
\end{equation}
and
\begin{equation}\label{fsf}
f(k)=(k^2+A^2)^{\frac{1}{2}}.
\end{equation}
Here the functions $f(k)$ and $w(k)$ are fixed to be analytic 
in $\in\mathbb{C}\setminus \overline{ B}$, 
where 
$$
B=(-iA,iA)\subset i\mathbb{R},
$$
 and to have the 
asymptotics
$$
f(k)=k+O(k^{-1}),\quad k\to\infty\quad\text{and}\quad
w(k)=1+O(k^{-1}),\quad k\to\infty.
$$
Particularly, we have $f(k)=\sqrt{k^2+A^2}$ for $k\in(0,\infty)$ and 
$f(k)=-\sqrt{k^2+A^2}$ for $k\in(-\infty,0)$. 

In what follows, $f_\pm(k)$ and $w_\pm(k)$ denote the limiting values 
of the corresponding function
as $k$ approaches $B$ (oriented upward from $-iA$ to $iA$)
from the left/right (and similarly for $\mathcal{E}_\pm (k)$).
Notice that in spite of the jumps of $f$ and $\cal E$ 
across $B$, $G(x,y,t,k)$ is entire w.r.t. $k$ for all $x$, $y$, and $t$.

Since $f(k)$ is real for $k\in\mathbb{R}$
and $f_\pm(k)$ are real for $k\in B$, the integral equations (\ref{fsPsi}) 
are well-defined for  $k\in\mathbb{R}\cup B$.

The columns of the matrices $\Psi_j(x,t,k)$, $j=1,2$ play a crucial role in the construction of the basic Riemann-Hilbert problem:  a sectionally holomorphic matrix function can be defined in terms of the corresponding columns of $\Psi_j(x,t,k)$, $j=1,2$ (see (\ref{fsM}) below). In the next proposition we summarize the main properties of $\Psi_j(x,t,k)$, $j=1,2$ (we use the following notations: $Q^{[j]}$ stands for the $j$-th column of a matrix $Q$; $\mathbb{C}^{\pm}=\left\{k\in\mathbb{C}\,|\pm\Im k>0\right\}$;
$\overline{\mathbb{C}^{\pm}}=\left\{k\in\mathbb{C}\,|\pm\Im k\ge 0\right\}$):
\begin{proposition}
\label{fsproppsi1}
\begin{enumerate}[(i)]
\item The matrices $\Psi_{j\pm}(x,t,k)$, $j=1,2$ are well defined for $k\in B$
as the solutions of the integral equations (cf. (\ref{fsPsi}))
\[
\Psi_{j\pm}(x,t,k)=e^{iA^2t\sigma_3}\mathcal{E}_\pm(k)
+\int\limits_{(-1)^j\infty}^{x}G(x,y,t,k)(U(y,t)-U_0(t))\Psi_{j\pm}(y,t,k)
e^{i(x-y)f_\pm(k)\sigma_3}\,dy.
\]

\item The columns $\Psi_1^{[1]}(x,t,k)$ and $\Psi_2^{[2]}(x,t,k)$ are well-defined for $k\in\overline{\mathbb{C}^+}\setminus[0, iA]$, analytic for $k\in\mathbb{C}^+\setminus(0,iA]$ and continuous for $k\in\overline{\mathbb{C}^+}\setminus[0,iA]$; moreover,
\[
\Psi_1^{[1]}(x,t,k)=e^{iA^2t}
\begin{pmatrix}
1\\
0\end{pmatrix}
+O(k^{-1}),\quad \Psi_2^{[2]}(x,t,k)=
e^{-iA^2t}
\begin{pmatrix}
0\\
1\end{pmatrix}
+O(k^{-1}),\quad  
k\rightarrow\infty, \quad  k\in\mathbb{C}^+,
\]
and $\Psi_1^{[1]}(x,t,k),\Psi_2^{[2]}(x,t,k)=
O\left((k\pm iA)^{-\frac{1}{4}}\right)$ as $k\to\mp iA$.
\item The columns $\Psi_1^{[2]}(x,t,k)$ and $\Psi_2^{[1]}(x,t,k)$ are well-defined for $k\in\overline{\mathbb{C}^-}\setminus[-iA,0]$, analytic for $k\in\mathbb{C}^-\setminus[-iA,0)$ and continuous for $k\in\overline{\mathbb{C}^-}\setminus[-iA,0]$; moreover,
\[
\Psi_1^{[2]}(x,t,k)=
e^{-iA^2t}
\begin{pmatrix}
0\\
1\end{pmatrix}
+O(k^{-1}),\quad 
\Psi_2^{[1]}(x,t,k)=
e^{iA^2t}
\begin{pmatrix}
1\\
0\end{pmatrix}
+O(k^{-1}),\quad  k\rightarrow\infty,\quad k\in\mathbb{C}^-,
\]
and $\Psi_1^{[2]}(x,t,k),\Psi_2^{[1]}(x,t,k)=O\left((k\pm iA)^{-\frac{1}{4}}\right)$ as $k\to \mp iA$.
\item The functions $\Phi_j(x,t,k)$, $j=1,2$ defined by
\begin{subequations}\label{fsjost}
\begin{align}
&\Phi_j(x,t,k)=\Psi_j(x,t,k)e^{-(ix+2itk)f(k)\sigma_3},&& k\in\mathbb{R}\setminus\{0\},\quad j=1,2,\\
&\Phi_{j\pm}(x,t,k)=\Psi_{j\pm}(x,t,k)
e^{-(ix+2itk)f_\pm(k)\sigma_3},&&k\in B,\quad j=1,2,
\end{align}
\end{subequations}
are the  (Jost) solutions of the Lax pair equations (\ref{fsLP}) satisfying 
the boundary conditions
\begin{subequations}
\begin{align}
&\Phi_j(x,t,k)\rightarrow\Phi_0(x,t,k),&& x\rightarrow\pm\infty,\quad j=1,2,\quad
k\in\mathbb{R}\setminus\{0\},\\
&\Phi_{j\pm}(x,t,k)\rightarrow\Phi_{0\pm}(x,t,k),&& x\rightarrow\pm\infty,\quad j=1,2,\quad
k\in B,
\end{align}
\end{subequations}
where $\Phi_0(x,t,k)=e^{iA^2t\sigma_3}
\mathcal{E}(k)e^{-(ix+2itk)f(k)\sigma_3}$ and 
$\Phi_{0\pm}(x,t,k)=e^{iA^2t\sigma_3}
\mathcal{E}_{\pm}(k)e^{-(ix+2itk)f_\pm(k)\sigma_3}$.
		
\item $\det\Psi_j(x,t,k)\equiv 1,\quad k\in\mathbb{R} \quad$ and $\quad\det\Psi_{j\pm}(x,t,k)\equiv 1,\quad k\in B, \quad j=1,2$.

\item The  following symmetry relations hold:
\begin{subequations}\label{fssymmpsi}
\begin{equation}
\label{fssymmR}
\begin{aligned}
&\sigma_1\overline{\Psi_1^{[1]}(-x,t,-\bar{k})}=
\Psi_2^{[2]}(x,t,k), 
\quad k\in\overline{\mathbb{C}^+}\setminus[0,iA],\\
&\sigma_1\overline{\Psi_1^{[2]}(-x,t,-\bar{k})}=
\Psi_2^{[1]}(x,t,k), 
\quad k\in\overline{\mathbb{C}^-}\setminus[-iA,0],
\end{aligned}
\end{equation}
and 
\begin{equation}
\label{fssymmseg}
\Psi_{j+}(x,t,k)=i\Psi_{j-}(x,t,k)\sigma_1,\quad k\in B,\quad j=1,2,
\end{equation}
\end{subequations}
where $\sigma_1=\bigl(\begin{smallmatrix}0& 1\\1 & 0\end{smallmatrix}\bigl)$ is the first Pauli matrix.
\end{enumerate}
\end{proposition}
\begin{proof}
Items (i)--(iv) follow directly from the integral representation (\ref{fsPsi}).  Since the matrix $U(x,t)$ is traceless  and 
$\det \mathcal{E}(k)=1$,  Item (v) follows. 
Finally, (\ref{fssymmR}) in Item (vi) follows from the symmetries 
\begin{equation}
\sigma_1 \overline{U}(-x,t)\sigma_1^{-1}=-U(x,t),\quad \text{and}\quad 
\sigma_1 \overline{G(-x,-y,t,-\bar{k})}\sigma_1^{-1}=G(x,y,t,k),\, k\in\mathbb{C}, 
\end{equation}
whereas (\ref{fssymmseg}) is a consequence of the symmetry
\begin{equation}\label{fssymK}
\mathcal{E}_+(k)=i\mathcal{E}_-(k)\sigma_1,\quad k\in B.
\end{equation}
\end{proof}

\subsection{Spectral functions}\label{fssectspfunct}
The Jost solutions $\Phi_j(x,t,k)$, $j=1,2$ of the Lax pair (\ref{fsLP}) are related by a matrix independent of $x$ and $t$, which allows us to introduce 
the  scattering matrices $S(k)$  and $S_\pm(k)$ as follows:
\begin{equation}\label{fssr}
\Phi_1(x,t,k)=\Phi_2(x,t,k)S(k),\quad k\in\mathbb{R}\setminus\{0\};
\quad
\Phi_{1\pm}(x,t,k)=\Phi_{2\pm}(x,t,k)S_\pm(k),\quad  k\in B.
\end{equation}
In terms of $\Psi_j(x,t,k)$, $j=1,2$,
the scattering relations (\ref{fssr}) read  as follows (see (\ref{fsjost})):
\begin{subequations}\label{fssrpsi}
\begin{align}
\label{fssrpsia}
&\Psi_1(x,t,k)=\Psi_2(x,t,k)e^{-(ix+2itk)f(k)\sigma_3}S(k)
e^{(ix+2itk)f(k)\sigma_3},&& k\in\mathbb{R}\setminus\{0\},\\
&
\label{fssrpsib} \Psi_{1\pm}(x,t,k)=\Psi_{2\pm}(x,t,k)e^{-(ix+2itk)f_\pm(k)\sigma_3}S_\pm(k)
e^{(ix+2itk)f_\pm(k)\sigma_3},&& k\in B.
\end{align}
\end{subequations}

Let us denote the entries of $S(k)$ by
\begin{equation}\label{fsscatt}
S(k)=\begin{pmatrix}
a_1(k) & -b_2(k)\\
b_1(k) & a_2(k)
\end{pmatrix},\quad k\in\mathbb{R}\setminus\{0\};\quad
S_\pm(k)=\begin{pmatrix}
a_{1\pm}(k) & -b_{2\pm}(k)\\
b_{1\pm}(k) & a_{2\pm}(k)
\end{pmatrix},\quad k\in B.
\end{equation}
Taking into account the analytical properties of the columns of matrices $\Psi_j$, $j=1,2$ (see Items (ii), (iii) in Proposition \ref{fsproppsi1}), the function $a_1(k)$ is analytic in $\mathbb{C}^{+}\setminus(0, iA]$ and $a_2(k)$ is analytic in $\mathbb{C}^{-}\setminus[-iA, 0)$ whereas $b_j(k)$, $j=1,2$ are well-defined for $k\in\mathbb{R}\setminus\{0\}$. Indeed, (\ref{fssrpsia}) implies
\begin{subequations}
\begin{align}
&a_1(k)=\det\left(\Psi_1^{[1]}(0,0,k),\Psi_2^{[2]}(0,0,k)\right),\quad
k\in\overline{\mathbb{C}^{+}}\setminus[0, iA],\\
&a_2(k)=\det\left(\Psi_2^{[1]}(0,0,k),\Psi_1^{[2]}(0,0,k)\right),\quad
k\in\overline{\mathbb{C}^{-}}\setminus[-iA, 0],\\
&b_1(k)=\det\left(\Psi_2^{[1]}(0,0,k),\Psi_1^{[1]}(0,0,k)\right),\quad
k\in\mathbb{R}\setminus\{0\},\\
&b_2(k)=\det\left(\Psi_2^{[2]}(0,0,k),\Psi_1^{[2]}(0,0,k)\right),\quad
k\in\mathbb{R}\setminus\{0\}.
\end{align}
\end{subequations}
Moreover, $a_j(k)$, $b_j(k)$, $j=1,2$ have the following asymptotics as $k\to\infty$:
\begin{subequations}
\begin{align}
&a_1(k)=1+O(k^{-1}),&k&\to\infty,\quad k\in \overline{\mathbb{C}^{+}},\\
&a_2(k)=1+O(k^{-1}),&k&\to\infty,\quad k\in \overline{\mathbb{C}^{-}},\\
&b_j(k)=O(k^{-1}),&k&\to\infty,\quad k\in\mathbb{R}.
\end{align}
\end{subequations}

For $k\in B$, we define 
 $a_{j\pm}(k)$ and $b_{j\pm}(k)$, $j=1,2$  with the aid of the similar determinant representations:
\begin{subequations}
	\begin{align}
	&a_{1\pm}(k)=\det\left(\Psi_{1\pm}^{[1]}(0,0,k),\Psi_{2\pm}^{[2]}(0,0,k)\right),
	\quad k\in B,\\
	&a_{2\pm}(k)=\det\left(\Psi_{2\pm}^{[1]}(0,0,k),\Psi_{1\pm}^{[2]}(0,0,k)\right),
	\quad k\in B,\\
	&b_{1\pm}(k)=\det\left(\Psi_{2\pm}^{[1]}(0,0,k),\Psi_{1\pm}^{[1]}(0,0,k)\right),
	\quad k\in B,\\
	&b_{2\pm}(k)=\det\left(\Psi_{2\pm}^{[2]}(0,0,k),\Psi_{1\pm}^{[2]}(0,0,k)\right),
	\quad k\in B.
	\end{align}
\end{subequations}
In this way, $a_{1\pm}(k)$ turn to be the limiting values of $a_1(k)$ 
from the left/right for $k\in (0, iA)$ and 
$a_{2\pm}(k)$ are the limiting values of $a_2(k)$ 
from the left/right for $k\in (-iA,0)$.

The symmetries (\ref{fssymmpsi}) imply the following symmetries of the spectral functions (cf. \cite{ALM18}):
\begin{subequations}\label{fsabsym}
\begin{align}
\label{fsajsym}
&\overline{a_1(-\bar{k})}=a_1(k),\quad k\in \overline{\mathbb{C}^{+}}\setminus[0, iA],\quad
\overline{a_2(-\bar{k})}=a_2(k),\quad k\in \overline{\mathbb{C}^{-}}\setminus[-iA, 0],\\
\label{fsbjsymm}
& b_2(k)=\overline{b_1(-k)},\quad k\in\mathbb{R}\setminus\{0\},
\end{align}
\end{subequations}
and
\begin{equation}
\label{fsajsymB}
a_{1\pm}(k)=a_{2\mp}(k),
\quad b_{1\pm}(k)=-b_{2\mp}(k),\quad k\in B.
\end{equation}
Notice that (\ref{fsajsym}) implies $\overline{a_{j+}(0)}=a_{j-}(0)$, $j=1,2$;   combining this with (\ref{fsajsymB}) we arrive at the equality 
$a_{1\pm}(0)=a_{2\pm}(0)$.

From Item (v) of Proposition \ref{fsproppsi1}, (\ref{fsjost}) and (\ref{fssr})
it follows that 
$a_j(k)$ and $b_j(k)$, $j=1,2$ 
satisfy the determinant relations:
\begin{equation}
a_1(k)a_2(k)+b_1(k)b_2(k)=1,\ \  k\in\mathbb{R}\setminus\{0\};\quad
a_{1\pm}(k)a_{2\pm}(k)+b_{1\pm}(k)b_{2\pm}(k)=1,
\ \  k\in B.
\end{equation}

Notice that due to the Schwarz symmetry breaking for the solutions $\Psi_j(x,t,k)$, $j=1,2$, see (\ref{fssymmR}),  the values of $a_1(k)$ for $k\in\mathbb{C}^+\setminus(0,iA]$ and $a_2(k)$ for $k\in\mathbb{C}^-\setminus(0,-iA]$ are, in general, \textit{not} 
related.

Finally, we point out that $a_j(k),\,b_j(k)=O\left((k\pm iA)^{-\frac{1}{2}}\right)$ as $k\to\mp iA$.

\underline{\textbf{Notations}}
From now on we identify the values of $f(k)$, $w(k)$, $a_j(k)$ and $b_j(k)$, $j=1,2$ for  $k\in B$ as the limiting values of the corresponding function from the right:
\begin{equation}\label{fsfj}
f(k):=f_-(k)=\sqrt{k^2+A^2},\quad k\in B,\quad
w(k):=w_-(k) ,\quad k\in B.
\end{equation}
and
\begin{equation}\label{fsajB}
a_j(k):=a_{j-}(k),\quad j=1,2,\quad k\in B,\quad
b_j(k):=b_{j-}(k),\quad j=1,2,\quad k\in B.
\end{equation}

\underline{\textbf{Assumption}}
\textit{(Zeros of the spectral functions $a_j(k), j=1,2$).}
We assume that $a_1(k)$ and $a_2(k)$ do not have zeros in $\overline{\mathbb{C}^+}$ and $\overline{\mathbb{C}^-}$ respectively (see, however, Remark \ref{remz} below).
Particularly, there are no spectral singularities (in other words, no zeros on $\mathbb{R}\cup \overline{B}$), particularly the virtual levels are absent (i.e., $a_1(iA)\neq0$ and $a_2(-iA)\neq0$).

\begin{remark}
	Above Assumption is motivated by the fact that the pure background (i.e., $q(x,0)\equiv A$) gives rise to the spectral functions $a_j(k)\equiv 1$ and $b_j(k)\equiv0$, $j=1,2$.
	Indeed, for such initial data $\Psi_j(0,0,k)=\mathcal{E}(k)$, $j=1,2$, therefore (\ref{fssrpsi}) implies that $S(k)=\mathcal{E}^{-1}(k)\mathcal{E}(k)=I$.
\end{remark}
\subsection{Riemann-Hilbert problem}

Now we are in a position to formulate the basic Riemann-Hilbert problem. Assuming that $a_j(k)$, $j=1,2$ have no zeros in the corresponding half-planes,  define a $2\times 2$ matrix function $M(x,t,k)$ as follows:
\begin{equation}
\label{fsM}
M(x,t,k)=
\left\{
\begin{array}{lcl}
e^{-iA^2t\sigma_3}
\left(\frac{\Psi_1^{[1]}(x,t,k)}{a_{1}(k)},\Psi_2^{[2]}(x,t,k)\right),\quad k\in\mathbb{C}^+\setminus(0,iA],\\
e^{-iA^2t\sigma_3}
\left(\Psi_2^{[1]}(x,t,k),\frac{\Psi_1^{[2]}(x,t,k)}{a_{2}(k)}\right),\quad k\in\mathbb{C}^-\setminus[-iA, 0).\\
\end{array}
\right.
\end{equation}
Then the scattering relation (\ref{fssrpsi}) imply that $M(x,t,k)$ satisfies the  multiplicative jump condition across the contour $\mathbb{R}\cup B$:
\begin{equation}\label{fsj}
M_+(x,t,k)=M_-(x,t,k)J(x,t,k),\quad k\in \mathbb{R}\cup B,
\end{equation}
where $M_{\pm}(\cdot,\cdot,k)$ denotes the nontangental limits of $M(\cdot,\cdot,k)$ as $k$ approaches the contour from the corresponding side (the real axis $\mathbb{R}$ is oriented from left to right and the segment $B$ is oriented upwards).
The jump matrix $J(x,t,k)$ has the form (follows from the scattering relations (\ref{fssrpsi}) and the symmetry condition (\ref{fssymmseg})):
\begin{equation}\label{fsjbrh}
J(x,t,k)=
\begin{cases}
\begin{pmatrix}
1+r_1(k)r_2(k) & r_2(k)e^{-(2ix+4itk)f(k)}\\
r_1(k)e^{(2ix+4itk)f(k)} & 1
\end{pmatrix}, &k\in\mathbb{R}\setminus\{0\},
\\
\begin{pmatrix}
-ir_2(k)e^{-(2ix+4itk)f(k)} & i\\
i(1+r_1(k)r_2(k)) & -ir_1(k)e^{(2ix+4itk)f(k)}
\end{pmatrix},& k\in(0,iA),
\\
\begin{pmatrix}
ir_2(k)e^{-(2ix+4itk)f(k)} & i(1+r_1(k)r_2(k))\\
i& ir_1(k)e^{(2ix+4itk)f(k)}
\end{pmatrix},& k\in(-iA,0),
\end{cases}
\end{equation}
with the reflection coefficients $r_j(k)=\frac{b_j(k)}{a_j(k)}$, $j=1,2$ (recall that the values of $r_j(k)$, $j=1,2$ and $f(k)$ are taken from the ``$-$'' (right) side of the segment $B$, see (\ref{fsfj}) and (\ref{fsajB})). Notice that $r_j(k)$, $j=1,2$ are bounded as $k$ approaches $\pm iA$:
\begin{equation}
r_j(k)=O(1),\quad k\to\pm iA.
\end{equation}
The matrix $M(x,t,k)$ satisfies the normalization condition
\begin{equation}\label{fsnorm}
M(x,t,k)=I+O(k^{-1}),\quad k\to\infty.
\end{equation}
Besides, since $k=\pm iA$ are not virtual levels (see Assumption in Section \ref{fssectspfunct} and Remark \ref{remz}), the possible singularities of 
$M(x,t,k)$ at $k=\pm iA$ are  square integrable:
\begin{equation}\label{fssing}
M(x,t,k)=O\left((k\pm iA)^{-\frac{1}{4}}\right),\quad k\to\mp iA.
\end{equation}

The basic Riemann-Hilbert problem consists in finding a sectionally holomorphic $2\times 2$ matrix $M(x,t,k)$, which satisfies (i) the multiplicative jump conditions (\ref{fsj}) given in terms of known reflection coefficients $r_j(k)$, $j=1,2$ and (ii) the normalization at infinity (\ref{fsnorm});
moreover,  (iii) it is square integrable (satisfying (\ref{fssing})) at the endpoints $k=\pm iA$.

Assuming that the solution $M(x,t,k)$ of the basic Riemann-Hilbert problem (\ref{fsj}), (\ref{fsnorm}), (\ref{fssing}) exists, the solution $q(x,t)$ of the original initial value problem (\ref{fsivp}) can be found via the $(12)$ or $(21)$ entry of $M(x,t,k)$ as $k\to\infty$ (follows from the first equation in the Lax pair (\ref{fsLP})):
\begin{subequations}\label{fsasol}
\begin{equation}\label{fssol}
q(x,t)=2ie^{2iA^2t}
\lim_{k\to\infty}kM_{12}(x,t,k),
\end{equation}
and
\begin{equation}\label{fssol1}
q(-x,t)=-2ie^{2iA^2t}
\lim_{k\to\infty}k\overline{M_{21}(x,t,k)}.
\end{equation}
\end{subequations}

Notice that solution of the basic Riemann-Hilbert is unique, if it exists. Indeed, let $M$ and $\tilde M$ be
two solutions. Then, due to (\ref{fssing}), the possible singularities of 
$M \tilde M^{-1}$  at the endpoints $k=\pm iA$ are weak, and $M \tilde M^{-1}$ has no jump in $\mathbb{C}$. Then, by the Liouville theorem,  $M \tilde M^{-1}= I$ 
for all $k\in\mathbb{C}$.

\begin{remark}\label{fsxgz}
From  (\ref{fsasol}) it follows that for obtaining the solution $q(x,t)$ of (\ref{fsivp}) for all $x\in \mathbb R$, it is enough to have the solution of the Riemann-Hilbert problem for $x\ge 0$ only.
\end{remark}
\begin{remark}
	\label{remz}
	Generically, $a_1(k)$ and $a_2(k)$ can have finite number of simple zeros in $\overline{\mathbb{C}^+}\setminus(0,iA]$ and $\overline{\mathbb{C}^-}\setminus(0,-iA]$ respectively (notice that since $a_1(k)$ and $a_2(k)$ are not related, their zeroes do not, in general, constitute pairs associated with solitons, as it takes place for the  classical NLS equation \cite{BLM21}).
	In this case the basic Riemann-Hilbert problem has to be supplemented by  appropriate residue conditions defined in terms of the (known) zeros and norming constants.
On the other hand,  spectral singularities may arise for non-generic initial data
(similarly to 
 the NLS equation with zero background \cite{Zh89-sing}); particularly, virtual levels (cf. Appendix B in \cite{BK14}) 
	introduce non-square-integrable singularities 
	at the endpoints $k=\pm iA$ 
	in the basic RH problem for $M(x,t,k)$. 
	To deal with singularities of all types in a systematic way, one can reformulate the basic RH problem in such a way that all singularities are replaced by additional jump conditions on an auxiliary part of the contour, which is a circle centered at $k=0$ with sufficiently large radius 
	$r=r(q_0(\cdot))$ (see \cite{BM19, BS04, Zh98}).
	In this way, the original RH problem becomes regular;
	 however, its direct asymptotic analysis (without resorting to residue conditions) is problematic.
\end{remark}

\begin{remark}
	Since $a_1(k)$ and $a_2(k)$ are not related, the spectral picture for the nonlocal problem is more rich and variegated comparing with the local case.
	First, the number of zeros in the upper and in the lower half planes can be different (in view of the symmetry relations (\ref{fsajsym}), zeros are symmetric w.r.t. $i\mathbb{R}$, but not w.r.t. $\mathbb{R}$ as for the classical NLS equation).
	Moreover, even if zeros of $a_1(k)$ and $a_2(k)$ constitute appropriate pairs, solitons associated to these zeros can blow up in finite time (see (4.62) in \cite{ALM18}).
	Cases with equal number of purely imaginary zeros in $\mathbb{C}^{+}$ and $\mathbb{C}^{-}$ with the absolute values which are greater or equal to $A$ give rise to three different types of rogue waves 
	given in Theorems 1-3 in \cite{YY}.
	Notice that apart from the case when the absolute values of these zeros are less than $A$,  it is also possible
	that, say, $a_1(k)$ has a (simple) zero at $k=ih$, $0<h<A$ and $a_2(k)$ has a zero at $k=-i\tilde{h}$, $\tilde{h}=A$ or $\tilde{h}>A$. Consequently, there can be 
	a ``mixture'' of spectral portraits, which in the local case correspond to the Kuznetsov-Ma (zeros at $k=\pm ih$, $h>A$), Peregrine (zeros at $k=\pm iA$) or Akhmediev (zeros at $k=\pm ih$, $0<h<A$) breathers.
\end{remark}
\section{Long-time asymptotics}

The implementation of the Deift--Zhou nonlinear steepest descent method is guided (by analogy with the classical steepest decent method) by the signature structure of the phase function in the jump matrix $J(x,t,k)$ given by (\ref{fsjbrh}), which is similar to that  in the case of the NLS equation 
with nonzero background \cite{BM17}. 

Introduce the phase function $\theta(k,\xi)$ by (cf. (2.34) in \cite{BM17})
\begin{equation}\label{fstheta}
2ixf(k)+4itkf(k)=2it(4\xi+2k)f(k)\equiv 2it\theta(k,\xi),
\end{equation}
where  $\xi=\frac{x}{4t}$.
The signature table of $\Im\theta(k,\xi)$ has qualitatively different form for $\xi>\sqrt{2}A$, $\xi=\sqrt{2}A$, $\sqrt{2}A>\xi>0$, and $\xi=0$, see Figure \ref{sign_pw} and Figure \ref{sign_ew} (cf. Figure 3.2 in \cite{BM17}). 
For $\xi\ge\sqrt{2}A$, the stationary phase points (i.e., the points $k_s\in\mathbb{C}$ such that $\theta_k^\prime(k_s,\xi)=0$)  are real and have the values
\begin{equation}
k_{1}=\frac{1}{2}\left(-\xi-\sqrt{\xi^2-2A^2}\right),\quad
k_{2}=\frac{1}{2}\left(-\xi+\sqrt{\xi^2-2A^2}\right),
\end{equation}
whereas for $\sqrt{2}A>\xi\ge 0$ the stationary phase points are complex conjugate:
\begin{equation}\label{fsstphc}
\tilde{k}_{1}=\frac{1}{2}\left(-\xi+i\sqrt{2A^2-\xi^2}\right),\quad
\tilde{k}_{2}=\frac{1}{2}\left(-\xi-i\sqrt{2A^2-\xi^2}\right).
\end{equation}
\begin{figure}[h]
	\begin{minipage}[h]{0.49\linewidth}
		\centering{\includegraphics[width=0.99\linewidth]{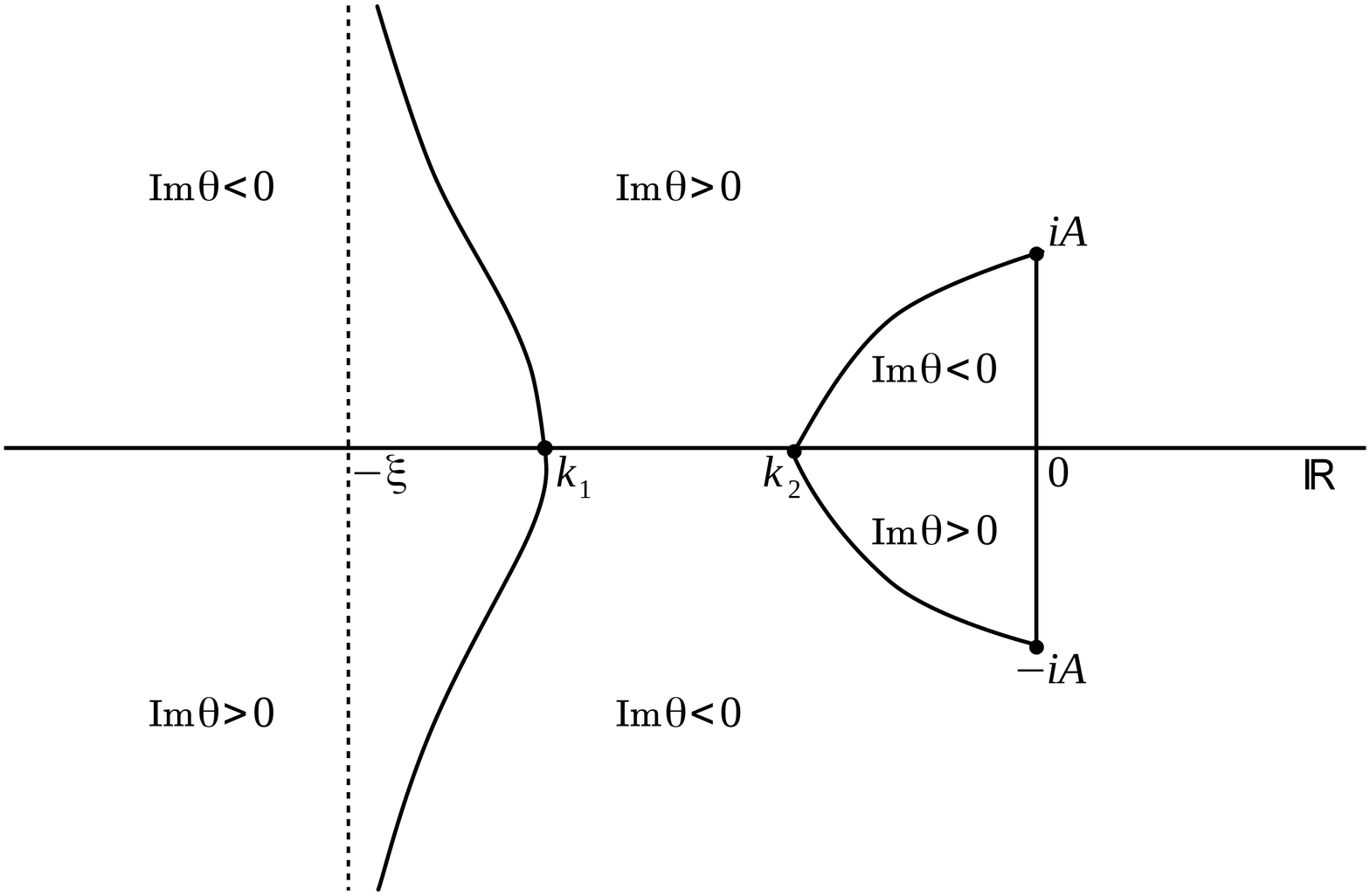}}
		\caption{Signature table of the phase function $\theta(k,\xi)$ in the plane wave region: $\xi>\sqrt{2}A$.}
		\label{sign_pw}
	\end{minipage}
	\hfill
	\begin{minipage}[h]{0.49\linewidth}
		\centering{\includegraphics[width=0.99\linewidth]{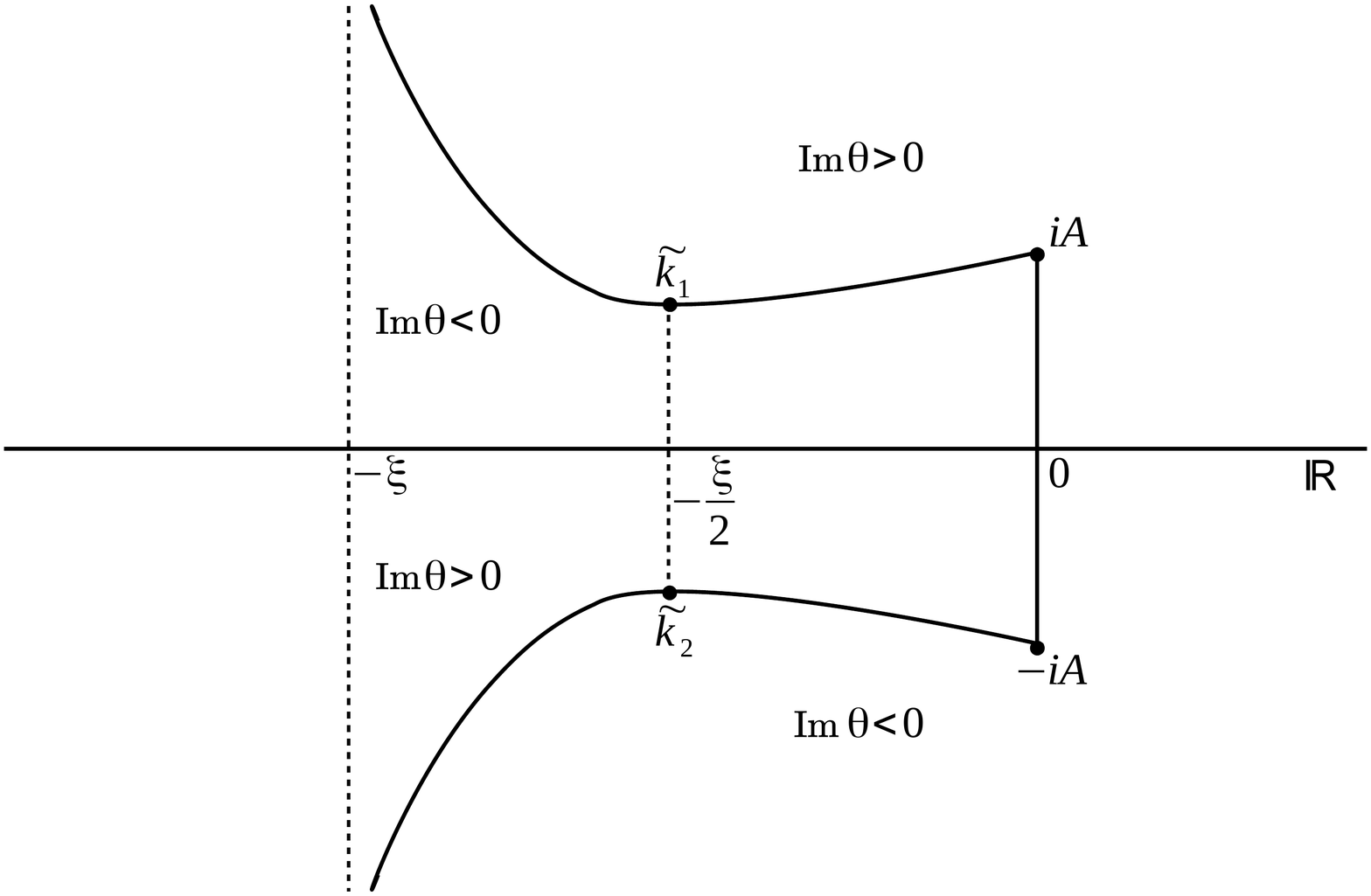}}
		\caption{Signature table of the phase function $\theta(k,\xi)$ in the elliptic wave region: $\sqrt{2}A>\xi>0$.}
		\label{sign_ew}
	\end{minipage}
\end{figure}

	In the present paper we consider the asymptotics of the solution $q(x,t)$ along the rays $\frac{x}{4t}=const$ for $|\frac{x}{4t}|>\sqrt{2}A$ (plane wave region) and for $0<|\frac{x}{4t}|<\sqrt{2}A$ (modulated elliptic wave region).
	The  transition zones (containing $x=0$ and $\frac{x}{4t}=\pm \sqrt{2}A$) will be addressed elsewhere.

\subsection{Plane wave regions: $|\frac{x}{4t}|>\sqrt{2}A$}
In the present Section we investigate the large-time behavior of the solution $q(x,t)$ along the rays $\frac{x}{4t}=const$ for $|\frac{x}{4t}|>\sqrt{2}A$. In view of (\ref{fsasol}) (see also Remark \ref{fsxgz}) it is enough to consider 
$\xi=\frac{x}{4t}>\sqrt{2}A$ only.

According to the signature table of $\theta(k,\xi)$, see Figure \ref{sign_pw}, two different triangular factorizations of the jump matrix $J(x,t,k)$ needed for $k\in\mathbb{R}$ (cf. \cite{BM17, DIZ, RSs}):
\begin{subequations}\label{fstr}
	\begin{equation}
	\label{fstr1}
	J(x,t,k)=
	\begin{pmatrix}
	1& 0\\
	\frac{r_1(k)e^{2it\theta}}{1+ r_1(k)r_2(k)}& 1\\
	\end{pmatrix}
	\begin{pmatrix}
	1+ r_1(k)r_2(k)& 0\\
	0& \frac{1}{1+ r_1(k)r_2(k)}\\
	\end{pmatrix}
	\begin{pmatrix}
	1& \frac{r_2(k)e^{-2it\theta}}{1+ r_1(k)r_2(k)}\\
	0& 1\\
	\end{pmatrix},\, k<k_1,
	\end{equation}
	and
	\begin{equation}
	\label{fstr2}
	J(x,t,k)=
	\begin{pmatrix}
	1& r_2(k)e^{-2it\theta}\\
	0& 1\\
	\end{pmatrix}
	\begin{pmatrix}
	1& 0\\
	r_1(k)e^{2it\theta}& 1\\
	\end{pmatrix},\, k>k_1.
	\end{equation}
\end{subequations}
In order to get rid of the diagonal factor in (\ref{fstr1}) we introduce the auxiliary function $\delta(k,k_1)$ as the solution of the following scalar Riemann-Hilbert problem:
\begin{equation}
\label{fsdrh}
\begin{aligned}
&\delta_+(k,k_1)=\delta_-(k,k_1)(1+r_1(k)r_2(k)),&& k\in(-\infty,k_1),\\
&\delta(k,k_1)\rightarrow 1, && k\rightarrow\infty.
\end{aligned}
\end{equation}
Though problem (\ref{fsdrh}) looks similar to that 
in the case of the local NLS equation, it has an important distinction.
Namely, the jump $(1+r_1(k)r_2(k))$ is not, in general, real valued for $k\in(-\infty,k_1)$ (as it holds for the local NLS equation due to the inherent symmetries of the spectral functions, see \cite{BK14, BM17}).
The nonzero imaginary part gives rise to the singularity of $\delta$ (or $\delta^{-1}$, depending on the sign) at the endpoint $k=k_1$.
Indeed, by using the Plemelj-Sokhotski formula we obtain the following integral representation for $\delta(k,k_1)$:
\begin{equation}\label{fsdelta-int}
\delta(k,k_1)=\exp\left\{\frac{1}{2\pi i}\int_{-\infty}^{k_1}
\frac{\ln(1+r_1(\zeta)r_2(\zeta))}{\zeta-k}\,d\zeta
\right\}.
\end{equation}
Integrating by parts one concludes that
\begin{equation}\label{fsdelta-singular}
\delta(k,k_1)=
(k-k_1)^{i\nu(k_1)}e^{\chi(k,k_1)},
\end{equation}
where
\begin{equation}\label{fschi}
\chi(k,k_1)=-\frac{1}{2\pi i}\int_{-\infty}^{k_1}\ln(k-\zeta)d_{\zeta}\ln(1+ r_1(\zeta)r_2(\zeta)),
\end{equation}
and 
\begin{equation}\label{fsnu}
\nu(k_1)=-\frac{1}{2\pi}\ln(1+r_1(k_1)r_2(k_1))
= -\frac{1}{2\pi}\ln|1+r_1(k_1)r_2(k_1)|-
\frac{i}{2\pi}\Delta(k_1),
\end{equation}
with 
\begin{equation}\label{fsDelta-arg}
\Delta(k_1)=\int_{-\infty}^{k_1}d\arg(1+ r_1(\zeta)r_2(\zeta)).
\end{equation}
Since $\Delta(k_1)$ is not, in general, equal to zero, the matrix function $\delta^{\sigma_3}(k,k_1)$ is not bounded at $k=k_1$.

In order to justify the asymptotics in the plane wave region (see Theorem \ref{fsth1pw} below) we need the following additional

\underline{\textbf{Assumption}} 
\textit{(Winding of the argument, plane wave region).}
We impose the following restriction on the winding of the argument of $(1+r_1(k)r_2(k))$ for $k<-\frac{A}{\sqrt{2}}$ in the plane wave region:
\begin{equation}\label{fsarg-ass}
\int_{-\infty}^{k}d\arg(1+ r_1(\zeta)r_2(\zeta))\in(-\pi,\pi),\quad \text{for all}\ \ k\in\left(-\infty,-\frac{A}{\sqrt{2}}\right).
\end{equation}
This implies that $|\Im \nu(k_1)|<\frac{1}{2}$ and, consequently, $\delta^{\sigma_3}(k,k_1)$ has a square integrable singularity at $k=k_1$: 
$$
\delta(k,k_1)=O\left((k-k_1)^{-\frac{1}{2}+\varepsilon(k_1)}\right),\quad k\to k_1,\,\varepsilon(k_1)>0.
$$

Using the function $\delta(k,k_1)$ we make the first transformation of $M(x,t,k)$ as follows:
\begin{equation}\label{fsm1def}
M^{(1)}(x,t,k)=M(x,t,k)\delta^{-\sigma_3}(k,k_1),\quad k\in\mathbb{C}\setminus\left\{\mathbb{R}\cup \overline{B}\right\},
\end{equation}
where $\overline{B}=[-iA,iA]$. Then $M^{(1)}$ solves the following Riemann-Hilbert problem:
\begin{subequations}
\begin{align}
&M^{(1)}_+(x,t,k)=M^{(1)}_-(x,t,k)J^{(1)}(x,t,k),&& k\in\mathbb{R}\cup B,\\
&M^{(1)}(x,t,k)=I+O(k^{-1}), &&k\to\infty,\\
&M^{(1)}(x,t,k)=O\left((k\pm iA)^{-\frac{1}{4}}\right),&& k\to\mp iA,\\
\label{fsm1k1}
&M^{(1)}(x,t,k)=O\left((k-k_1)^{-\frac{1}{2}+\varepsilon(k_1)}\right),&& k\to k_1,\,\varepsilon(k_1)>0,
\end{align}
\end{subequations}
with the jump matrix $J^{(1)}(x,t,k)$ in the form
\begin{equation}
J^{(1)}=
\begin{cases}
\begin{pmatrix}
1& 0\\
\frac{r_1(k)\delta_-^{-2}(k,k_1)}{1+ r_1(k)r_2(k)}e^{2it\theta}& 1\\
\end{pmatrix}
\begin{pmatrix}
1& \frac{r_2(k)\delta_+^{2}(k,k_1)}{1+ r_1(k)r_2(k)}e^{-2it\theta}\\
0& 1\\
\end{pmatrix},& k<k_1,\\

\begin{pmatrix}
1& r_2(k)\delta^2(k,k_1)e^{-2it\theta}\\
0& 1\\
\end{pmatrix}
\begin{pmatrix}
1& 0\\
r_1(k)\delta^{-2}(k,k_1)e^{2it\theta}& 1\\
\end{pmatrix}, &k>k_1,\\
\begin{pmatrix}
-ir_2(k)e^{-2it\theta} & i\delta^{2}(k,k_1)\\
i(1+r_1(k)r_2(k))\delta^{-2}(k,k_1) & -ir_1(k)e^{2it\theta}
\end{pmatrix},& k\in(0,iA),\\
\begin{pmatrix}
ir_2(k)e^{-2it\theta} & i(1+r_1(k)r_2(k))\delta^{2}(k,k_1)\\
i\delta^{-2}(k,k_1)& ir_1(k)e^{2it\theta}
\end{pmatrix},& k\in(-iA,0).
\end{cases}
\end{equation}

Now we are in a position to start the so-called ``opening lenses'' procedure: the Riemann-Hilbert problem is transformed in such a way that 
the jump matrix across the new contours significantly simplifies as 
$t\to\infty$, and the original RH problem can be approximated by an exactly solvable one with well controlled errors. 
For this purpose we need the spectral functions $r_j(k)$, $j=1,2$ to be analytically continued at least in a neighborhood of the contour $\mathbb{R}\cup B$. 
This takes place, particularly, when the initial data $q_0(x)$ is a compact perturbation of the boundary values: for such $q_0(x)$, the reflection coefficients $r_j(k)$, $j=1,2$ can be continued to the whole complex plane.

Assuming that $r_j(k)$, $j=1,2$ can be continued at least in a 
neighborhood of $\mathbb{R}\cup B$, we introduce $M^{(2)}(x,t,k)$ as follows (we drop the arguments of $M^{(j)}(x,t,k)$, $j=1,2$):
\begin{figure}[h]
\begin{minipage}[h]{0.99\linewidth}
	\centering{\includegraphics[width=0.79\linewidth]{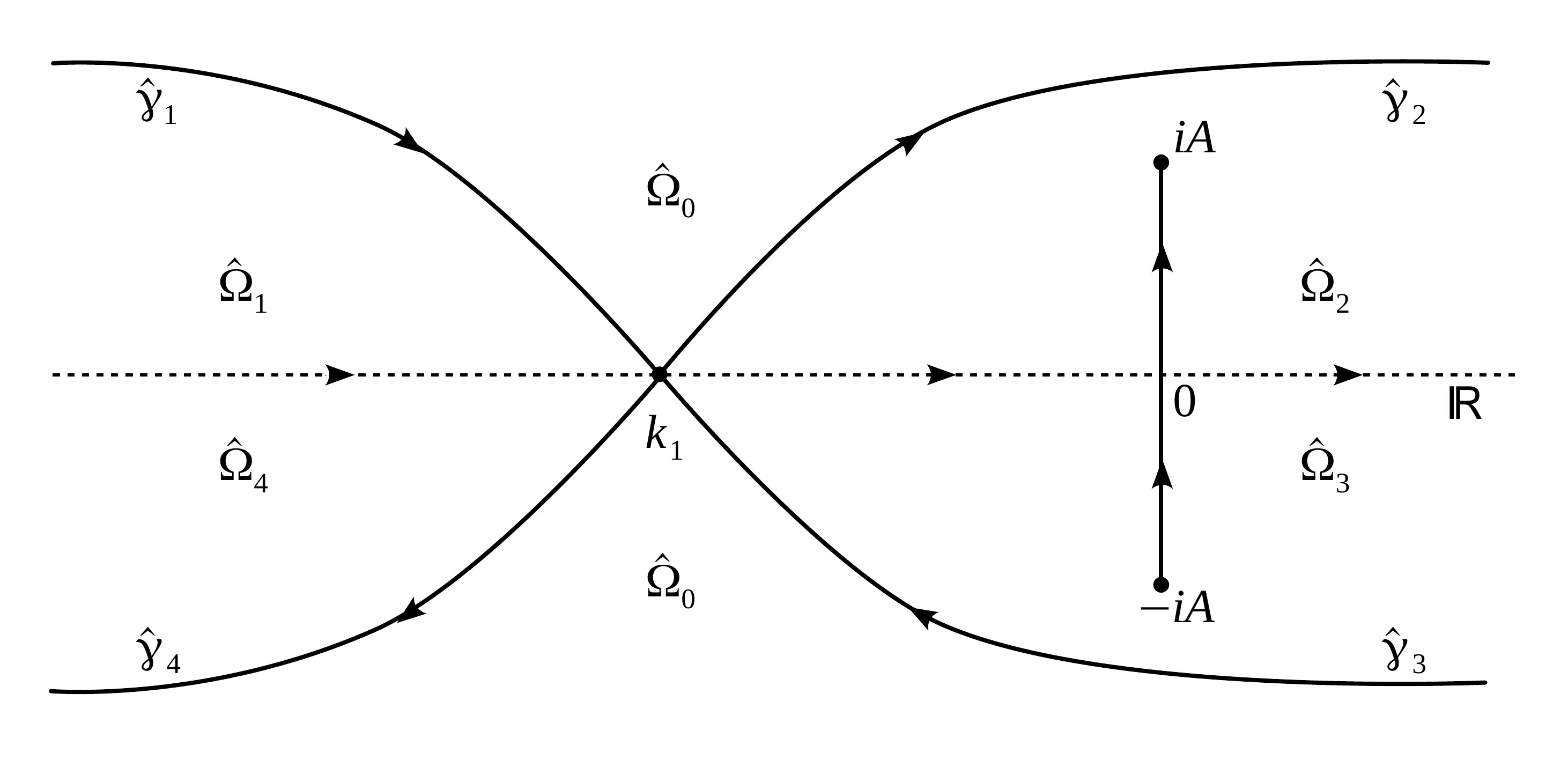}}
	\caption{Contour $\hat\Gamma=\hat\gamma_1\cup\dots\cup\hat\gamma_4$ and domains $\Omega_j$, $j=0,\dots,4$ in the plane wave region.}
	\label{cont_1_pw}
\end{minipage}
\end{figure}
$$
M^{(2)}=
\begin{cases}
M^{(1)},\, k\in\hat\Omega_0;\quad
M^{(1)}
\begin{pmatrix}
1& \frac{-r_2(k)\delta^{2}(k,k_1)}{1+r_1(k)r_2(k)}e^{-2it\theta}\\
0& 1\\
\end{pmatrix}
,\, k\in\hat\Omega_1;
\\
M^{(1)}
\begin{pmatrix}
1& 0\\
-r_1(k)\delta^{-2}(k,k_1)e^{2it\theta}& 1\\
\end{pmatrix}
,\, k\in\hat\Omega_2;\quad
M^{(1)}
\begin{pmatrix}
1& r_2(k)\delta^2(k,k_1)e^{-2it\theta}\\
0& 1\\
\end{pmatrix}
,\, k\in\hat\Omega_3;
\\
M^{(1)}
\begin{pmatrix}
1& 0\\
\frac{r_1(k)\delta^{-2}(k,k_1)}{1+r_1(k)r_2(k)}e^{2it\theta}& 1\\
\end{pmatrix}
,\, k\in\hat\Omega_4,
\end{cases}
$$
where $\Omega_j$, $j=1,\dots,4$ are depictured in Figure \ref{cont_1_pw}.
Then $M^{(2)}(x,t,k)$ solves the following Riemann-Hilbert problem on the contour $\hat\Gamma\cup B$,
$\hat\Gamma=\hat\gamma_1\cup\dots\cup\hat\gamma_4$ (see Figure \ref{cont_1_pw}):
\begin{subequations}
\begin{align}
&M^{(2)}_+(x,t,k)=M^{(2)}_-(x,t,k)J^{(2)}(x,t,k),&&k\in\hat\Gamma\cup B,\\
&M^{(2)}(x,t,k)=I+O(k^{-1}), &&k\to\infty,\\
&M^{(2)}(x,t,k)=O\left((k\pm iA)^{-\frac{1}{4}}\right),&& k\to\mp iA,\\
&M^{(2)}(x,t,k)=O\left((k-k_1)^{-\frac{1}{2}+\varepsilon(k_1)}\right),&& k\to k_1,\,\varepsilon(k_1)>0,
\end{align}
\end{subequations}
where
\begin{equation}
\label{fsJ2}
J^{(2)}=
\begin{cases}
\begin{pmatrix}
1& \frac{r_2(k)\delta^{2}(k,k_1)}{1+r_1(k)r_2(k)}e^{-2it\theta}\\
0& 1\\
\end{pmatrix}
,\, k\in\hat\gamma_1;\quad
\begin{pmatrix}
1& 0\\
r_1(k)\delta^{-2}(k,k_1)e^{2it\theta}& 1\\
\end{pmatrix}
,\ k\in\hat\gamma_2;
\\
\begin{pmatrix}
1& -r_2(k)\delta^2(k,k_1)e^{-2it\theta}\\
0& 1\\
\end{pmatrix}
,\, k\in\hat\gamma_3;\quad
\begin{pmatrix}
1& 0\\
\frac{-r_1(k)\delta^{-2}(k,k_1)}{1+r_1(k)r_2(k)}e^{2it\theta}& 1\\
\end{pmatrix}
,\, k\in\hat\gamma_4;\\
\begin{pmatrix}
0& i\delta^{2}(k,k_1)\\
i\delta^{-2}(k,k_1)&0
\end{pmatrix},\,k\in B.
\end{cases}
\end{equation}

Notice that according to the signature table of $\theta(k,\xi)$, the jump matrix $J^{(2)}(x,t,k)$ rapidly decays to the identity matrix for $k\in\hat\Gamma$ away from any vicinity of the stationary phase point $k=k_1$. In order to arrive to an exactly solvable (model) RH problem, we have to eliminate the dependence of $k$ in the jump across $k\in B$. This can be done by introducing the function $F(k,k_1)$ as a solution of the following scalar Riemann-Hilbert problem (cf. (4.10) in \cite{BKS11}):
\begin{subequations}\label{fsFRHab}
\begin{align}\label{fsFRH}
&F_+(k,k_1)F_-(k,k_1)=\delta^{2}(k,k_1),&& k\in B,\\
&F(k,k_1)=O(1),&&k\to\pm iA,\,k\to\infty.
\end{align}
\end{subequations}
Taking the logarithm of both sides in (\ref{fsFRH}) and dividing by $f(k)$ we obtain:
$$
\left(\frac{\ln F(k,k_1)}{f(k)}\right)_+ -
\left(\frac{\ln F(k,k_1)}{f(k)}\right)_-=
-\frac{2\ln\delta(k,k_1)}{f_-(k)},\quad k\in B.
$$
Then the Plemelj-Sokhotski formula gives the following integral representation for $F(k,k_1)$ (see (\ref{fsfj})):
\begin{equation}\label{fsF}
F(k,k_1)=\exp\left\{-\frac{f(k)}{\pi i}\int_{B}\frac{\ln\delta(\zeta,k_1)}{f(\zeta)(\zeta-k)}\,d\zeta\right\},\quad k\in\mathbb{C}\setminus\overline{B},
\end{equation}
which satisfies the RH problem (\ref{fsFRH}). In order to recover  $q(x,t)$
from the solution of the RH problem, we will need the large-$k$ behavior of $F(k,k_1)$:
\begin{equation}
F(k,k_1)= e^{iF_{\infty}(k_1)}+O(k^{-1}),\quad
k\to\infty,
\end{equation}
where
\begin{equation}\label{fsFinf}
F_{\infty}(k_1)=-\frac{1}{\pi}
\int_{B}\frac{\ln\delta(\zeta,k_1)}{f(\zeta)}\,d\zeta.
\end{equation}

Taking into account (\ref{fsdelta-int}) and that the contour of integration in (\ref{fsFinf}) lies on the imaginary axis, the real and imaginary part of $F_{\infty}(k_1)$ can be written as follows (cf. purely real $g_{\infty}$ given by (4.25) in \cite{BM17} in the case of the local NLS equation):
\begin{subequations}\label{fsreimFinf}
	\begin{align}
	&\Re F_{\infty}(k_1)=-\frac{1}{2i\pi^2}\int_{B}\frac{1}{f(\zeta)}
	\int_{-\infty}^{k_1}\frac{\ln|1+r_1(s)r_2(s)|}{s-\zeta}\,ds\,d\zeta,\\
	&\Im F_{\infty}(k_1)=-\frac{1}{2i\pi^2}\int_{B}\frac{1}{f(\zeta)}
	\int_{-\infty}^{k_1}\frac{\Delta(s)}{s-\zeta}\,ds\,d\zeta,
	\end{align}
\end{subequations}
where $\Delta(s)$ is given by (\ref{fsDelta-arg}).
\begin{remark}
	Due to the ``lack of symmetry'' for $r_1(k)$ and $r_2(k)$, the imaginary part of $F_\infty(k_1)$ is, in general, nonzero.
	It is in a sharp contrast with the local case, where $\Im F_{\infty}$ (see $g_{\infty}$ given by (4.25) in \cite{BM17}) is always zero due to the symmetry $r_2(k)=\bar{r}_1(k)$ for $k\in\mathbb{R}$. 
	Consequently, the asymptotics of the modulus of $q(x,t)$ 
will depend on  details of the initial data (cf. ``modulated constant'' region in the step-like problem for the NNLS equation \cite{RSs}).  This is qualitatively different from the local case, where the asymptotics in the plane wave region depends on $q_0(x)$ through the phase shift only \cite{BM17}.
\end{remark}
\begin{remark}
	\label{fsboundf}
	Taking into account that (see Chapter I, Paragraph 8.3 in \cite{G66})
	$$
	\int_B\frac{d\zeta}{f(\zeta)(\zeta-k)}=i\pi
	\left[(k+iA)^{-\frac{1}{2}}-(k-iA)^{-\frac{1}{2}}\right]+F_0(k),
	$$
	where $F_0(k)$ is an analytic function,
	one concludes that $F(k,k_1)$ is bounded at the endpoints $k=\pm iA$.
\end{remark}

Now we define  $M^{(3)}$ with the help of $F(k,k_1)$:
\begin{equation}
M^{(3)}(x,t,k)=e^{-iF_{\infty}(k_1)\sigma_3}
M^{(2)}(x,t,k)F^{\sigma_3}(k,k_1),
\quad k\in\mathbb{C}\setminus\left\{\mathbb{R}\cup\overline{B}\right\}.
\end{equation}
Then $M^{(3)}$ solves the RH problem with the constant jump across $B$:
\begin{subequations}\label{fsM3}
\begin{align}
&M^{(3)}_+(x,t,k)=M^{(3)}_-(x,t,k)J^{(3)}(x,t,k),&&
k\in\hat\Gamma\cup B,\\
&M^{(3)}(x,t,k)=I+O(k^{-1}), &&k\to\infty,\\
&M^{(3)}(x,t,k)=O\left((k\pm iA)^{-\frac{1}{4}}\right),&& k\to\mp iA,\\
&M^{(3)}(x,t,k)=O\left((k-k_1)^{-\frac{1}{2}+\varepsilon(k_1)}\right),&& k\to k_1,\,\varepsilon(k_1)>0,
\end{align}
\end{subequations}
with 
\begin{equation}
\label{fsJ3}
J^{(3)}(x,t,k)=
\begin{cases}
i\sigma_1,&k\in B,\\
F^{-\sigma_3}(k,k_1)J^{(2)}(x,t,k)F^{\sigma_3}(k,k_1),&k\in\hat\Gamma.
\end{cases}
\end{equation}

The solution $q(x,t)$ of the Cauchy problem (\ref{fsivp}) can be found in terms of $M^{(3)}(x,t,k)$ as follows:
\begin{subequations}\label{fsasolM3}
\begin{align}\label{fssolM3}
&q(x,t)=2ie^{2iA^2t+2iF_\infty(k_1)}\lim_{k\to\infty}kM_{12}^{(3)}(x,t,k),&& x>0,\\
\label{fssol1M3}
&q(-x,t)=-2ie^{2iA^2t+2i\overline{F_\infty(k_1)}}\lim_{k\to\infty}k\overline{M_{21}^{(3)}(x,t,k)},&& x>0.
\end{align}
\end{subequations}

Now we can calculate the long-time asymptotics of the solution in the plane wave region. 
\begin{theorem}\label{fsth1pw}
	(Plane wave region).\\
	Suppose that the initial data $q_0(x)$ is a compact perturbation of the background (\ref{fsivp-c}) and that the associated spectral functions $a_j(k)$ and $r_j(k)=\frac{b_j(k)}{a_j(k)}$, $j=1,2$ satisfy the following conditions:
	\begin{enumerate}
		\item $a_1(k)$ and $a_2(k)$ have no zeros in $\overline{\mathbb{C}^{+}}$ and $\overline{\mathbb{C}^{-}}$ respectively,
		\item $\int_{-\infty}^{k}d\arg(1+ r_1(\zeta)r_2(\zeta))\in(-\pi,\pi)$, for all $k<-\frac{A}{\sqrt{2}}$.
	\end{enumerate}
	Then the solution $q(x,t)$ 
	of problem (\ref{fsivp})
	has the following long-time asymptotics along the rays $\frac{x}{4t}=const$ uniformly in any compact subset of $\{\frac{x}{4t}\in\mathbb{R}:|\frac{x}{4t}|>\sqrt{2}A\}$:
	\begin{subequations}
	\begin{align}
	\label{fssolpw1}
	&q(x,t)=Ae^{-2\Im F_{\infty}(k_1)}
	e^{2i(A^2t+\Re F_{\infty}(k_1))}
	+E_1(x,t),&& t\to\infty,\quad \frac{x}{4t}>\sqrt{2}A,\\
	\label{fssolpw2}
	&q(-x,t)=Ae^{2\Im F_{\infty}(k_1)}
	e^{2i(A^2t+\Re F_{\infty}(k_1))}
	+E_2(x,t),&& t\to\infty,\quad -\frac{x}{4t}<-\sqrt{2}A,
	\end{align}
	\end{subequations}
	where $k_{1}=\frac{1}{2}\left(-\xi-\sqrt{\xi^2-2A^2}\right)$
	with $\xi=\frac{x}{4t}>0$, $\Re F_{\infty}(k_1)$ and $\Im F_{\infty}(k_1)$ are given by (\ref{fsreimFinf}), and the decaying terms $E_j(x,t)$, $j=1,2$ have the following form depending on the value of $\Im\nu(k_1)$:
	\begin{description}
	\item[(a)] if $\Im\nu(k_1)\in\left(-\frac{1}{2},-\frac{1}{6}\right]$, then
	\begin{align*}
	&E_1(x,t)=t^{-\frac{1}{2}-\Im\nu(k_1)}c_1(k_1)
	\exp\left\{2it(A^2+\theta(k_1,\xi))+i\Re\nu(k_1)\ln t\right\}
	+R(x,t),\\
	&E_2(x,t)=t^{-\frac{1}{2}-\Im\nu(k_1)}c_3(k_1)
	\exp\left\{2it(A^2-\theta(k_1,\xi))-i\Re\nu(k_1)\ln t\right\}
	+R(x,t).
	\end{align*}
	\item[(b)] if $\Im\nu(k_1)\in\left(-\frac{1}{6},\frac{1}{6}\right)$, then
	\begin{align}
	\label{fsE1b}
	\nonumber
	E_1(x,t)&=t^{-\frac{1}{2}+\Im\nu(k_1)}c_2(k_1)
	\exp\left\{2it(A^2-\theta(k_1,\xi))-i\Re\nu(k_1)\ln t\right\}&\\
	\nonumber
	&\quad+ 
	t^{-\frac{1}{2}-\Im\nu(k_1)}c_1(k_1)
	\exp\left\{2it(A^2+\theta(k_1,\xi))+i\Re\nu(k_1)\ln t\right\}
	+R(x,t),
	\end{align}
	\begin{align}
	\nonumber
	E_2(x,t)&=t^{-\frac{1}{2}+\Im\nu(k_1)}c_4(k_1)
	\exp\left\{2it(A^2+\theta(k_1,\xi))+i\Re\nu(k_1)\ln t\right\}&\\
	\nonumber
	&\quad+ 
	t^{-\frac{1}{2}-\Im\nu(k_1)}c_3(k_1)
	\exp\left\{2it(A^2-\theta(k_1,\xi))-i\Re\nu(k_1)\ln t\right\}
	+R(x,t).
	\end{align}
	\item[(c)] if $\Im\nu(k_1)\in\left[\frac{1}{6},\frac{1}{2}\right)$, then
	\begin{align*}
	&E_1(x,t)=t^{-\frac{1}{2}+\Im\nu(k_1)}c_2(k_1)
	\exp\left\{2it(A^2-\theta(k_1,\xi))-i\Re\nu(k_1)\ln t\right\}+R(x,t),\\
	&E_2(x,t)=t^{-\frac{1}{2}+\Im\nu(k_1)}c_4(k_1)
	\exp\left\{2it(A^2+\theta(k_1,\xi))+i\Re\nu(k_1)\ln t\right\}+R(x,t).
	\end{align*}
	\end{description}
	Here
	\begin{align}
	\nonumber
	c_1(k_1)=&\frac{\sqrt{\pi}\left(w(k_1)-1/w(k_1)\right)^2F^2(k_1,k_1)}
	{\sqrt{2}r_2(k_1)\Gamma(i\nu(k_1))}\\
	\nonumber
	&\times\exp\left\{
	2iF_{\infty}(k_1)-\frac{\pi}{2}\nu(k_1)+\frac{3\pi i}{4}-2\chi(k_1,k_1)-
	(1-2i\nu(k_1))\ln 2\sqrt{\frac{4k_1+2\xi}{f(k_1)}}
	\right\},
	\end{align}
	\begin{align}
	\nonumber
	c_2(k_1)=&\frac{\sqrt{\pi}\left(w(k_1)+1/w(k_1)\right)^2}
	{\sqrt{2}r_1(k_1)\Gamma(-i\nu(k_1))F^2(k_1,k_1)}\\
	\nonumber
	&\times\exp\left\{
	2iF_{\infty}(k_1)-\frac{\pi}{2}\nu(k_1)+\frac{\pi i}{4}+2\chi(k_1,k_1)-
	(1+2i\nu(k_1))\ln 2\sqrt{\frac{4k_1+2\xi}{f(k_1)}}
	\right\},
	\end{align}
	\begin{align}
	\nonumber
	c_3(k_1)=&\frac{\sqrt{\pi}\left(w(k_1)+1/w(k_1)\right)^2
		\overline{F^2(k_1,k_1)}}
	{\sqrt{2}\,\overline{r_2}(k_1)\Gamma(-i\overline{\nu(k_1)})}\\
	\nonumber
	&\times\exp\left\{
	2i\overline{F_{\infty}(k_1)}-\frac{\pi}{2}\overline{\nu(k_1)}+\frac{\pi i}{4}-2\overline{\chi}(k_1,k_1)-
	(1+2i\overline{\nu(k_1)})\ln 2\sqrt{\frac{4k_1+2\xi}{f(k_1)}}
	\right\},
	\end{align}
	\begin{align}
	\nonumber
	c_4(k_1)=&\frac{\sqrt{\pi}\left(w(k_1)-1/w(k_1)\right)^2}
	{\sqrt{2}\,\overline{r_1}(k_1)\Gamma(i\overline{\nu(k_1)})
		\overline{F^2(k_1,k_1)}}\\
	\nonumber
	&\times\exp\left\{
	2i\overline{F_{\infty}(k_1)}-\frac{\pi}{2}\overline{\nu(k_1)}+\frac{3\pi i}{4}+2\overline{\chi}(k_1,k_1)-
	(1-2i\overline{\nu(k_1)})\ln 2\sqrt{\frac{4k_1+2\xi}{f(k_1)}}
	\right\},
	\end{align}
	where $\Gamma(\cdot)$ is the Euler Gamma-function and $w(k_1)$, $f(k_1)$, $\chi(k_1, k_1)$, $\nu(k_1)$, and $F(k_1,k_1)$ are given by (\ref{fsK}), (\ref{fsf}), (\ref{fschi}), (\ref{fsnu}) and (\ref{fsF}), respectively. The error term $R(x,t)$ has the form
	\begin{equation}
	\label{fsR3}
	R(x,t)=
	\begin{cases}
	O\left(t^{-1+2|\Im\nu(k_1)|}\right),&\Im\nu(k_1)\not=0,\\
	O\left(t^{-1}\ln t\right),&\Im\nu(k_1)=0.
	\end{cases}
	\end{equation}
\end{theorem}
\begin{proof}
See Appendix A.
\end{proof}

\subsection{The modulated elliptic wave regions: $0<|\frac{x}{4t}|<\sqrt{2}A$}\label{fsmewr}
The present Subsection is devoted to obtaining the large-time asymptotics of the solution in the modulated elliptic wave regions,
which are defined by $0<|\frac{x}{4t}|<\sqrt{2}A$.

The main difference of this case comparing with the plane wave regions is the absence of real stationary phase points (see (\ref{fsstphc})). Similarly to
\cite{BV07} (see also \cite{BM17, BKS11}), we set a change-of-factorization point $k=k_0$, $k_0<0$ (see Figure \ref{cont_1_ew}), and then deal with exponentially growing (as $t\to\infty$) jump matrices. This artificial ``stationary phase point'' will be found as the solution of a linear integral equation (see (\ref{fsintk0}) below).

The transformations from  $M(x,t,k)$ to $M^{(2)}(x,t,k)$ are as in the plane wave case, with the only difference: instead of $\delta(k,k_1)$, in the modulated elliptic wave region we take $\delta(k,k_0)$. 
As in the plane wave case, we make an assumption on the winding of the argument of $(1+r_1(k)r_2(k))$:

\underline{\textbf{Assumption}} 
\textit{(Winding of the argument, elliptic wave region).}
We impose the following restriction on the winding of the argument of $(1+r_1(k)r_2(k))$ for $-\frac{A}{\sqrt{2}}<k<0$ in the modulated elliptic wave region (cf. (\ref{fsarg-ass})):
\begin{equation}\label{fsDelta-argew}
\int_{-\infty}^{k}d\arg(1+ r_1(\zeta)r_2(\zeta))
\in(-\pi,\pi),\quad \text{for all}\ \ k\in\left(-\frac{A}{\sqrt{2}},0\right),
\end{equation}

Then $\delta^{\sigma_3}(k,k_0)$ has square integrable singularity at $k=k_0$, and the Riemann-Hilbert problem for $M^{(2)}(x,t,k)$ has the form (see Figure \ref{cont_1_ew}):
\begin{figure}[h]
	\begin{minipage}[h]{0.99\linewidth}
		\centering{\includegraphics[width=0.79\linewidth]{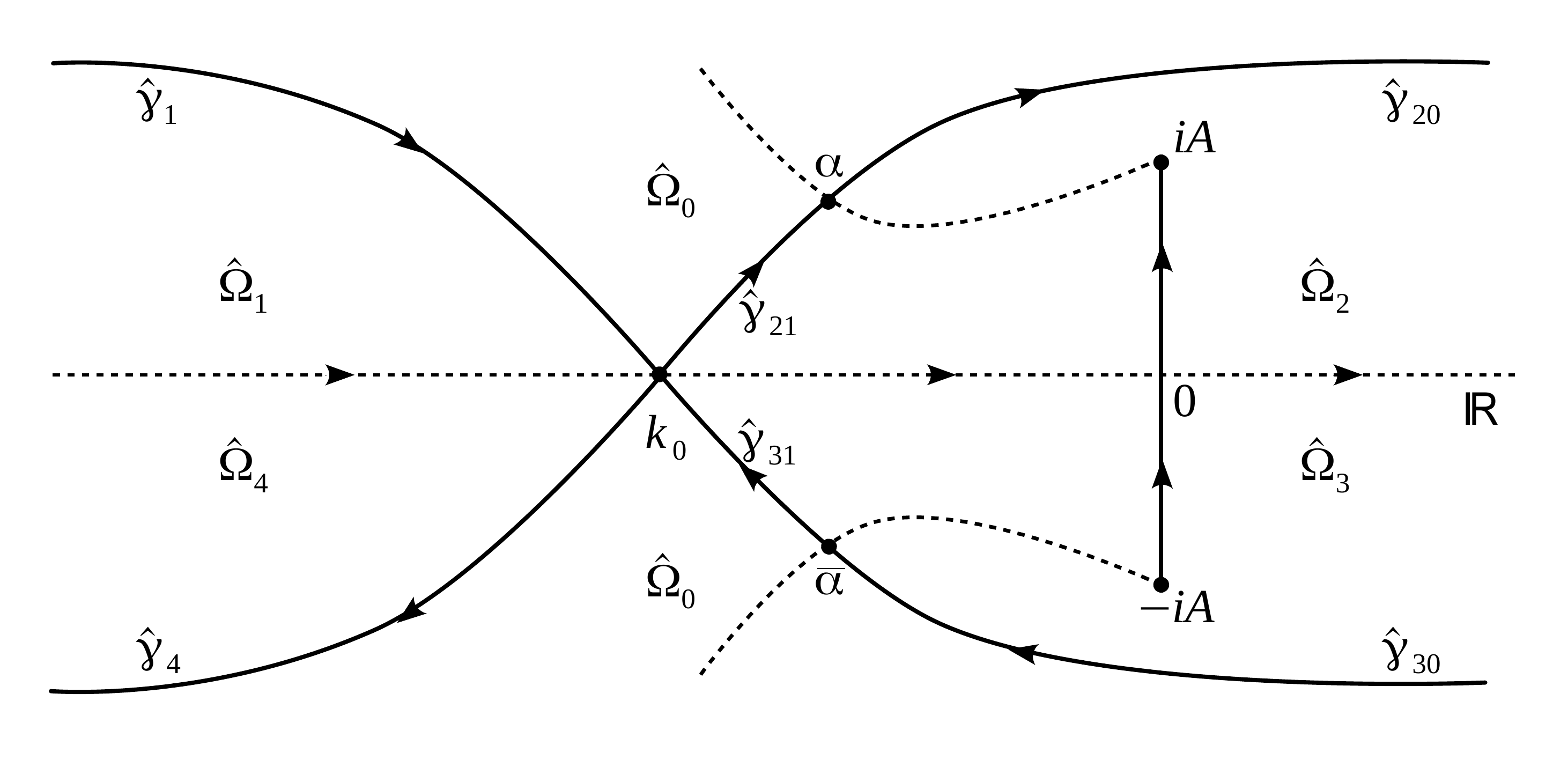}}
		\caption{Contour $\hat\Gamma=\hat\gamma_1\cup\dots\cup\hat\gamma_4$, where $\hat\gamma_j=\hat\gamma_{j0}\cup\hat\gamma_{j1}$, $j=2,3$, and domains $\Omega_j$, $j=0,\dots,4$ for the Riemann-Hilbert problem $M^{(2)}$ in the elliptic wave region.}
		\label{cont_1_ew}
	\end{minipage}
\end{figure}
\begin{subequations}
\begin{align}
&M^{(2)}_+(x,t,k)=M^{(2)}_-(x,t,k)J^{(2)}(x,t,k),&&k\in\hat\Gamma\cup B,\\
&M^{(2)}(x,t,k)=I+O(k^{-1}), &&k\to\infty,\\
&M^{(2)}(x,t,k)=O\left((k\pm iA)^{-\frac{1}{4}}\right),&& k\to\mp iA,\\
&M^{(2)}(x,t,k)=O\left((k-k_0)^{-\frac{1}{2}+\varepsilon(k_0)}\right),&& k\to k_0,\,\varepsilon(k_0)>0,
\end{align}
\end{subequations}
where the point $\alpha$ will be specified below (see (\ref{fsreaima})) and
\begin{equation}
\label{fsJ2ew}
J^{(2)}=
\begin{cases}
\begin{pmatrix}
1& \frac{r_2(k)\delta^{2}(k,k_0)}{1+r_1(k)r_2(k)}e^{-2it\theta}\\
0& 1\\
\end{pmatrix}
,\, k\in\hat\gamma_1;\quad
\begin{pmatrix}
1& 0\\
r_1(k)\delta^{-2}(k,k_0)e^{2it\theta}& 1\\
\end{pmatrix}
,\ k\in\hat\gamma_2;
\\
\begin{pmatrix}
1& -r_2(k)\delta^2(k,k_0)e^{-2it\theta}\\
0& 1\\
\end{pmatrix}
,\, k\in\hat\gamma_3;\quad
\begin{pmatrix}
1& 0\\
\frac{-r_1(k)\delta^{-2}(k,k_0)}{1+r_1(k)r_2(k)}e^{2it\theta}& 1\\
\end{pmatrix}
,\, k\in\hat\gamma_4;\\
\begin{pmatrix}
0& i\delta^{2}(k,k_0)\\
i\delta^{-2}(k,k_0)&0
\end{pmatrix},\,k\in B.
\end{cases}
\end{equation}

Notice that the jump matrix $J^{(2)}(x,t,k)$ is exponentially growing for $k\in\hat\gamma_{21}\cup\hat\gamma_{31}$.
To resolve the issue of this  growth, we use the factorizations on this part of the contour, which are similar to 
(4.2), (4.3) in \cite{BV07} (see also (5.3) in \cite{BM17}).
Namely, we define $M^{(3)}(x,t,k)$ as follows (see Figure \ref{cont_2_ew}):
$$
M^{(3)}=
\begin{cases}
M^{(2)}
\begin{pmatrix}
1& -\frac{\delta^{2}(k,k_0)}{r_1(k)}e^{-2it\theta}\\
0&1
\end{pmatrix},\, k\in\hat\Omega_{01};\quad
M^{(2)}
\begin{pmatrix}
1& \frac{\delta^{2}(k,k_0)}{r_1(k)}e^{-2it\theta}\\
0&1
\end{pmatrix},\, k\in\hat\Omega_{21};\\
M^{(2)}
\begin{pmatrix}
1& 0\\
\frac{e^{2it\theta}}{r_2(k)\delta^{2}(k,k_0)}&1
\end{pmatrix},\, k\in\hat\Omega_{02};\quad
M^{(2)}
\begin{pmatrix}
1& 0\\
-\frac{e^{2it\theta}}{r_2(k)\delta^{2}(k,k_0)}&1
\end{pmatrix},\, k\in\hat\Omega_{31};\\
M^{(2)},\text{ otherwise}.
\end{cases}
$$
\begin{figure}[h]
	\begin{minipage}[h]{0.99\linewidth}
		\centering{\includegraphics[width=0.79\linewidth]{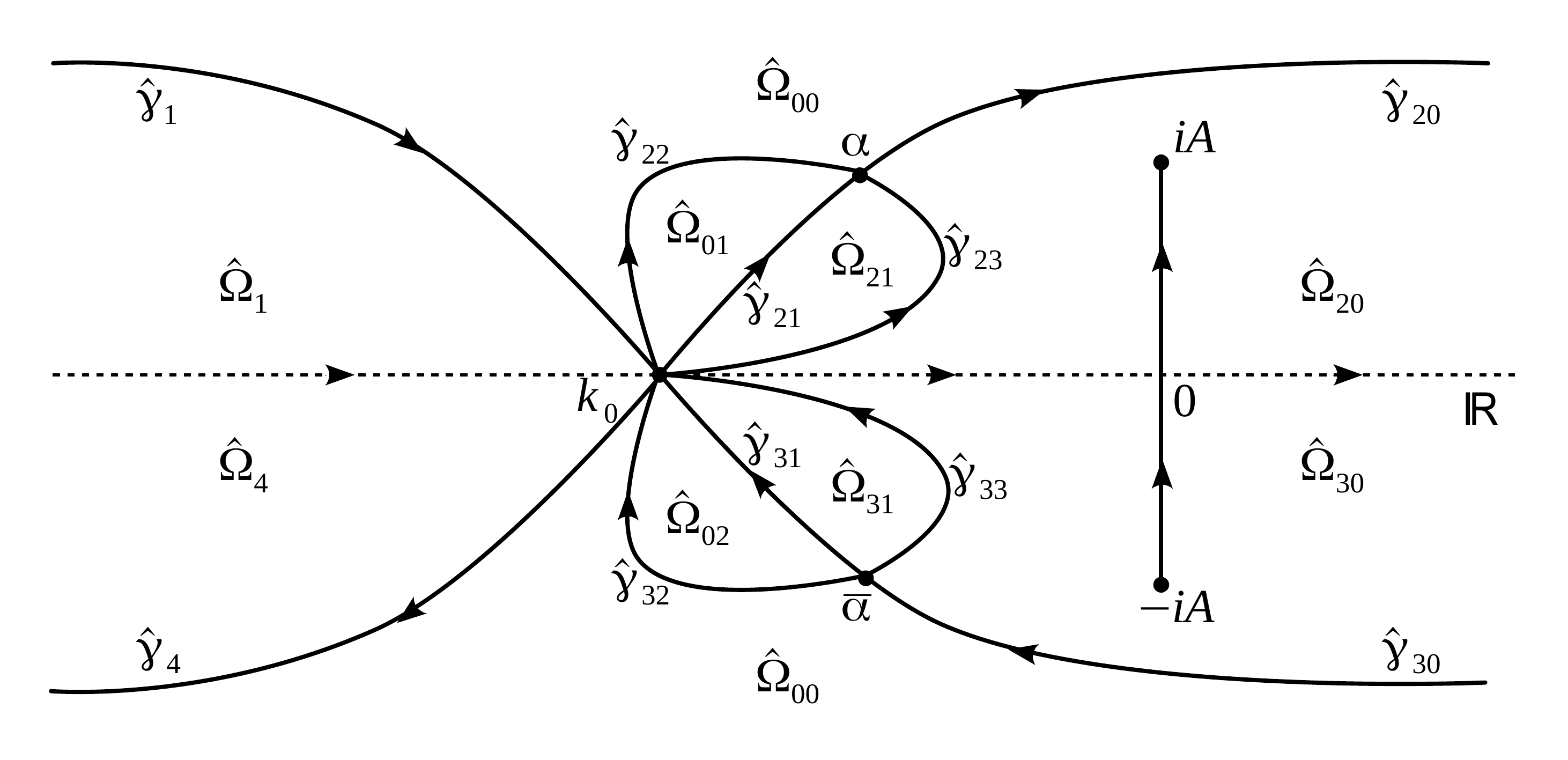}}
		\caption{Contour $\tilde{\Gamma}=\hat\Gamma\cup
			\left(
			\bigcup\limits_{i,j=2}^{3}
			\hat\gamma_{ij}\right)
			$ 
			and domains in the Riemann-Hilbert problems for $M^{(j)}(x,t,k)$, $j=3,4,5$ in the elliptic wave region.}
		\label{cont_2_ew}
	\end{minipage}
\end{figure}
It is straightforward to show that jump matrix $J^{(3)}(x,t,k)$ associated with sectionally holomorphic matrix $M^{(3)}(x,t,k)$ grows as $t\to\infty$
for $k\in\hat\gamma_{21}\cup\hat\gamma_{31}$ only:

\begin{equation}
\label{fsJ3ew}
J^{(3)}=
\begin{cases}
\begin{pmatrix}
0& -\frac{\delta^{2}(k,k_0)}{r_1(k)}e^{-2it\theta}\\
\frac{r_1(k)}{\delta^{2}(k,k_0)}e^{2it\theta}&0
\end{pmatrix},\,k\in\hat\gamma_{21};\quad
\begin{pmatrix}
1& \frac{\delta^{2}(k,k_0)}{r_1(k)}e^{-2it\theta}\\
0&1
\end{pmatrix},\, k\in\hat\gamma_{22}\cup\hat\gamma_{23};
\\
\begin{pmatrix}
0& -r_2(k)\delta^{2}(k,k_0)e^{-2it\theta}\\
\frac{e^{2it\theta}}{r_2(k)\delta^{2}(k,k_0)}&0
\end{pmatrix},\,k\in\hat\gamma_{31};\quad
\begin{pmatrix}
1& 0\\
-\frac{e^{2it\theta}}{r_2(k)\delta^{2}(k,k_0)}&1
\end{pmatrix},\, k\in\hat\gamma_{32}\cup\hat\gamma_{33};
\\
J^{(2)}(x,t,k),\text{ otherwise}.
\end{cases}
\end{equation}
To circumvent the problem of exponentially growing anti-diagonal jump matrices across $\hat\gamma_{21}\cup\hat\gamma_{31}$, we employ the $g$-function mechanism, firstly introduced by Deift, Venakides and Zhou \cite{DVZ94, DVZ97}. 
Namely, introduce a new phase function $h(k,\xi,k_0,\alpha)$ as a sectionally analytic function, which has a jump along $k\in
B\cup\hat\gamma_{21}\cup\hat\gamma_{31}$ and is bounded at the endpoints.
In order to deal with bounded at infinity sectionally analytic functions, $h(k)=h(k,\xi,k_0,\alpha)$ should have a  behavior for large $k$ that is
 similar to $\theta(k,\xi)$. Since (see (\ref{fstheta}))
\begin{equation}\label{fsthetakinf}
\theta(k,\xi)=2k^2+4\xi k+A^2+O(k^{-1}), \quad k\to\infty,
\end{equation}
we assume that
\begin{equation}\label{fshkinf}
h(k)=2k^2+4\xi k+H_\infty+O(k^{-1}), \quad k\to\infty,
\end{equation}
where $H_\infty$ will be found below (see (\ref{fsH0})).

Supposing that an appropriate $h(k)$ is known, we define $M^{(4)}(x,t,k)$ as follows:
\begin{equation}
M^{(4)}(x,t,k)=e^{-it(H_\infty-A^2)\sigma_3}M^{(3)}(x,t,k)
e^{-it(\theta(k,\xi)-h(k,\xi,k_0,\alpha))\sigma_3},\quad k\in\mathbb{C}\setminus
\left\{\tilde\Gamma\cup \overline{B}\right\}.
\end{equation}
Then $M^{(4)}(x,t,k)$ satisfies the following RH problem
\begin{subequations}
\begin{align}
&M^{(4)}_+(x,t,k)=M^{(4)}_-(x,t,k)J^{(4)}(x,t,k),&&k\in\tilde\Gamma\cup B,\\
&M^{(4)}(x,t,k)=I+O(k^{-1}), &&k\to\infty,\\
&M^{(4)}(x,t,k)=O\left((k\pm iA)^{-\frac{1}{4}}\right),&& k\to\mp iA,\\
&M^{(4)}(x,t,k)=O\left((k-k_0)^{-\frac{1}{2}+\varepsilon(k_0)}\right),&& k\to k_0,\,\varepsilon(k_0)>0,
\end{align}
\end{subequations}
where
\begin{equation}
\label{fsJ4ew}
J^{(4)}=
\begin{cases}
\begin{pmatrix}
0& -\frac{\delta^{2}(k,k_0)}{r_1(k)}e^{-it(h_++h_-)}\\
\frac{r_1(k)}{\delta^{2}(k,k_0)}e^{it(h_++h_-)}&0
\end{pmatrix},\,k\in\hat\gamma_{21};\quad
\begin{pmatrix}
1& \frac{\delta^{2}(k,k_0)}{r_1(k)}e^{-2ith}\\
0&1
\end{pmatrix},\, k\in\hat\gamma_{22}\cup\hat\gamma_{23};
\\
\begin{pmatrix}
0& -r_2(k)\delta^{2}(k,k_0)e^{-it(h_++h_-)}\\
\frac{e^{it(h_++h_-)}}{r_2(k)\delta^{2}(k,k_0)}&0
\end{pmatrix},\,k\in\hat\gamma_{31};\,
\begin{pmatrix}
1& 0\\
-\frac{e^{2ith}}{r_2(k)\delta^{2}(k,k_0)}&1
\end{pmatrix},\, k\in\hat\gamma_{32}\cup\hat\gamma_{33};\\
\begin{pmatrix}
1& \frac{r_2(k)\delta^{2}(k,k_0)}{1+r_1(k)r_2(k)}e^{-2ith}\\
0& 1\\
\end{pmatrix}
,\, k\in\hat\gamma_1;\quad
\begin{pmatrix}
1& 0\\
r_1(k)\delta^{-2}(k,k_0)e^{2ith}& 1\\
\end{pmatrix}
,\ k\in\hat\gamma_{20};
\\
\begin{pmatrix}
1& -r_2(k)\delta^2(k,k_0)e^{-2ith}\\
0& 1\\
\end{pmatrix}
,\, k\in\hat\gamma_{30};\quad
\begin{pmatrix}
1& 0\\
\frac{-r_1(k)\delta^{-2}(k,k_0)}{1+r_1(k)r_2(k)}e^{2ith}& 1\\
\end{pmatrix}
,\, k\in\hat\gamma_4;\\
\begin{pmatrix}
0& i\delta^{2}(k,k_0)e^{-it(h_++h_-)}\\
i\delta^{-2}(k,k_0)e^{it(h_++h_-)}&0
\end{pmatrix},\,k\in B.
\end{cases}
\end{equation}
The solution $q(x,t)$ can be found via $M^{(4)}(x,t,k)$ as follows:
\begin{subequations}\label{fsasolM4}
	\begin{align}\label{fssolM4}
	&q(x,t)=2ie^{2itH_\infty}\lim_{k\to\infty}kM_{12}^{(4)}(x,t,k),&& x>0,\\
	\label{fssol1M4}
	&q(-x,t)=-2ie^{2itH_\infty}\lim_{k\to\infty}k\overline{M_{21}^{(4)}(x,t,k)},&& x>0,
	\end{align}
\end{subequations}
where the real constant $H_{\infty}$ is given by (\ref{fsH0}).

Now the task is to define 
$h(k)\equiv h(k,\xi,k_0,\alpha)$ in a manner that the jump matrix $J^{(4)}(x,t,k)$ decays (exponentially fast) to $I$ except neighborhoods of stationary phase points and, possibly, some parts of the contour, where the jump matrix remains bounded as $t\to\infty$.
Similarly to \cite{BKS11}, Section 4.3.1, we define the differential $dh(k)$ which   has three zeros at $k=\alpha$, $k=\overline{\alpha}$ and $k=k_0$:
\begin{equation}\label{fsdh}
dh(k)=4\frac{(k-k_0)(k-\alpha)(k-\overline{\alpha})}
{\sqrt{(k^2+A^2)(k-\alpha)(k-\overline{\alpha})}}dk.
\end{equation}
We consider $dh(k)$ as an Abelian differential of the second kind with poles at $\infty^{\pm}$ 
on the genus-1 Riemann surface $\Sigma$ of
\begin{equation}\label{fsgammaRS}
\gamma(k)=\sqrt{(k^2+A^2)(k-\alpha)(k-\overline{\alpha})},
\end{equation}
where the branch of the square root is specified by the asymptotics on the upper sheet $\Sigma_u$: $\gamma(k)\sim k^2$ as $k\to\infty^{+}$. 
On the cuts $B$ and $(\overline{\alpha},\alpha)$ we set $\gamma(k)=\gamma_-(k)\equiv-\gamma_+(k)$. 
The canonical basis $\{\mathfrak{a},\mathfrak{b}\}$ of cycles on this Riemann surface is defined as follows. 
The $\mathfrak{b}$-cycle is a closed counterclockwise oriented simple loop around the branch cut $B$, which lies on the lower sheet $\Sigma_l$. 
The $\mathfrak{a}$-cycle starts on the upper sheet from the point $\overline{\alpha}$, then approaches $-iA$ and returns to the	starting point on the lower sheet (see Figure \ref{riemann_surf}).
\begin{figure}[h]
	\begin{minipage}[h]{0.49\linewidth}
		\centering{\includegraphics[width=0.99\linewidth]{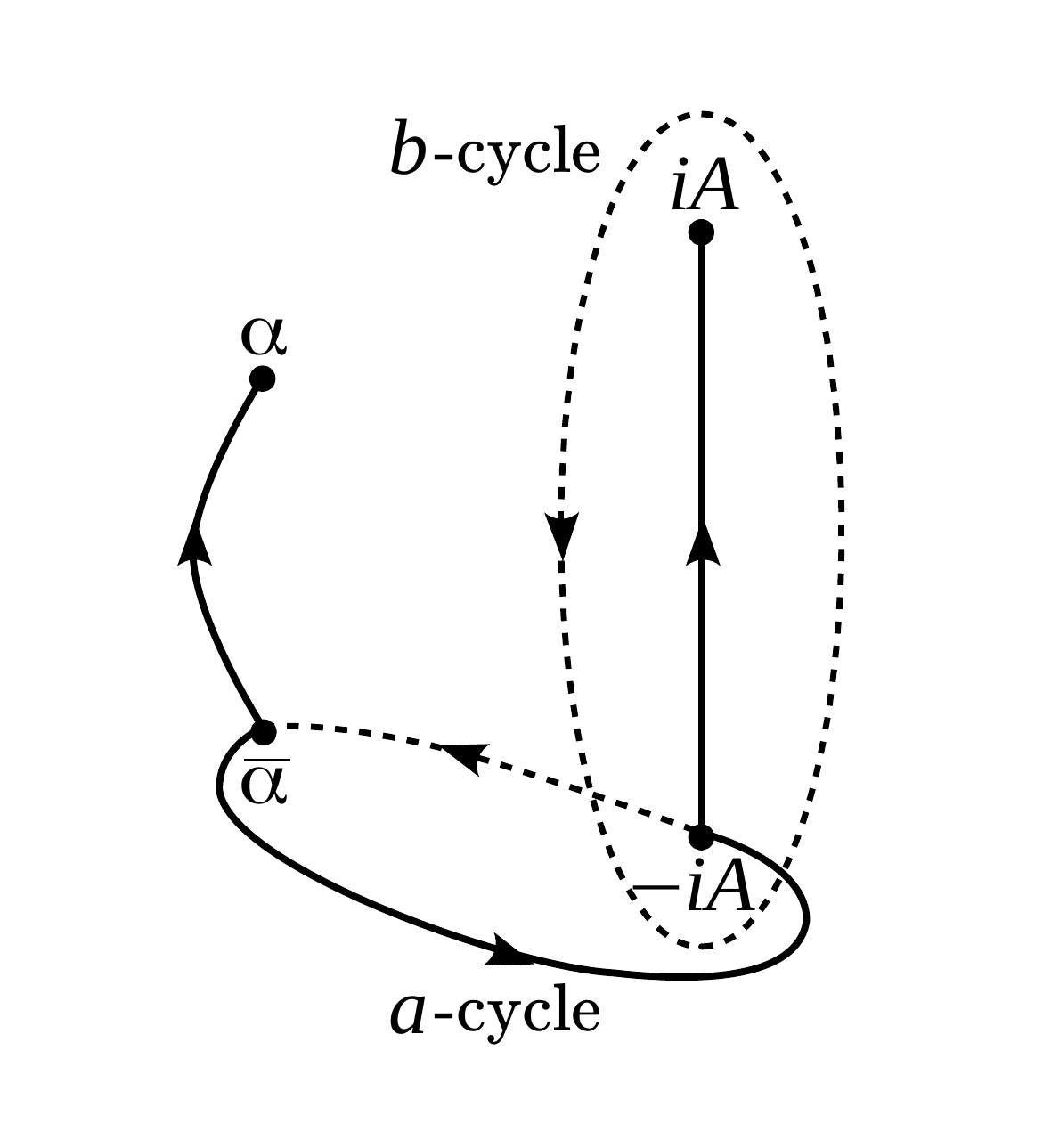}}
		\caption{Canonical basis $\{\mathfrak{a},\mathfrak{b}\}$ of cycles on the Riemann surface $\Sigma$ of $\gamma(k)$. The solid lines lie on the upper sheet $\Sigma_u$, while the dashed lines lie on the lower sheet $\Sigma_l$.}
		\label{riemann_surf}
	\end{minipage}
	\hfill
	\begin{minipage}[h]{0.49\linewidth}
		\centering{\includegraphics[width=0.99\linewidth]{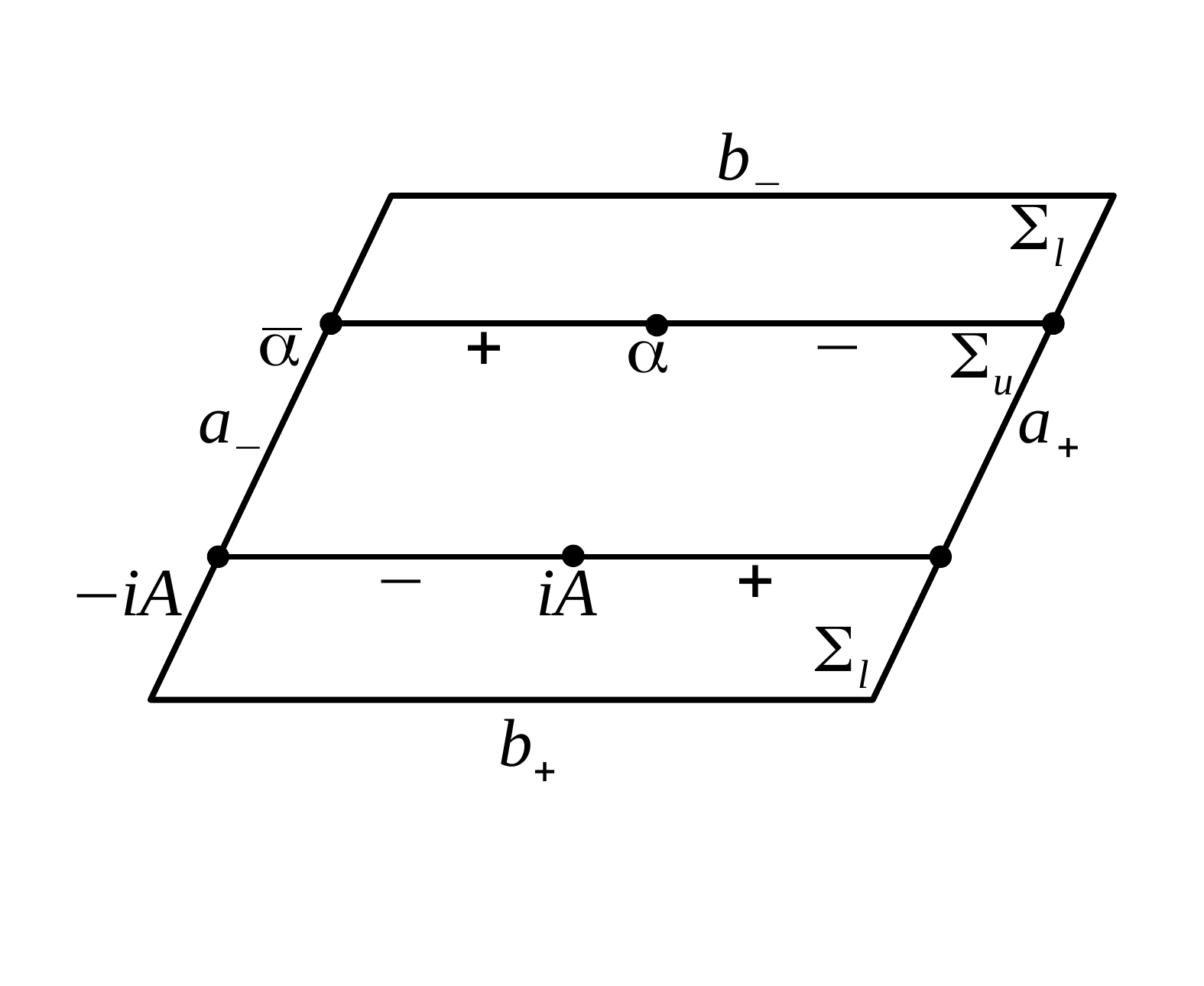}}
		\caption{The fundamental rectangle of the Riemann surface $\Sigma$, obtained by cutting it along the basis $\{\mathfrak{a},\mathfrak{b}\}$. $\Sigma_{u}$ and $\Sigma_{l}$ are the upper and the lower sheets of $\Sigma$ respectively; ``$+$'' and ``$-$'' denotes the left and the right sides of the corresponding branch cuts.}
		\label{fund_rect}
	\end{minipage}
\end{figure}
Define $h(k)$ as the sum of two Abelian integrals (cf. \cite{BKS11, BM17}):
\begin{equation}\label{fsh}
h(k)=\frac{1}{2}\left(\int_{iA}^{k}+\int_{-iA}^{k}\right)\,dh(z),
\end{equation}
where $dh(z)$ is given by (\ref{fsdh}).
In order to make the locally holomorphic multivalued function $h(k)$ single valued, we assume that all integrals along paths on the Riemann surface lie on the fundamental rectangle obtained by cutting this surface along the basis $\{\mathfrak{a},\mathfrak{b}\}$, see Figure \ref{fund_rect}.

The definition (\ref{fsh}) involves three real parameters: $k_0$, $\Re\alpha$ and $\Im\alpha$. We choose these parameters  in such a way that the jump matrix $J^{(4)}(x,t,k)$ become bounded for $k\in B$ and $k\in\hat\gamma_{21}\cup\hat\gamma_{31}$ and decaying for the other parts of the  contour:
\begin{enumerate}[(i)]
\item $h_+(k)+h_-(k)=0$, for $k\in B$,
\item $h_+(k)+h_-(k)\in\mathbb{R}$, for $k\in \hat\gamma_{21}\cup\hat\gamma_{31}$, is similar as that of $\Im\theta(k,\xi)$.
\end{enumerate}
These conditions are fulfilled for the parameters given in the following Proposition (cf. \cite{BM17, BKS11}).
\begin{proposition}\label{fsdefh}
	The phase function $h(k)$ defined by (\ref{fsh}), satisfies conditions (i)--(iii) if the parameters  $k_0$ and $\alpha$ have the following values:
	\begin{equation}\label{fsreaima}
	\Re\alpha=-k_0-\xi,\quad \Im\alpha=\sqrt{A^2+2k_0^2+2k_0\xi},
	\end{equation}
	and $k_0\in(-\xi/2,0)$ is the single solution of the integral equation:
	\begin{equation}\label{fsintk0}
	\int_B\sqrt{
	\frac{(k+k_0+\xi)^2+2k_0^2+2\xi k_0+A^2}
	{k^2+A^2}}
	(k-k_0)\,dk =0.
	\end{equation}
For such values of the parameters, $h(k)$ has the large $k$ asymptotics (\ref{fshkinf}), where
$H_\infty\in\mathbb{R}$ has the form
\begin{equation}\label{fsH0}
H_\infty=2\left(\int_{iA}^{\infty}+\int_{-iA}^{\infty}\right)\left[
\frac{(k-k_0)(k-\alpha)(k-\overline{\alpha})}{\gamma(k)}-
(k+\xi)\right]\,dk+2A^2.
\end{equation}
Moreover, for $k\in\hat\gamma_{21}\cup\hat\gamma_{31}$ the Abel integral $h(k)$ satisfies the following jump condition:
\begin{equation}\label{fshjumpa}
h_+(k)+h_-(k)=\Omega,
\quad k\in\hat\gamma_{21}\cup\hat\gamma_{31},
\end{equation}
where the real constant $\Omega$ is as follows:
\begin{equation}\label{fsOmega}
\Omega = \left(\int_{iA}^{\alpha}+\int_{-iA}^{\overline{\alpha}}\right)\,dh(k).
\end{equation}
\end{proposition}
\begin{proof}
See Appendix B.
\end{proof}
Having constructed the new phase function $h(k)$, which satisfies conditions (i)--(iii), we must eliminate the dependence of $k$ from the jump matrix $J^{(4)}(x,t,k)$ across $B$ and $\hat\gamma_{21}\cup\hat\gamma_{31}$.
To do this, we define the scalar function $G(k)=G(k,k_0,\alpha)$ as the solution of a scalar Riemann-Hilbert problem (cf. (\ref{fsFRHab})):
\begin{subequations}\label{fsGRHall}
	\begin{align}\label{fsGRH}
	&G_+(k,k_0,\alpha)G_-(k,k_0,\alpha)=\delta^{2}(k,k_0),&& k\in B,\\
	&G_+(k,k_0,\alpha)G_-(k,k_0,\alpha)=
	\frac{\delta^2(k,k_0)}{r_1(k)}e^{i\omega},
	&&k\in\hat\gamma_{21},\\
	&G_+(k,k_0,\alpha)G_-(k,k_0,\alpha)=
	\delta^2(k,k_0)r_2(k)e^{i\omega},
	&&k\in\hat\gamma_{31},\\
	\label{fsGbound}
	&G(k,k_0,\alpha)=O(1),&&k\to\pm iA,\,k\to\Re\alpha\pm i\Im\alpha,\,k\to\infty.
	\end{align}
\end{subequations}
Using the arguments similar to those in the plane wave case for the auxiliary function $F(k,k_1)$, we arrive at the following integral representation for $G(k,k_0,\alpha)$:
\begin{align}
\nonumber
G(k,k_0,\alpha)=
\exp\left\{
-\frac{\gamma(k)}{2\pi i}\left(
\int_{B}\frac{2\ln\delta(\zeta,k_0)}{\gamma(\zeta)(\zeta-k)}\,d\zeta
+\int_{\hat\gamma_{21}}
\frac{\ln\frac{\delta^2(\zeta,k_0)}{r_1(\zeta)}+i\omega}
{\gamma(\zeta)(\zeta-k)}\,d\zeta\right.\right.\\
\label{fsintG}
\left.
\left.
+\int_{\hat\gamma_{31}}
\frac{\ln\left[\delta^2(\zeta,k_0)r_2(\zeta)\right]+i\omega}
{\gamma(\zeta)(\zeta-k)}\,d\zeta
\right)
\right\},
\end{align}
where $\omega\in\mathbb{C}$ is chosen to ensure that $G(k)$ is bounded as $k\to\infty$:
\begin{equation}\label{fsomega}
\omega=i\frac
{\int_{B}\frac{2\ln\delta(\zeta,k_0)}{\gamma(\zeta)}\,d\zeta
+\int_{\hat\gamma_{21}}
\ln\left[\frac{\delta^2(\zeta,k_0)}{r_1(\zeta)}\right]
\frac{d\zeta}{\gamma(\zeta)}
+\int_{\hat\gamma_{31}}
\ln[\delta^2(\zeta,k_0)r_2(\zeta)]
\frac{d\zeta}{\gamma(\zeta)}}
{\int_{\hat\gamma_{21}\cup\hat\gamma_{31}}\frac{d\zeta}{\gamma(\zeta)}}.
\end{equation}
This implies that $G(k)$ has the following behavior for the large $k$:
\begin{equation}
G(k,k_0,\alpha)=e^{iG_{\infty}(k_0,\alpha)}+O(k^{-1}),\quad k\to\infty,
\end{equation}
where the \textit{complex} constant $G_{\infty}(k_0,\alpha)$ has the form
\begin{align}
\nonumber
G_{\infty}(k_0,\alpha)=-\frac{1}{2\pi}\left(
\int_{B}\frac{2\ln\delta(\zeta,k_0)}{\gamma(\zeta)}\zeta\,d\zeta
+\int_{\hat\gamma_{21}}
\frac{\ln\frac{\delta^2(\zeta,k_0)}{r_1(\zeta)}+i\omega}
{\gamma(\zeta)}\zeta\,d\zeta\right.\\
\label{fsGinfty}
\left.+\int_{\hat\gamma_{31}}
\frac{\ln[\delta^2(\zeta,k_0)r_2(\zeta)]+i\omega}
{\gamma(\zeta)}\zeta\,d\zeta
\right).
\end{align}
Finally, arguing as in Remark \ref{fsboundf} we conclude that $G(k,k_0,\alpha)$ given by (\ref{fsintG}) satisfies (\ref{fsGbound}).
\begin{remark}
	In the case of the classical (local) NLS equation, the reflection coefficients $r_1(k)$ and $r_2(k)$ have symmetries for $k\in\mathbb{R}$.
	Since the symmetry axis of the contour of the RH problem (\ref{fsGRHall}) for $G(k,k_0,\alpha)$ is the real line, in the case of the local NLS equation the constants $\omega$ and $G_{\infty}(k_0,\alpha)$ are purely real (cf. $\omega$ and $g_{\infty}$ given by (5.44) and (5.47) respectively in \cite{BM17}).
	In our case, however, the spectral data have symmetries with respect to the imaginary axis (see (\ref{fsabsym})), which leads to, in general, \textit{complex} valued constants $\omega$ and $G_{\infty}(k_0,\alpha)$.
	Therefore, the large-time asymptotics of $q(x,t)$ depends on the initial data and, as in the plane wave region, the effect of modulation instability is  non-universal in the modulated elliptic wave regions.
\end{remark}
\begin{remark}
	In view of lack of symmetry between $r_1(k)$ and $r_2(k)$,
	the function $G(k,k_0,\alpha)$ can have a singularity as $k$ approaches $k_0$, 
	which is in contrast with the case of the classical NLS equation.
	Indeed, considering the behavior of contour integrals at the endpoints with logarithmic singularity (see Chapter I, Paragraph 8.5 in \cite{G66}) and using (\ref{fsdelta-singular}) with $k_0$ instead of $k_1$ we have that
	\begin{subequations}
	\begin{align}
	\nonumber
	&\frac{1}{2\pi i}
	\int_{\hat\gamma_{21}}
	\frac{\ln\frac{\delta^2(\zeta,k_0)}{r_1(\zeta)}}
	{\gamma(\zeta)(\zeta-k)}\,d\zeta=
	\frac{\ln r_1(k_0)-2\chi(k_0,k_0)-2\pi\nu(k_0)}{2\pi i\gamma_-(k_0)}
	\ln(k-k_0)\\
	& \qquad\qquad\qquad\qquad\qquad
	-\frac{\nu(k_0)}{2\pi\gamma_-(k_0)}
	\ln^2(k-k_0)+I_{k_0,1},\\
	\nonumber
	&\frac{1}{2\pi i}
	\int_{\hat\gamma_{31}}
	\frac{\ln[\delta^2(\zeta,k_0)r_2(\zeta)]}
	{\gamma(\zeta)(\zeta-k)}\,d\zeta=
	\frac{\ln r_2(k_0)+2\chi(k_0,k_0)+2\pi\nu(k_0)}{2\pi i\gamma_-(k_0)}
	\ln(k-k_0)\\
	& \qquad\qquad\qquad\qquad\qquad\qquad
	+\frac{\nu(k_0)}{2\pi\gamma_-(k_0)}
	\ln^2(k-k_0)+I_{k_0,2},
	\end{align}
	\end{subequations}
	where $I_{k_0,j}$, $j=1,2$ are analytic in a neighborhood of $k=k_0$. Consequently we have that 
	\begin{equation}
	G(k,k_0,\alpha)(k-k_0)^{\frac{\gamma(k)}{2\pi i\gamma_-(k_0)}
	\ln r_1(k_0)r_2(k_0)}=O(1), \quad k\to k_0.
	\end{equation}
	Despite the  possible strong singularity of $G(k,k_0,\alpha)$ (and, consequently, $M^{(5)}(x,t,k)$, see (\ref{fsM5}) and (\ref{fsM5RHd}) below) at $k=k_0$, we are able to construct the parametrix and make correct asymptotic estimates under the assumption (similar to the plane wave case) about smallness of $\arg(1+r_1(k)r_2(k))$, see Theorem \ref{fsth2} and Appendix C below.
\end{remark}
Now we are in a  position to make the final transformation of the Riemann-Hilbert problem:
\begin{equation}\label{fsM5}
M^{(5)}(x,t,k)=e^{-iG_{\infty}(k_0,\alpha)\sigma_3}
M^{(4)}(x,t,k)G^{\sigma_3}(k,k_0,\alpha),\quad
k\in\mathbb{C}\setminus
\left\{\tilde\Gamma\cup \overline{B}\right\}.
\end{equation}
Then $M^{(5)}(x,t,k)$ solves the RH problem with a constant jump along the contours $B$, $\hat\gamma_{21}$ and $\hat\gamma_{31}$:
\begin{subequations}\label{fsM5RH}
	\begin{align}
	&M^{(5)}_+(x,t,k)=M^{(5)}_-(x,t,k)J^{(5)}(x,t,k),&&k\in\tilde\Gamma\cup B,\\
	&M^{(5)}(x,t,k)=I+O(k^{-1}), &&k\to\infty,\\
	&M^{(5)}(x,t,k)=O\left((k\pm iA)^{-\frac{1}{4}}\right),&& k\to\mp iA,\\
	\label{fsM5RHd}
	&M^{(5)}(x,t,k)
	\begin{pmatrix}
	(k-k_0)^{P(k)}& 0\\
	0& 
	(k-k_0)^{-P(k)}
	\end{pmatrix}
	=O\left((k-k_0)^{-\frac{1}{2}+\varepsilon(k_0)}\right),&& k\to k_0,\,\varepsilon(k_0)>0,
	\end{align}
\end{subequations}
where 
$P(k)=\frac{\gamma(k)\arg(r_1(k_0)r_2(k_0))}
{2\pi\gamma_-(k_0)}$ 
and (recall that $\Omega\in\mathbb{R}$ is given by (\ref{fsOmega}))
\begin{equation}
\label{fsJ5ew}
J^{(5)}(x,t,k)=
\begin{cases}
\begin{pmatrix}
0& -e^{-it\Omega-i\omega}\\
e^{it\Omega+i\omega}&0
\end{pmatrix},&k\in\hat\gamma_{21}\cup\hat\gamma_{31},\\
i\sigma_1,&k\in B,\\
G^{-\sigma_3}(k,k_0,\alpha)J^{(4)}(x,t,k)G^{\sigma_3}(k,k_0,\alpha),&\text{othewise}.
\end{cases}
\end{equation}
The solution $q(x,t)$ can be found via $M^{(5)}(x,t,k)$ as follows:
\begin{subequations}\label{fsasolM5}
	\begin{align}\label{fssolM5}
	&q(x,t)=2ie^{2itH_\infty+2iG_{\infty}(k_0,\alpha)}
	\lim_{k\to\infty}kM_{12}^{(5)}(x,t,k),&& x>0,\\
	\label{fssol1M5}
	&q(-x,t)=-2ie^{2itH_\infty+2i\overline{G_{\infty}(k_0,\alpha)}}
	\lim_{k\to\infty}k\overline{M_{21}^{(5)}(x,t,k)},&& x>0.
	\end{align}
\end{subequations}

\begin{theorem}(Elliptic wave region)
	
	\label{fsth2}
	Suppose that the initial data $q_0(x)$ is a compact perturbation of the background (\ref{fsivp-c}) and that the associated spectral functions $a_j(k)$ and $r_j(k)=\frac{b_j(k)}{a_j(k)}$, $j=1,2$ satisfy the following conditions:
	\begin{enumerate}
		\item $a_1(k)$ and $a_2(k)$ have no zeros in $\overline{\mathbb{C}^{+}}$ and $\overline{\mathbb{C}^{-}}$ respectively,
		\item $\int_{-\infty}^{k}d\arg(1+ r_1(\zeta)r_2(\zeta))\in(-\pi,\pi)$, for all $k\in\left(-\frac{A}{\sqrt{2}},0\right)$.
	\end{enumerate}
	
	Then the solution $q(x,t)$ 
	of problem (\ref{fsivp})
	has the following large-time asymptotics along the rays $\frac{x}{4t}=const$ uniformly in any compact subset of $\{\frac{x}{4t}\in\mathbb{R}:\sqrt{2}A<|\frac{x}{4t}|<0\}$:
	\begin{subequations}\label{fsasellw}
	\begin{align}
	\nonumber
	&q(x,t)=(A+\Im\alpha)
	e^{-2\Im G_{\infty}(k_0,\alpha)}
	\frac{\Theta(\frac{\Omega t}{2\pi}
	+\frac{\omega}{2\pi}-\frac{1}{4}-v_{\infty}+c)
	\Theta(v_{\infty}+c)}
	{\Theta(\frac{\Omega t}{2\pi}
	+\frac{\omega}{2\pi}-\frac{1}{4}+v_{\infty}+c)
	\Theta(-v_{\infty}+c)}\\
	&\qquad\qquad
	\times e^{2i(tH_{\infty}+\Re G_{\infty}(k_0,\alpha))}
	+E_3(x,t),\quad 0<\frac{x}{4t}<\sqrt{2}A,\\
	\nonumber
	&q(-x,t)=(A+\Im\alpha)
	e^{2\Im G_{\infty}(k_0,\alpha)}
	\frac{\Theta(\frac{\Omega t}{2\pi}
	+\frac{\overline{\omega}}{2\pi}-\frac{1}{4}
	+\overline{v_{\infty}}-\overline{c})
	\Theta(-\overline{v_{\infty}}-\overline{c})}
	{\Theta(\frac{\Omega t}{2\pi}
	+\frac{\overline{\omega}}{2\pi}-\frac{1}{4}
	-\overline{v_{\infty}}-\overline{c})
	\Theta(\overline{v_{\infty}}-\overline{c})}\\
	&\qquad\qquad
	\times e^{2i(tH_{\infty}+\Re G_{\infty}(k_0,\alpha))}
	+E_4(x,t),\quad 0>-\frac{x}{4t}>-\sqrt{2}A.
	\end{align}
	\end{subequations}
	Here the genus-1 theta function $\Theta$ is given by (\ref{fsg1thf}), the constants $\alpha$, $\Omega$, $v_{\infty}$ and $c$ which  do not depend on the initial data $q_0(x)$ are given by (\ref{fsreaima}), (\ref{fsOmega}), (\ref{fsvinfty}) and (\ref{fsc}) respectively. Moreover, depended on the
	associated to the initial value spectral data constants $\omega$ and $H_{\infty}$ and function $G_{\infty}(k_0,\alpha)$ are given by (\ref{fsomega}), (\ref{fsH0}) and (\ref{fsGinfty}) respectively.
	
	Depending on the value of $\Im\nu(k_0)$, the decaying terms $E_3(x,t)$ and $E_4(x,t)$ are given by 
	\begin{description}
		\item[(a)] if $\Im\nu(k_0)\in\left(-\frac{1}{2},-\frac{1}{6}\right]$, then
		\begin{align*}
		&E_3(x,t)=t^{-\frac{1}{2}-\Im\nu(k_0)}c_5(x,t)+R(x,t),\\
		&E_4(x,t)=t^{-\frac{1}{2}-\Im\nu(k_0)}c_7(x,t)+R(x,t),
		\end{align*}
		\item[(b)] if $\Im\nu(k_0)\in\left(-\frac{1}{6},\frac{1}{6}\right)$, then
		\begin{align*}
		&E_3(x,t)=t^{-\frac{1}{2}+\Im\nu(k_0)}c_6(x,t)+ 
		t^{-\frac{1}{2}-\Im\nu(k_0)}c_5(x,t)+R(x,t),\\
		&E_4(x,t)=t^{-\frac{1}{2}+\Im\nu(k_0)}c_8(x,t)+ 
		t^{-\frac{1}{2}-\Im\nu(k_0)}c_7(x,t)+R(x,t),
		\end{align*}
		\item[(c)] if $\Im\nu(k_0)\in\left[\frac{1}{6},\frac{1}{2}\right)$, then
		\begin{align*}
		&E_3(x,t)=t^{-\frac{1}{2}+\Im\nu(k_0)}c_6(x,t)+R(x,t),\\
		&E_4(x,t)=t^{-\frac{1}{2}+\Im\nu(k_0)}c_8(x,t)+R(x,t).
		\end{align*}
	\end{description}
	The oscillating terms $c_j(x,t)$, $j=\overline{5,8}$ have the following form
	\begin{align}
	&c_j(x,t)=2ie^{2itH_{\infty}+2iG_{\infty}(k_0,\alpha)}
	\tilde{c}_j(x,t),\quad j=5,6,\\
	&c_j(x,t)=-2ie^{2itH_{\infty}
	+2i\overline{G_{\infty}(k_0,\alpha)}}
	\overline{\tilde{c}_j(x,t)},\quad j=7,8,
	\end{align}
	where $\tilde{c}_j(x,t)$, $j=\overline{5,8}$ and the remainder $R(x,t)$ are given by (\ref{fstck0}) and (\ref{fsR3}) respectively.
\end{theorem}
\begin{proof}
	See Appendix C.
\end{proof}
\appendix
\begin{appendices}

\section{Appendix A}
Proof of Theorem \ref{fsth1pw}.
 
The adaptation of the nonlinear steepest decent method \cite{DZ, DIZ} needed in the plane wave region  is close to that used in \cite{BV07} (see also \cite{BM17}) for the local NLS equation. The main difference is related to 
the non-reality of $\nu(k_1)$, which significantly affects both the main terms and the estimations of the residue terms.
Hereinafter we emphasize the main steps of the proof, paying special attention to its peculiarities while referring the reader to relevant literature for details.

Recall that due to (\ref{fssolM3}) it is sufficient to consider $\xi>\sqrt{2}A$ only. Let $S_{k_1}^{\varepsilon}$ be a counterclockwise oriented circle of radius $\varepsilon>0$ centered at $k=k_1$. Define $M^{err}(x,t,k)$ as follows
\begin{equation}\label{fsMerr}
M^{err}(x,t,k)=
\begin{cases}
M^{(3)}(x,t,k)\left[M^{mod}(k)\right]^{-1}, &k\text{ outside }S_{k_1}^{\varepsilon},\\
M^{(3)}(x,t,k)\left[M^{par}(x,t,k)\right]^{-1}\left[M^{mod}(k)\right]^{-1},&
k\text{ inside }S_{k_1}^{\varepsilon},
\end{cases}
\end{equation}
where $0<\varepsilon\ll1$ is such that $S_{k_1}^{\varepsilon}$ does not intersect $B$. Matrix $M^{mod}(k)$ is the solution of the model Riemann-Hilbert problem with a constant jump matrix (RH problem for $M^{(3)}$ without ``vanishing'' jump across $\hat\Gamma$):
\begin{subequations}
	\begin{align}
	&M^{mod}_+(k)=M^{mod}_-(k)J^{mod},&& k\in B,\\
	&M^{mod}_+(k)=I+O(k^{-1}),&&k\to\infty,\\
	&M^{mod}_+(k)=O\left((k\pm iA)^{-\frac{1}{4}}\right),&& k\to\mp iA,
	\end{align}
\end{subequations}
with $J^{mod}=i\sigma_1$. This problem has the explicit
solution   (see (\ref{fsK}) and (\ref{fssymK})):
\begin{equation}
M^{mod}(k)=\mathcal{E}(k),\quad k\in\mathbb{C}\setminus\overline{B}.
\end{equation}

The matrix function
$M^{par}(x,t,k)$, which is an appropriate approximation (as $t\to\infty$) of the RH problem stated solely on the cross $\hat\Gamma$,  can be constructed explicitly in terms of the parabolic cylinder functions \cite{I1}.
Indeed, let's expand the phase function $\theta(k,\xi)$ in a neighborhood of the stationary phase point $k=k_1$ \cite{BV07}:
\begin{equation}\label{fsths}
\theta(k,\xi)=\theta(k_1,\xi)+\sum\limits_{n=2}^{\infty}\theta_n(k-k_1)^n,
\quad \theta_n\equiv\theta_n(k_1,\xi)=\frac{\theta^{(n)}(k_1,\xi)}{n!}.
\end{equation}
Then introduce the scaled spectral parameter $z$:
\begin{equation}
z=2\sqrt{t}(k-k_1)\left(
\sum\limits_{n=2}^{\infty}\theta_n(k-k_1)^{n-2}\right)^{1/2}\equiv
\sqrt{t}\sum\limits_{n=1}^{\infty}\alpha_n(k-k_1)^n,
\end{equation}
and write the phase function in terms of the new variables (cf. \cite{DIZ, RS}):
\begin{equation}\label{fsthz}
\theta(k(z),\xi)=\theta(k_1,\xi)+\frac{z^2}{4t}.
\end{equation}
Since $\theta_2=\frac{4k_1+2\xi}{f(k_1)}>0$ in the plane wave case, we have $\alpha_1=2\sqrt\theta_2\neq0$ and thus, by the inverse function theorem, we can write $k-k_1$ as a series in $z/\sqrt{t}$:
\begin{equation}\label{fskk1}
k-k_1=\sum\limits_{n=1}^{\infty}\beta_n\left(\frac{z}{\sqrt{t}}\right)^n,\quad
k\text{ inside }S_{k_1}^{\varepsilon},
\end{equation}
where $\beta_n\in\mathbb{C}$, $n\in\mathbb{N}$ can be found recursively in terms of $\alpha_n$, $n\in\mathbb{N}$ starting from
\begin{equation}
\beta_1=1/\alpha_1\equiv \frac{1}{2}\sqrt{\frac{f(k_1)}{4k_1+2\xi}}.
\end{equation}
Define $\tilde{r}_j(k)$, $j=1,2$ by
\begin{equation}\label{fsrjtilde}
\tilde{r}_1(k)=r_1(k)F^2(k,k_1),\quad
\tilde{r}_2(k)=r_2(k)F^{-2}(k,k_1).
\end{equation}
Now using the following approximations of the functions $\tilde{r}_j(k)$, $\delta(k,k_1)$ near $k=k_1$ involved in $J^{(3)}(x,t,k)$ on the cross $\hat\Gamma$:
\begin{equation}
\tilde{r}_j(k(z))\approx \tilde{r}_j(k_1),\quad j=1,2,\qquad
\delta(k(z),k_1)\approx \beta_1^{i\nu(k_1)}e^{\chi(k_1,k_1)}t^{-\frac{i\nu(k_1)}{2}}z^{i\nu(k_1)},
\end{equation}
define $M^{par}$ as follows (cf. (3.19) in \cite{RS}):
\begin{equation}\label{fsMpar}
M^{par}(x,t,k)=\Delta_0(k_1,t)
m^{\Gamma^0_{k_1}}(k_1,z(k))\Delta_0^{-1}(k_1,t),
\end{equation}
where $\Delta_0(k_1,t)=
\beta_1^{i\nu(k_1)\sigma_3}t^{-\frac{i\nu(k_1)}{2}\sigma_3}
e^{\chi(k_1,k_1)\sigma_3-it\theta(k_1,\xi)\sigma_3}
$
(cf. $\hat\delta^0$ in \cite{BV07}) and $m^{\Gamma^0_{k_1}}(k_1,z)$ satisfies the RH problem on the cross (see (3.17), (3.18) in \cite{RS}, where instead of $r_j(-\xi)$, $j=1,2$, and $\nu(-\xi)$ we have $\tilde{r}_j(k_1)$, $j=1,2$, and $\nu(k_1)$ respectively).
The latter, in turn, can be determined in terms of the explicitly solvable problem for $m_0(k_1,z)$ as follows:
\begin{equation}\label{fsm-g-0}
m^{\Gamma^0_{k_1}}(k_1,z) = m_0(k_1,z) D^{-1}_{j}(k_1,z),\quad z\in\Omega_j,\,\,j=0,\ldots,4,
\end{equation}
where domains $\Omega_j$ are shown in Figure \ref{mod2}
\begin{figure}[h]
	\begin{minipage}[h]{0.99\linewidth}
		\centering{\includegraphics[width=0.8\linewidth]{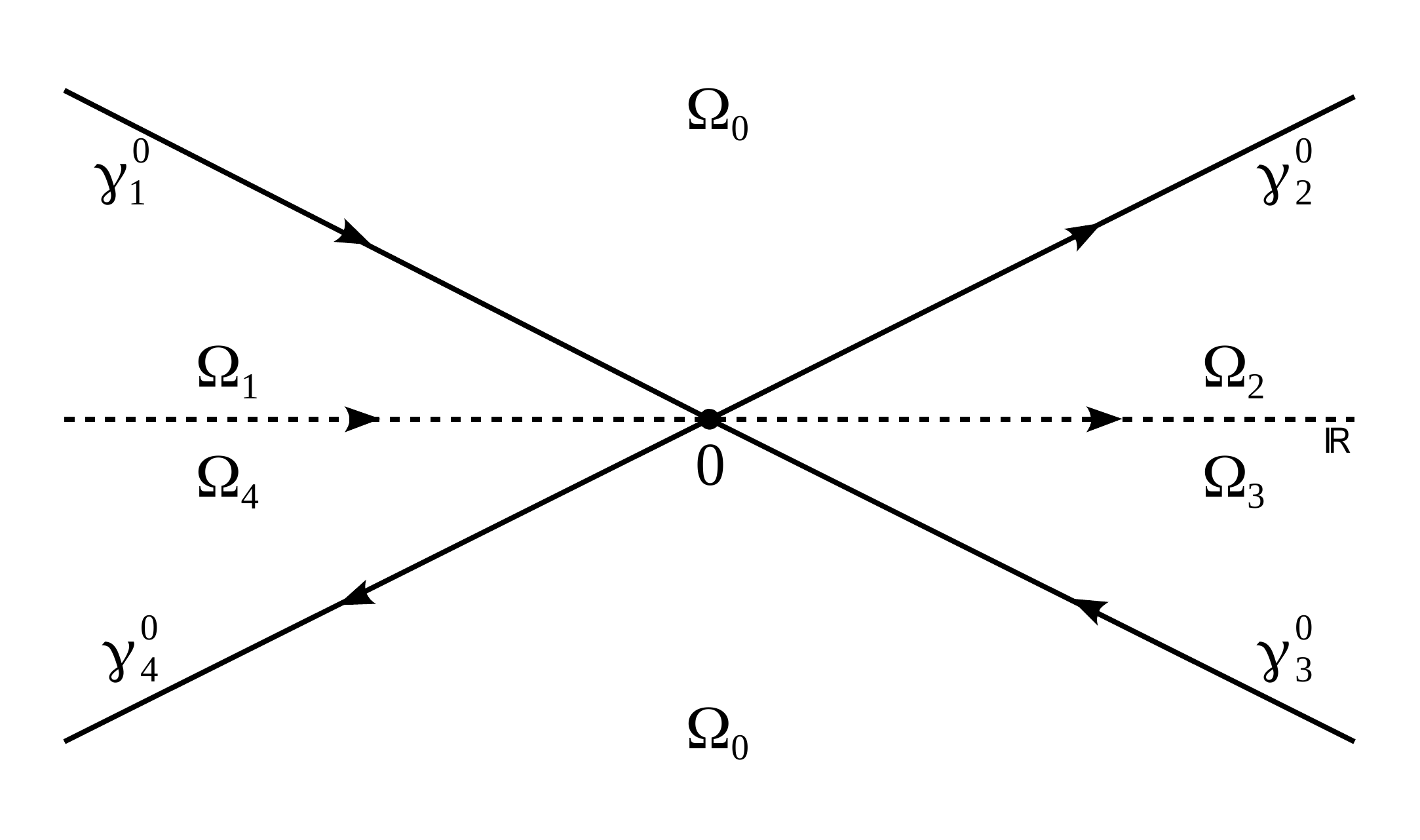}}
		\caption{}
		\textit{Domains $\Omega_j$, $j=\overline{0,4}$ in $z$-plane and contour $\Gamma^0_{k_1}=\gamma_1^0\cup...\cup\gamma_4^0$.}
		\label{mod2}
	\end{minipage}
\end{figure} 
and matrices $D_j$ have the form (cf. \cite{RS})
\begin{equation}\label{fsD0}
D_0(k_1,z)=e^{-i\frac{z^2}{4}\sigma_3}z^{i\nu(k_1)\sigma_3},
\end{equation}
and
\begin{equation}
\nonumber
\begin{matrix}
D_1(k_1,z)=D_0(k_1,z)
\begin{pmatrix}
1& \frac{\tilde r_2(k_1)}{1+\tilde r_1(k_1)\tilde r_2(k_1)}\\
0& 1\\
\end{pmatrix},
&&
D_2(k_1,z)=D_0(k_1,z)
\begin{pmatrix}
1& 0\\
\tilde r_1(k_1)& 1\\
\end{pmatrix},\\
D_3(k_1,z)=D_0(k_1,z)
\begin{pmatrix}
1& -\tilde r_2(k_1)\\
0& 1\\
\end{pmatrix},
&&
D_4(k_1,z)=D_0(k_1,z)
\begin{pmatrix}
1& 0\\
\frac{-\tilde r_1(k_1)}{1+\tilde r_1(k_1)\tilde r_2(k_1)}& 1\\
\end{pmatrix}.
\end{matrix}
\end{equation}
The RH problem for $m_0(k_1,z)$ with a constant jump matrix has the form:
\begin{subequations}\label{fsm0}
\begin{align}
&m_{0+}(k_1,z)=m_{0-}(k_1,z)j_0(k_1),&&z\in\mathbb{R},\\
&m_0(k_1,z)= \left(I+O(1/z)\right)
e^{-i\frac{z^2}{4}\sigma_3}z^{i\nu(k_1)\sigma_3},&&z\rightarrow\infty,
\end{align}
\end{subequations}
where
\begin{equation}\label{fsj0}
j_0(k_1)=
\begin{pmatrix}
1+\tilde{r}_1(k_1)\tilde{r}_2(k_1) & \tilde{r}_2(k_1)\\
\tilde{r}_1(k_1) & 1
\end{pmatrix},
\end{equation}
which can be solved explicitly in terms of the parabolic cylinder functions \cite{I1} (see also, e.g., Appendix A in \cite{BLS20, RS}).

For the derivation of the asymptotic formulas in Theorem \ref{fsth1pw},
 we give the large-$z$ behavior of $m^{\Gamma^0_{k_1}}(k_1,z)$ \cite{RS}:
\begin{equation}\label{fsmgamma}
m^{\Gamma^0_{k_1}}(k_1, z) = I + \frac{i}{z}\begin{pmatrix}
0 & \beta(k_1) \\ -\gamma(k_1) & 0
\end{pmatrix} + O(z^{-2}), \quad z\to\infty,
\end{equation}
where 
\begin{equation}
\label{fsbeta}
\beta(k_1)=\dfrac{\sqrt{2\pi}e^{-\frac{\pi}{2}\nu(k_1)}e^{-\frac{3\pi i}{4}}}{\tilde{r}_1(k_1)\Gamma(-i\nu(k_1))},
\end{equation}
\begin{equation}
\label{fsgamma}
\gamma(k_1)=\dfrac{\sqrt{2\pi}e^{-\frac{\pi}{2}\nu(k_1)}e^{-\frac{\pi i}{4}}}{\tilde{r}_2(k_1)\Gamma(i\nu(k_1))}.
\end{equation}

Having constructed all functions involved in (\ref{fsMerr}), let us formulate the RH problem for $M^{err}(x,t,k)$:
\begin{subequations}
	\begin{align}
	&M^{err}_+(x,t,k)=M^{err}_-(x,t,k)J^{err}(x,t,k),&& k\in \hat\Gamma_1,\\
	&M^{err}_+(x,t,k)=I+O(k^{-1}),&&k\to\infty,
	\end{align}
\end{subequations}
where $\hat\Gamma_1=\hat\Gamma\cup S_{k_1}^{\varepsilon}$ and
\begin{equation}
J^{err}=
\begin{cases}
M^{mod}(k)J^{(3)}(x,t,k)\left[M^{mod}(k)\right]^{-1},&k\in\hat\Gamma,\,
k\text{ outside } S_{k_1}^{\varepsilon},\\
M^{mod}(k)\left[M^{par}(x,t,k)\right]^{-1}\left[M^{mod}(k)\right]^{-1},& k\in S_{k_1}^{\varepsilon},\\
M^{mod}(k)M_-^{par}(x,t,k)J^{(3)}(x,t,k)\left[M_+^{par}(x,t,k)\right]^{-1}
\left[M^{mod}(k)\right]^{-1},&k\in\hat\Gamma,\,
k\text{ inside } S_{k_1}^{\varepsilon}.
\end{cases}
\end{equation}
The solution $q(x,t)$ can be obtained from the large-$k$ behavior of $M^{err}$ as follows:
\begin{subequations}\label{fsasolMerr}
	\begin{align}\label{fssolMerr}
	&q(x,t)=Ae^{2iA^2t+2iF_{\infty}(k_1)}+
	2ie^{2iA^2t+2iF_\infty(k_1)}\lim_{k\to\infty}kM_{12}^{err}(x,t,k),&& x>0,\\
	\label{fssol1Merr}
	&q(-x,t)=Ae^{2iA^2t+2i\overline{F_{\infty}(k_1)}}
	-2ie^{2iA^2t+2i\overline{F_\infty(k_1)}}\lim_{k\to\infty}k\overline{M_{21}^{err}(x,t,k)},&& x>0.
	\end{align}
\end{subequations}

In order to estimate $M^{err}(x,t,k)$ as $t\to\infty$, we can argue  as in Section III.C in \cite{RS}, which leads us to the following estimates for $(J^{err}(x,t,k)-I)$ as $t\to\infty$:
\begin{align}
\nonumber
\|J^{err}(x,t,\cdot)-I\|_{L^2(\hat\Gamma_1)}&=M^{mod}(k)
\begin{pmatrix}
O\left(t^{-\frac{1}{2}-\Im\nu(k_1)}\right)& 
O\left(t^{-\frac{1}{2}+\Im\nu(k_1)}\right)\\
O\left(t^{-\frac{1}{2}-\Im\nu(k_1)}\right)&
O\left(t^{-\frac{1}{2}+\Im\nu(k_1)}\right)
\end{pmatrix}
\left[M^{mod}(k)\right]^{-1}\\
&=O\left(t^{-\frac{1}{2}+|\Im\nu(k_1)|}\right),\quad t\to\infty,
\end{align}
and
\begin{align}
\nonumber
\|J^{err}(x,t,\cdot)-I\|_{L^1(\hat\Gamma)}&=M^{mod}(k)
\begin{pmatrix}
O\left(t^{-1-\Im\nu(k_1)}\right)& 
O\left(t^{-1+\Im\nu(k_1)}\right)\\
O\left(t^{-1-\Im\nu(k_1)}\right)&
O\left(t^{-1+\Im\nu(k_1)}\right)
\end{pmatrix}
\left[M^{mod}(k)\right]^{-1}\\
&=O\left(t^{-1+|\Im\nu(k_1)|}\right),\quad t\to\infty.
\end{align}

The estimates for $\|J^{err}(x,t,\cdot)-I\|_{L^{\infty}(\hat\Gamma_1)}$ and 
$\|\mu^{err}(x,t,k)-I\|_{L^2(\hat\Gamma_1)}$  are the same as in the case of the decaying problem for the NNLS equation \cite{RS} (here $\mu^{err}$ is the solution of the associated singular integral equation, cf. $\mu(\xi,t,k)$ in \cite{RS}). Therefore we eventually arrive at (see (3.30) and (3.34) in \cite{RS})
\begin{equation}\label{fsasMerr}
\lim\limits_{k\to\infty} k(M^{err}(x,t,k)-I)=-\frac{1}{2\pi i}
\int\limits_{|k-k_1|=\varepsilon}M^{mod}(k)\left(\left[M^{par}(x,t,k)\right]^{-1}
-I\right)\left[M^{mod}(k)\right]^{-1}\,dk+\hat{R}(x,t).
\end{equation}
where the error matrix $\hat{R}(x,t)$ has the structure 
\begin{equation}
	\hat{R}(x,t)=
	\begin{pmatrix}
	R(x,t) & R(x,t)\\
	R(x,t) & R(x,t)
	\end{pmatrix},
\end{equation}
with $R(x,t)$ given by (\ref{fsR3}). 
From (\ref{fsMpar}) and (\ref{fsmgamma}) we have
(cf. $\tilde{m}_0^{-1}(\xi,t,k)$ in \cite{RS})
\begin{equation}\label{fsMparas}
\left[M^{par}(x,t,k)\right]^{-1} = \Delta_0(k_1,t)
\left[m^{\Gamma^0_{k_1}}(x, \beta_1^{-1}\sqrt{t}(k-k_1))\right]^{-1}
\Delta_0^{-1}(k_1,t) = I +
\frac{\beta_1B(x,t)}{\sqrt{t}(k-k_1)} + \tilde{R}(x,t),
\end{equation}
where $\tilde{R}^{[j]}(x,t)=O\left(t^{-1+(-1)^j\Im \nu(k_1)}\right)$, $t\to\infty$, $j=1,2$ and (cf. (3.32) in \cite{RS})
\begin{equation}\label{fsB}
B(x,t)=\begin{pmatrix}
0 & -i\beta(k_1)e^{-2it\theta(k_1,\xi) + 2\chi(k_1,k_1)}\left(\frac{t}{\beta_1^2}\right)^{-i\nu(k_1)} \\
i\gamma(k_1)e^{2it\theta(k_1,\xi) - 2\chi(k_1,k_1)}\left(\frac{t}{\beta_1^2}\right)^{i\nu(k_1)} & 0
\end{pmatrix}.
\end{equation}
Combining (\ref{fsasMerr}) and (\ref{fsMparas}) one obtains
\begin{equation}\label{kinfty}
\lim\limits_{k\to\infty}k\left(M^{err}(x,t,k)-I\right)=
\tilde{B}(x,t)+\hat{R}(x,t),
\end{equation}
with
\begin{subequations} \label{Btilde}
	\begin{align}
	&\tilde{B}_{11}(x,t)=\frac{\beta_1}{4\sqrt{t}}
	\left(
	\left(w^2(k_1)-\frac{1}{w^2(k_1)}\right)B_{12}(x,t)
	-\left(w^2(k_1)-\frac{1}{w^2(k_1)}\right)B_{21}(x,t)
	\right),\\
	&\tilde{B}_{12}(x,t)=\frac{\beta_1}{4\sqrt{t}}
	\left(
	\left(w(k_1)-\frac{1}{w(k_1)}\right)^2B_{21}(x,t)
	-\left(w(k_1)+\frac{1}{w(k_1)}\right)^2B_{12}(x,t)
	\right),\\
	&\tilde{B}_{21}(x,t)=\frac{\beta_1}{4\sqrt{t}}
	\left(
	\left(w(k_1)-\frac{1}{w(k_1)}\right)^2B_{12}(x,t)
	-\left(w(k_1)+\frac{1}{w(k_1)}\right)^2B_{21}(x,t)
	\right),\\
	&\tilde{B}_{22}(x,t)=-\tilde{B}_{11}(x,t),
	\end{align}
	where we have used the standard notations for the matrix entries of $\tilde{B}(x,t)$ and $B(x,t)$ (recall that $w(k)$ is given by (\ref{fsK})).
	Finally, substituting (\ref{Btilde}) into (\ref{kinfty}) and using (\ref{fssolMerr}) and (\ref{fssol1Merr}) we arrive at (\ref{fssolpw1}) and (\ref{fssolpw2}) respectively.
\end{subequations}

\section{Appendix B}
Proof of Proposition \ref{fsdefh}.

Let us write the Abelian differential $dh(k)$ in an equivalent form:
\begin{equation}
dh(k)=4\frac{k^3+c_2k^2+c_1k+c_0}{\gamma(k)}dk,
\end{equation}
where
\begin{equation}\label{fsc0c1c2}
c_0=-k_0|\alpha|^2,\quad c_1=|\alpha|^2+2k_0\Re\alpha,\quad
c_2=-k_0-2\Re\alpha.
\end{equation}
\textbf{(i)} Taking into account that $\left(\int_{iA}^k\,dh(k)\right)_+=
-\left(\int_{iA}^k\,dh(k)\right)_-$ for $k\in B$, one obtains:
\begin{equation}
h_+(k)+h_-(k)=\int_B\,dh(k),\quad k\in B.
\end{equation}
We normalize the Abel integral $h(k)$ so that its $\mathfrak{b}$-period vanishes:
\begin{equation}\label{fsb-period}
\int_B\,dh(k)=\frac{1}{2}\int_{\mathfrak{b}}dh(k)=0,
\end{equation}
which determines one of the three parameters:
\begin{equation}\label{fsc0int}
c_0=-\frac{\int_B(k^3+c_2k^2+c_1k)\frac{dk}{\gamma(k)}}
{\int_B\frac{dk}{\gamma(k)}}.
\end{equation}
\textbf{(ii)} The jump condition (\ref{fshjumpa}) follows from
$\left(\int_{\alpha}^k\,dh(k)\right)_+=
-\left(\int_{\alpha}^k\,dh(k)\right)_-$ for $k\in\hat\gamma_{21}\cup\hat\gamma_{31}$, 
$\int_{iA}^{\alpha}\,dh(k)=\int_{-iA}^{\overline{\alpha}}\,dh(k)
=-\frac{1}{2}\int_{\mathfrak{a}}\,dh(k)$, and (\ref{fsb-period}). Therefore, this condition on $h(k)$ does not impose additional restrictions on the parameters.
\\
\textbf{(iii)} We start with the analysis of $\Im h(k)$ as $k\to\infty$. From the large-$k$ behavior of $\gamma(k)$
\begin{equation}
\gamma(k)=k^2\left(1-\frac{\Re\alpha}{k}
+\frac{A^2+\Im^2\alpha}{2k^2}+O\left(k^{-3}\right)\right),\quad k\to\infty,
\end{equation}
it follows that 
\begin{equation}
\frac{dh(k)}{dk}=4k+h_0+\frac{h_{-1}}{k}+O\left(k^{-2}\right),\quad k\to\infty,
\end{equation}
where
\begin{equation}
h_0=4(c_2+\Re\alpha),\quad
h_{-1}=4c_1-2A^2-2\Im^2\alpha+h_0\Re\alpha.
\end{equation}
According to the expansion of $\frac{d\theta}{dk}$ as $k\to\infty$ (see (\ref{fsthetakinf})),
we set $h_0=4\xi$ and $h_{-1}=0$, which determines $c_1$ and $c_2$:
\begin{equation}\label{fsc1c2}
c_1=(A^2+\Im^2\alpha)/2-\xi\Re\alpha,\quad
c_2=\xi-\Re\alpha.
\end{equation}
Taking into account that
$$
2\left(\int_{iA}^{k}+\int_{-iA}^{k}\right)(z+\xi)\,dz=2k^2+4\xi k+2A^2,
$$
and expressing $h(k)$ in the form
\begin{equation}
h(k)=2\left(\int_{iA}^{k}+\int_{-iA}^{k}\right)\left[
\frac{z^3+c_2z^2+c_1z+c_0}{\gamma(z)}-
(z+\xi)\right]\,dz+2k^2+4\xi k+2A^2,
\end{equation}
we arrive at (\ref{fshkinf}) with $H_\infty$ given by (\ref{fsH0}).
Combining (\ref{fsc1c2}) and (\ref{fsc0c1c2}) we obtain (\ref{fsreaima}), whereas substituting $c_0=-k_0|\alpha|^2$ into (\ref{fsc0int}) with $\alpha$ given by (\ref{fsreaima}) we arrive at the integral equation (\ref{fsintk0}) for $k_0$. The existence and uniqueness of the solution $k_0=k_0(\xi)$, $k_0\in(-\xi/2,0)$, follows from the implicit function theorem \cite{BKS11} (see also Lemma 5.4 in \cite{BM17}).

At this point we must verify that $\Im h(k)$ has the same signature structure as $\theta(k,\xi)$ in the neighborhood of the points $k=\alpha$, $k=\overline{\alpha}$ and $k=0$. First, let us consider the signature of $\Im h(k)$ in the neighborhood of $k=0$. Again, the expansions
$$
(k^2+A^2)^{\frac{1}{2}}\sim A\,\mathrm{sign}(\Re k),\quad k\to 0\quad
\text{and}\quad [(k-\alpha)(k-\overline{\alpha})]^{\frac{1}{2}}\sim|\alpha|,\quad k\to 0,
$$
imply
$$
\frac{dh}{dk}\sim-\frac{4|\alpha|k_0}{A\,\mathrm{sign}(\Re k)},\quad k\to 0.
$$
From the above expansion of $\frac{dh}{dk}$ we obtain
\begin{equation}
h(k)\sim h(0)-\frac{4|\alpha|k_0}{A\,\mathrm{sign}(\Re k)}k,\quad k\to 0.
\end{equation}
From the definition (\ref{fsh}) of $h(k)$ it follows that $h(0)\in\mathbb{R}$. Therefore, $\Im h(k)$ has the following behavior as $k\to 0$:
\begin{equation}\label{imhz}
\Im h(k)\sim-\frac{4|\alpha|k_0}{A\,\mathrm{sign}(\Re k)}\Im k,\quad k\to 0,
\end{equation}
which implies that $\Im h(k)$ has the same signature structure at the vicinity of $k=0$ as $\Im\theta(k,\xi)$.

Now we consider $k=\alpha$ (the treatment of $k=\overline{\alpha}$ is similar). Taking into account that
$$
(k^2+A^2)^{\frac{1}{2}}=(\alpha^2+A^2)^{\frac{1}{2}}
\left[
1
+O(k-\alpha)
\right],\quad k\to\alpha,
$$
$$
[(k-\alpha)(k-\overline{\alpha})]^{\frac{1}{2}}=
(\alpha-\overline{\alpha})^{\frac{1}{2}}
\left[
(k-\alpha)^{\frac{1}{2}}
+O((k-\alpha)^{\frac{3}{2}})
\right],\quad k\to\alpha,
$$
we arrive at
\begin{align*}
\frac{dh}{dk}=\frac{4(\alpha-\overline{\alpha})^{\frac{1}{2}}(\alpha-k_0)}
{(\alpha^2+A^2)^{\frac{1}{2}}}
\left[
(k-\alpha)^{\frac{1}{2}}+
+O((k-\alpha)^{\frac{3}{2}})
\right],\quad k\to\alpha.
\end{align*}
Integrating the latter expansion we obtain
\begin{equation}\label{fshalph}
h(k)=h(\alpha)+(k-\alpha)^{\frac{3}{2}}H_{\alpha}(k),
\end{equation}
where the analytic function $H_{\alpha}(k)$ has the form
\begin{equation}
H_{\alpha}(k)=\frac{8(\alpha-\overline{\alpha})^{\frac{1}{2}}(\alpha-k_0)}
{3(\alpha^2+A^2)^{\frac{1}{2}}}
\left[
1
+O(k-\alpha)
\right],\quad k\to\alpha.
\end{equation}
In order to have the sign structure of $\Im h(k)$ in the vicinity of $k=\alpha$ the same as of $\Im\theta (k,\xi)$, the curve $\Im h(k)=0$ going out from $k=\alpha$ has to have three branches. Therefore, the leading term of the expansion of $\Im h(k)$ as $k\to\alpha$ has to be $O((k-\alpha)^{\frac{3}{2}})$. In view of (\ref{fshalph}), we have to verify that $\Im h(\alpha)=0$. Observing that $h(\alpha)$ can be written in the form
\begin{equation}
h(\alpha)=\frac{1}{2}\left(\int_{iA}^{\alpha}
+\int_{-iA}^{\overline{\alpha}}\right)dh(k)
+\frac{1}{2}\int_{\overline{\alpha}}^{\alpha}dh(k),
\end{equation}
we conclude that 
\begin{equation}
\Im h(\alpha)=\frac{1}{2}\int_{\overline{\alpha}}^{\alpha}dh(k).
\end{equation}
Since $\int_{\alpha}^{\overline{\alpha}}dh(k)=
\frac{1}{2}\int_{\mathfrak{b}}dh(k)$ and $\int_{\mathfrak{b}}dh(k)=0$, we have that $\Im h(\alpha)=0$.

Notice that since $h(iA)=0$, one branch of the curve $\Im h(k)=0$ has to  connect $k=\alpha$ with $k=iA$.
Since $\Im h(k)$ has the same behavior for large $k$ as $\Im\theta(k,\xi)$ (see (\ref{fshkinf})), the second branch has to go to infinity along the asymptotic line $\Re k=-\xi$.
Therefore, taking into account the behavior of $\Im h(k)$ at $k=0$ and at $k=\infty$, one concludes that the signature of $\Im h(k)$ near $k=\alpha$ is similar to that of $\Im\theta(k,\xi)$ near $k=\alpha$.

\section{Appendix C}
Proof of Theorem \ref{fsth2}.

Despite the essential singularity of $M^{(5)}(x,t,k)$ at $k=k_0$,
we show, by adapting the nonlinear steepest decent method, that the asymptotics in the elliptic wave region can be established under the condition on the winding of the argument of $(1+r_1(k)r_2(k))$ similar to that  in the plane wave region and for the nonlocal problems with decaying \cite{RS, HFX19} and step-like \cite{RS20,RSs} boundary conditions.
An additional care is needed for treating sectionally analytic phase function $h(k)$ in the vicinity of $k=k_0$.

Since some ingredients of the analysis below can be found in the existing works, here we fix the main steps of the proof referring the reader to the relevant literature for details and paying attention to peculiarities of the analysis for the present nonlocal problem.

Define the error matrix $M^{err}(x,t,k)$ as follows (cf. \cite{BV07})
\begin{equation}\label{fsMerrell}
M^{err}(x,t,k)=
\begin{cases}
M^{(5)}(x,t,k)\left[M^{mod}(x,t,k)\right]^{-1}, &k\text{ outside }S_{k_0}^{\varepsilon}\cup S_{\alpha}^{\varepsilon}\cup 
S_{\overline{\alpha}}^{\varepsilon},\\
M^{(5)}(x,t,k)\left[M^{par}_{k_0}(x,t,k)\right]^{-1}
\left[M^{mod}(x,t,k)\right]^{-1},&
k\text{ inside }S_{k_0}^{\varepsilon},\\
M^{(5)}(x,t,k)\left[M^{par}_{\alpha}(x,t,k)\right]^{-1}
\left[M^{mod}(x,t,k)\right]^{-1},&
k\text{ inside }S_{\alpha}^{\varepsilon},\\
M^{(5)}(x,t,k)
\left[M^{par}_{\overline{\alpha}}(x,t,k)\right]^{-1}
\left[M^{mod}(x,t,k)\right]^{-1},&
k\text{ inside }S_{\overline{\alpha}}^{\varepsilon},
\end{cases}
\end{equation}
where $S_{k_0}^{\varepsilon}$, $S_{\alpha}^{\varepsilon}$ and $S_{\overline{\alpha}}^{\varepsilon}$ are the small counterclockwise oriented circles of a small radius $\varepsilon$, $0<\varepsilon\ll 1$, centered at $k=k_0$, $k=\alpha$ and $k=\overline{\alpha}$ respectively.
The construction of the error matrix $M^{err}(x,t,k)$ involves (i) the solution of the model problem $M^{mod}(k)$, (ii) the local parametrix $M^{par}_{k_0}(x,t,k)$ at $k=k_0$, and (iii) the local parametrixes $M^{par}_{\alpha}(x,t,k)$ and $M^{par}_{\overline{\alpha}}(x,t,k)$ at $k=\alpha$ and $k=\overline{\alpha}$ respectively.
The solution $q(x,t)$ of the original IV problem (\ref{fsivp}) is given in terms of $M^{err}$ and $M^{mod}$ by
\begin{subequations}\label{fsasolMerrE}
	\begin{align}\label{fssolMerrE}
	&q(x,t)=2ie^{2itH_{\infty}+2iG_{\infty}(k_0,\alpha)}
	\lim_{k\to\infty}k
	\left(M_{12}^{mod}(x,t,k)+M_{12}^{err}(x,t,k)\right),&& x>0,\\
	\label{fssol1MerrE}
	&q(-x,t)=-2ie^{2itH_{\infty}+2i\overline{G_{\infty}(k_0,\alpha)}}
	\lim_{k\to\infty}k
	\left(\overline{M_{21}^{mod}(x,t,k)}+
	\overline{M_{21}^{err}(x,t,k)}\right),&& x>0.
	\end{align}
\end{subequations}
Below we will show that the main contribution to the asymptotics is given in terms of $M^{mod}(x,t,k)$ and the first decaying terms involve $M^{par}_{k_0}(x,t,k)$ whereas $M^{par}_{\alpha}(x,t,k)$ and $M^{par}_{\overline{\alpha}}(x,t,k)$ contribute to the error term only.

Let us specify yet undetermined functions in the definition (\ref{fsMerrell}) of $M^{err}(x,t,k)$.

(i) \textit{Model problem $M^{mod}(x,t,k)$}

Neglecting all parts of contour $\tilde{\Gamma}$ where the jump matrix $J^{(5)}(x,t,k)$ converges to the identity matrix, we arrive at the following model RH problem on the contour $\hat\gamma_{21}\cup\hat\gamma_{31}\cup B$:
\begin{subequations}
\begin{align}
	\label{fsMparj}
	&M^{mod}_+(x,t,k)=M^{mod}_-(x,t,k)J^{mod}(x,t,k),&& k\in\hat\gamma_{21}\cup\hat\gamma_{31}\cup B,\\
	&M^{mod}(x,t,k)=I+O(k^{-1}),&& k\to\infty,
\end{align}
\end{subequations}
with
\begin{equation}
\label{fsJmodew}
J^{mod}(x,t,k)=
\begin{cases}
\begin{pmatrix}
0& -e^{-it\Omega-i\omega}\\
e^{it\Omega+i\omega}&0
\end{pmatrix},&k\in\hat\gamma_{21}\cup\hat\gamma_{31},\\
i\sigma_1,&k\in B.
\end{cases}
\end{equation}
The jump condition (\ref{fsMparj}) is similar to (5.52a), (5.52b) in \cite{BM17}.
Therefore, arguing  as in \cite{BM17} one concludes that 
$M^{mod}(x,t,k)$ has the from (cf. (5.78) in \cite{BM17})
\begin{equation}\label{fsMmodth}
	M^{mod}(x,t,k)=N^{-1}(t,\infty,c)N(t,k,c).
\end{equation}
The $2\times2$ matrix $N(t,k,c)$ is defined in terms of the theta functions as follows (cf. (5.68) in \cite{BM17})
\begin{equation}\label{fsNk}
N(t,k,c)=\frac{1}{2}
\begin{pmatrix}
	[p(k)+p^{-1}(k)]\mathbf{M}_1(t,k,c)& i[p(k)-p^{-1}(k)]\mathbf{M}_2(t,k,c)\\
	-i[p(k)-p^{-1}(k)]\mathbf{M}_1(t,k,-c)& [p(k)+p^{-1}(k)]\mathbf{M}_2(t,k,-c)
\end{pmatrix},
\end{equation}
with
\begin{equation}
	p(k)=\left[\frac{(k-iA)(k-\alpha)}{(k+iA)(k-\bar\alpha)}\right]^{\frac{1}{4}},
\end{equation}
and
\begin{equation}
\label{fsMth}
\mathbf{M}_j(t,k,c)=\frac{\Theta(\frac{\Omega t}{2\pi}+\frac{\omega}{2\pi}
+\frac{1}{4}+(-1)^{j+1}v(k)+c)}
{e^{(-1)^{j+1}i\frac{\pi}{4}}\Theta((-1)^{j+1}v(k)+c)},
\quad j=1,2.
\end{equation}
Introducing the normalized Abelian differential on the genus-1 Riemann surface $\Sigma$ by
\begin{equation}
dw=\frac{C}{\gamma(k)}\quad \text{with}\ C=\left(\int_{\mathfrak{b}}\frac{dk}{\gamma(k)}\right)^{-1},
\end{equation}
where $\gamma(k)$ is given by (\ref{fsgammaRS}), define $\tau$ as its $\mathfrak{a}$-period:
\begin{equation}
\tau=\int_{\mathfrak{a}}\,dw.
\end{equation}
Then the genus-1 theta function $\Theta(k)$ in (\ref{fsMth}) is defined via the third Jacobi theta function $\theta_3$ as follows:
\begin{equation}\label{fsg1thf}
	\Theta(k)=\theta_3(\pi k,e^{i\pi\tau})\equiv\sum\limits_{l\in\mathbb{Z}}e^{2i\pi lk+i\pi l^2\tau},
\end{equation}
and $v(k)$ in the argument of $\Theta$ in (\ref{fsMth}) is defined as the following Abelian map
\begin{equation}\label{fsv}
v(k)=\int_{iA}^{k}dw.
\end{equation}
The constant $c$ is chosen to compensate a possible singularity of $N(k,c)$:
\begin{equation}\label{fsc}
c=v(\hat{k}_0)+\frac{1}{2}(1+\tau),
\end{equation}
where $\hat{k}_0$ is a unique simple zero of the function $p(k)-p^{-1}(k)$:
\begin{equation}
\hat{k}_0=\frac{A\Re\alpha}{A+\Im\alpha}.
\end{equation}
Finally, we notice that $N(t,\infty, c)$ in (\ref{fsMmodth}) has the form
\begin{equation}\label{fsNinf}
N(t,\infty, c)\equiv\lim\limits_{k\to\infty}N(t,k,c)=
\begin{pmatrix}
\mathbf{M}_1(t,\infty,c)&0\\
0&\mathbf{M}_2(t,\infty,-c)
\end{pmatrix},
\end{equation}
where (see (\ref{fsMth}))
\begin{subequations}
\begin{align}
&\mathbf{M}_j(t,\infty,c)\equiv\lim_{k\to\infty}\mathbf{M}_j(t,k,c)=
\frac{\Theta(\frac{\Omega t}{2\pi}+\frac{\omega}{2\pi}
-\frac{1}{4}+(-1)^{j+1}v_{\infty}+c)}
{e^{(-1)^{j}i\frac{\pi}{4}}\Theta((-1)^{j+1}v_{\infty}+c)},\quad j=1,2.
\end{align}
\end{subequations}
with (see (\ref{fsv}))
\begin{equation}\label{fsvinfty}
v_{\infty}\equiv\lim_{k\to\infty}v(k)=\int_{iA}^{\infty}\,dw.
\end{equation}
Combining (\ref{fsMmodth}), (\ref{fsNk}), (\ref{fsNinf}) and
$$
\frac{i}{2}(p(k)-p^{-1}(k))=\frac{A+\Im\alpha}{2k}+O(k^{-2}),
\quad k\to\infty,
$$
we arrive at
\begin{subequations}\label{fsMmodask}
\begin{align}
&\lim\limits_{k\to\infty}kM^{mod}_{12}(x,t,k)=
\frac{A+\Im\alpha}{2}M_1^{-1}(t,\infty,c)M_2(t,\infty,c),\\
&\lim\limits_{k\to\infty}kM^{mod}_{21}(x,t,k)=
-\frac{A+\Im\alpha}{2}M_1(t,\infty,-c)M_2^{-1}(t,\infty,-c).
\end{align}
\end{subequations}
(ii) \textit{Local parametrix $M^{par}_{k_0}(x,t,k)$}

In the neighborhood of $k=k_0$, the matrix $M^{(5)}(x,t,k)$ satisfies the following jump condition on the contour $\Gamma_{k_0}$ (see Figure \ref{par_k0})
\begin{equation}
M^{(5)}_+(x,t,k)=M^{(5)}_-(x,t,k)J^{(5)}(x,t,k),\quad k\in\Gamma_{k_0},\quad k\text{ inside }S_{k_0}^{\varepsilon},
\end{equation}
where $J^{(5)}(x,t,k)$ on $\Gamma_{k_0}$ has the form
\begin{equation}
\label{fsJpark0}
J^{(5)}=
\begin{cases}
\begin{pmatrix}
1& \frac{\delta^{2}(k,k_0)G^{-2}(k,k_0,\alpha)}{r_1(k)}e^{-2ith}\\
0&1
\end{pmatrix},\, k\in\gamma_{22}^{k_0}\cup\gamma_{23}^{k_0}; \quad
\begin{pmatrix}
1& 0\\
-\frac{e^{2ith}G^{2}(k,k_0,\alpha)}{r_2(k)\delta^{2}(k,k_0)}&1
\end{pmatrix},\, k\in\gamma_{32}^{k_0}\cup\gamma_{33}^{k_0};\\
\begin{pmatrix}
1& \frac{r_2(k)\delta^{2}(k,k_0)G^{-2}(k,k_0,\alpha)}
{1+r_1(k)r_2(k)}e^{-2ith}\\
0& 1\\
\end{pmatrix}
,\, k\in\gamma_1^{k_0};\quad
\begin{pmatrix}
1& 0\\
\frac{-r_1(k)\delta^{-2}(k,k_0)G^{2}(k,k_0,\alpha)}
{1+r_1(k)r_2(k)}e^{2ith}& 1\\
\end{pmatrix}
,\, k\in\gamma_4^{k_0};\\
\begin{pmatrix}
0& -e^{-it\Omega-i\omega}\\
e^{it\Omega+i\omega}& 0
\end{pmatrix},\,k\in\gamma_{21}^{k_0}\cup\gamma_{31}^{k_0}.
\end{cases}
\end{equation}
\begin{figure}[h]
	\begin{minipage}[h]{0.49\linewidth}
		\centering{\includegraphics[width=0.99\linewidth]{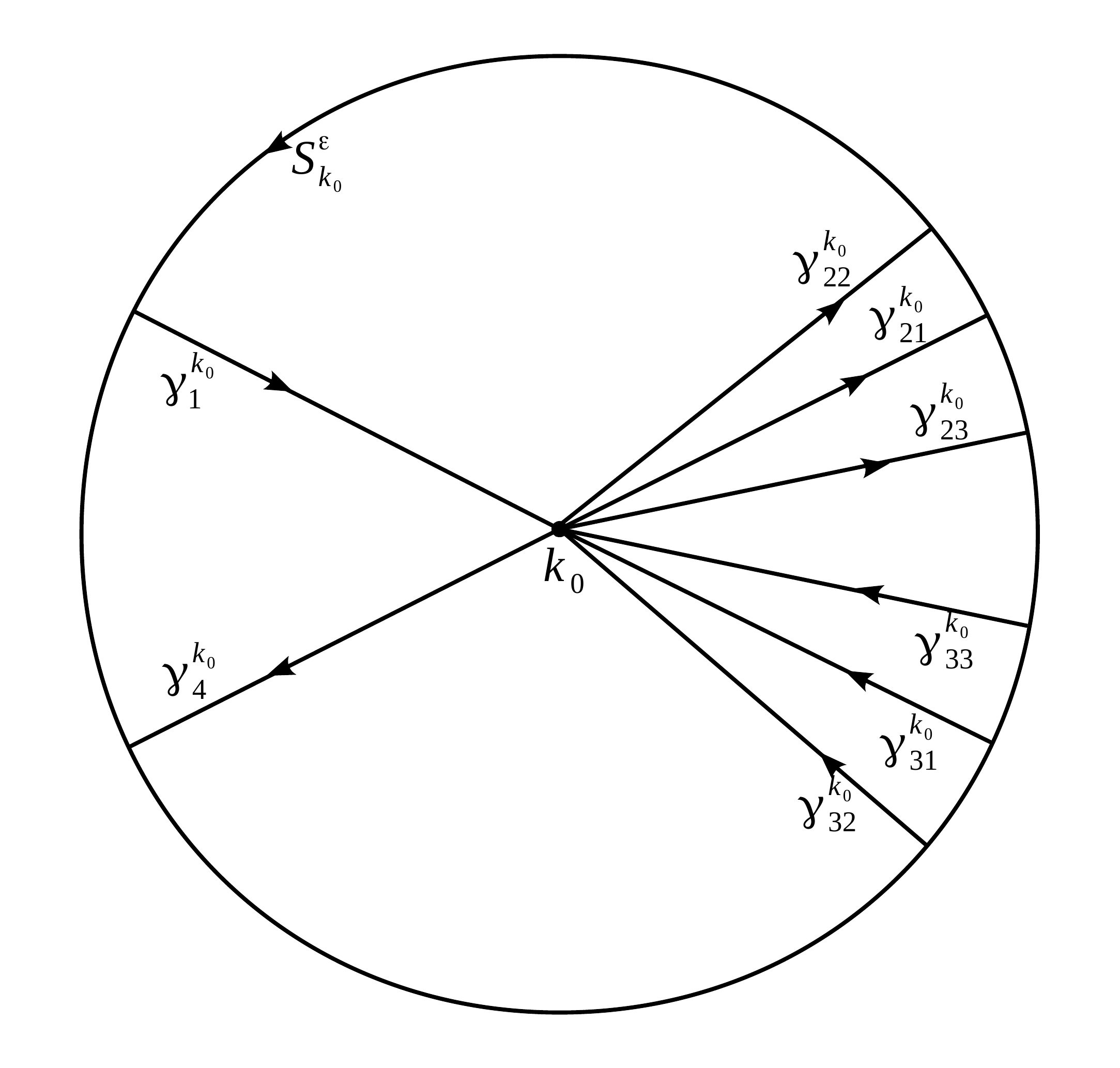}}
		\caption{Contour 
			$\Gamma_{k_0}=\gamma_{1}^{k_0}\cup\gamma_{4}^{k_0}\cup
			\left(
			\bigcup\limits_{i=2}^{3}
			\bigcup\limits_{j=1}^{3}
			\gamma_{ij}^{k_0}\right)$
			in $k$-plane. Here $\gamma_{ij}^{k_0}=\hat\gamma_{ij}$, $\gamma_1^{k_0}=\hat\gamma_1$ and $\gamma_4^{k_0}=\hat\gamma_4$ for $k$ inside $S_{k_0}^{\varepsilon}$.}
		\label{par_k0}
	\end{minipage}
	\hfill
	\begin{minipage}[h]{0.49\linewidth}
		\centering{\includegraphics[width=0.99\linewidth]{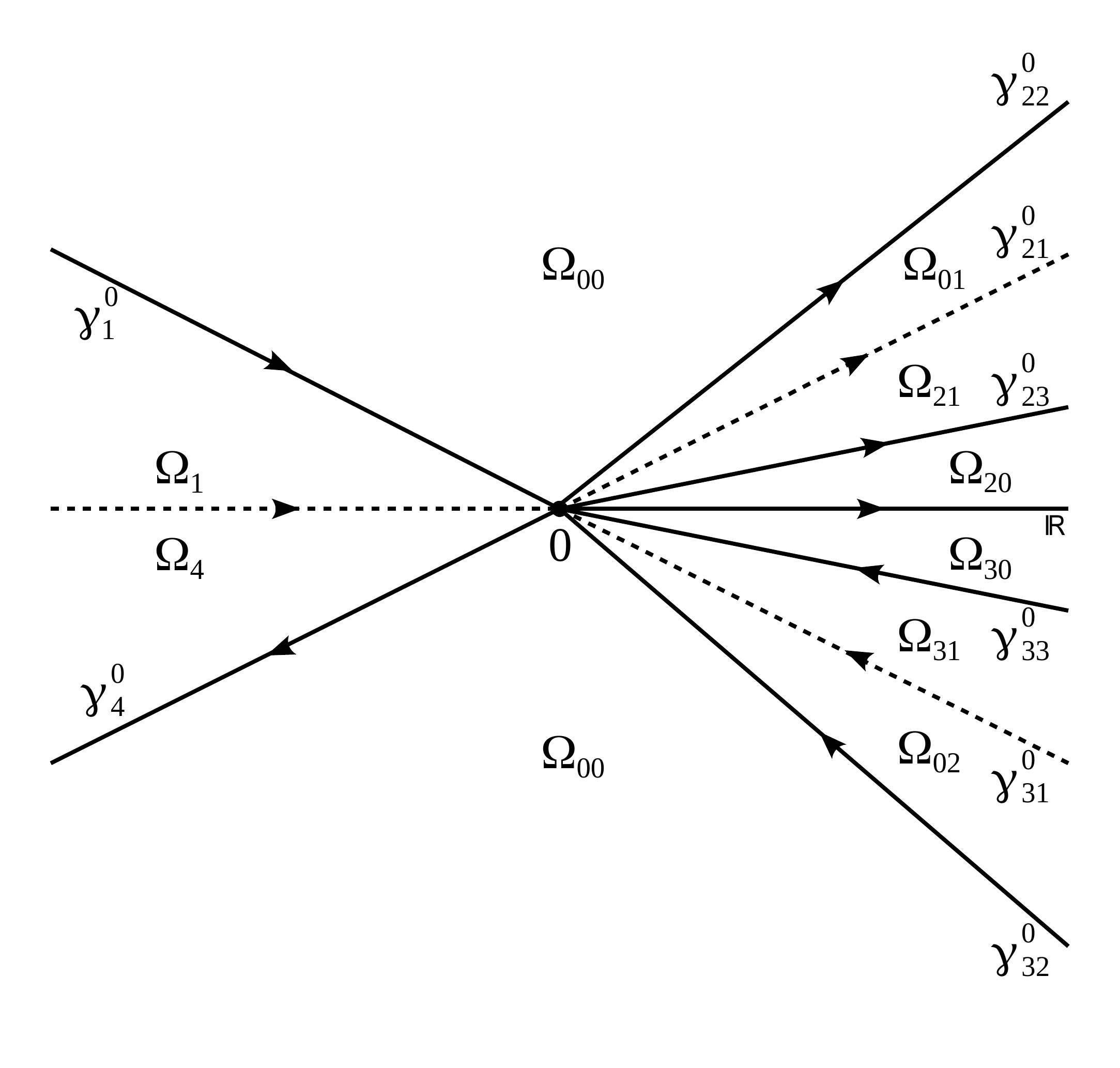}}
		\caption{Contour 
			$\Gamma_{k_0}^0=(0,\infty)\cup\gamma_{1}^{0}\cup\gamma_{4}^{0}\cup
			\left(
			\bigcup\limits_{i=2}^{3}
			\bigcup\limits_{j=2}^{3}
			\gamma_{ij}^{0}\right)$ and domains $\Omega$'s in $z$-plane.}
		\label{par_parab}
	\end{minipage}
\end{figure}
Introduce the piecewise constant function $B(k)$ by (cf. (6.13) in \cite{BLS20})
\begin{equation}
B(k)=
\begin{cases}
e^{-ith_+(k)\sigma_3},&k\text{ on the left of }
\gamma_{21}^{k_0}\cup\gamma_{31}^{k_0},
\quad k\text{ inside }S_{k_0}^{\varepsilon},\\
e^{-ith_-(k)\sigma_3},&k\text{ on the right of }
\gamma_{21}^{k_0}\cup\gamma_{31}^{k_0},
\quad k\text{ inside }S_{k_0}^{\varepsilon},\\
\end{cases}
\end{equation}
and define $M_{0}^{par}(x,t,k)$ in the neighborhood of $k_0$ as follows:
\begin{equation}
M_{0}^{par}(x,t,k)=M^{(5)}(x,t,k)B(k),
\quad k\text{ inside } S_{k_0}^{\varepsilon},\,k\not\in\Gamma_{k_0}.
\end{equation}
Then $M_{0}^{par}(x,t,k)$ satisfies the same jump condition on $\Gamma_{k_0}$ as $M^{(5)}(x,t,k)$ but (i) the jump across $\gamma_{21}^{k_0}\cup\gamma_{31}^{k_0}$ has the form 
$\begin{pmatrix}
0& -e^{-i\omega}\\
e^{i\omega}& 0
\end{pmatrix}$
and, more importantly, (ii) instead of $h(k)$, which is non-analytic inside $S_{k_0}^{\varepsilon}$, we arrive at the new phase function $h_{k_0}(k)$ given by
\begin{equation}
h_{k_0}(k)=
\begin{cases}
h(k)-h_+(k),&k\text{ on the left of }
\gamma_{21}^{k_0}\cup\gamma_{31}^{k_0},
\quad k\text{ inside }S_{k_0}^{\varepsilon},\\
h(k)-h_-(k),&k\text{ on the right of }
\gamma_{21}^{k_0}\cup\gamma_{31}^{k_0},
\quad k\text{ inside }S_{k_0}^{\varepsilon},\\
\end{cases}
\end{equation}
which has no jump across $\gamma_{21}^{k_0}\cup\gamma_{31}^{k_0}$.

Since $G(k,k_0,\alpha)$ does not have a limit as $k\to k_0$, we can't proceed  as in the plane wave case by introducing $\tilde{r}_1(k)=r_1(k)G^2(k,k_0,\alpha)$ and 
$\tilde{r}_2(k)=r_2(k)G^{-2}(k,k_0,\alpha)$, cf. (\ref{fsrjtilde}).
To eliminate the singularity at $k=k_0$ of $M_0^{par}$, we
introduce $\hat G(k)=\hat G(k,k_0,\alpha)$ so that $Y_{k_0}(x,t,k)$ (see (\ref{fsY}) below) is analytic inside $S_{k_0}^{\varepsilon}$:
\begin{equation}
\hat G(k)=
\begin{cases}
G^{\sigma_3}(k,k_0,\alpha),\quad k\text{ on the left of }
\gamma_{21}^{k_0}\cup\gamma_{31}^{k_0},
\quad k\text{ inside }S_{k_0}^{\varepsilon};\\
\begin{pmatrix}
0& \frac{\delta^2(k,k_0)}{r_1(k)}e^{-2ith_{k_0}(k)}\\
-\frac{r_1(k)}{\delta^2(k,k_0)}e^{2ith_{k_0}(k)}& 0
\end{pmatrix}
G^{\sigma_3}(k,k_0,\alpha),\\
\qquad\qquad\qquad\qquad\qquad k\text{ on the right of }
\gamma_{21}^{k_0}\cup\gamma_{31}^{k_0},
\, k\text{ inside }S_{k_0}^{\varepsilon},\,k\in\mathbb{C}^{+};\\
\begin{pmatrix}
0& \delta^2(k,k_0)r_2(k)e^{-2ith_{k_0}(k)}\\
-\delta^{-2}(k,k_0)r_2^{-1}(k)e^{2ith_{k_0}(k)}& 0
\end{pmatrix}
G^{\sigma_3}(k,k_0,\alpha),\\
\qquad\qquad\qquad\qquad\qquad k\text{ on the right of }
\gamma_{21}^{k_0}\cup\gamma_{31}^{k_0},
\, k\text{ inside }S_{k_0}^{\varepsilon},\,k\in\mathbb{C}^{-}.\\
\end{cases}
\end{equation}
Define $M^{par}_1(x,t,k)$ in the neighborhood of $k_0$ as follows:
\begin{equation}
M_{1}^{par}(x,t,k)=M_{0}^{par}(x,t,k)\hat G^{-1}(k),
\quad k\text{ inside } S_{k_0}^{\varepsilon},\,k\not\in\Gamma_{k_0}.
\end{equation}
Then $M_1^{par}(x,t,k)$ satisfies the following jump condition across $\Gamma_{k_0}\cup(k_0,k_0+\varepsilon)$:
\begin{equation}
(M_1^{par})_+(x,t,k)=(M_1^{par})_-(x,t,k)J_1^{par}(x,t,k),
\quad k\in\Gamma_{k_0}\cup(k_0,k_0+\varepsilon),
\quad k\text{ inside }S_{k_0}^{\varepsilon},
\end{equation}
where
\begin{equation}
\label{fsJpar1}
J_1^{par}=
\begin{cases}
\begin{pmatrix}
1& \frac{r_2(k)\delta^{2}(k,k_0)}
{1+r_1(k)r_2(k)}e^{-2ith_{k_0}}\\
0& 1\\
\end{pmatrix}
,\, k\in\gamma_1^{k_0};\quad
\begin{pmatrix}
1& 0\\
\frac{-r_1(k)\delta^{-2}(k,k_0)}
{1+r_1(k)r_2(k)}e^{2ith_{k_0}}& 1\\
\end{pmatrix}
,\, k\in\gamma_4^{k_0};\\
\begin{pmatrix}
1& \frac{\delta^{2}(k,k_0)}{r_1(k)}e^{-2ith_{k_0}}\\
0&1
\end{pmatrix},\, k\in\gamma_{22}^{k_0}; \quad
\begin{pmatrix}
1& 0\\
-\frac{r_1(k)}{\delta^{2}(k,k_0)}e^{2ith_{k_0}}
\end{pmatrix}
,\, k\in\gamma_{23}^{k_0};\\
\begin{pmatrix}
1& 0\\
-\frac{e^{2ith_{k_0}}}{r_2(k)\delta^{2}(k,k_0)}&1
\end{pmatrix},\, k\in\gamma_{32}^{k_0};\quad
\begin{pmatrix}
1& r_2(k)\delta^{2}(k,k_0)e^{-2ith_{k_0}}\\
0&1
\end{pmatrix},\, k\in\gamma_{33}^{k_0};\\
I,\,k\in\gamma_{21}^{k_0}\cup\gamma_{31}^{k_0};\quad
\begin{pmatrix}
r_1(k)r_2(k)& 0\\
0& r_1^{-1}(k)r_2^{-1}(k)
\end{pmatrix},\,k\in(k_0,k_0+\varepsilon).
\end{cases}
\end{equation}

Now we are in a position to approximate $M^{par}_1(x,t,k)$ by an exactly solvable RH problem,  which, in turn, leads to the approximation of $M^{(5)}(x,t,k)$ by a local parametrix inside $S_{k_0}^{\varepsilon}$.
Indeed, expanding the phase function $h_{k_0}(k)$ in a vicinity of the stationary phase point $k=k_0$ (cf. (\ref{fsths}))
\begin{equation}\label{fshk0s}
h_{k_0}(k)=\sum\limits_{n=2}^{\infty}\hat{h}_n(k-k_0)^n,
\quad \hat{h}_n\equiv\hat{h}_n(k_0)=\frac{h_{k_0}^{(n)}(k)}{n!},
\end{equation}
define the scaled parameter $z$ by
\begin{equation}
z=2\sqrt{t}(k-k_0)\left(
\sum\limits_{n=2}^{\infty}\hat{h}_n(k-k_0)^{n-2}\right)^{1/2}\equiv
\sqrt{t}\sum\limits_{n=1}^{\infty}\hat\alpha_n(k-k_0)^n.
\end{equation}
Then the phase function $h_{k_0}(k)$ can be written in terms of $z$ as follows (cf. (\ref{fsthz})):
\begin{equation}
h_{k_0}(k(z))=\frac{z^2}{4t},
\end{equation}
and $(k-k_0)$ can be expressed in terms of $z$ by (cf. (\ref{fskk1}))
\begin{equation}\label{fskk0}
k-k_0=\sum\limits_{n=1}^{\infty}\hat\beta_n
\left(\frac{z}{\sqrt{t}}\right)^n,\quad
k\text{ inside }S_{k_1}^{\varepsilon},
\end{equation}
where $\hat\beta_n$ can be found recursively in terms of $\hat\alpha_n$. Taking into account the following approximations
\begin{equation}
r_j(k(z))\approx r_j(k_0),\quad j=1,2,\qquad
\delta(k(z),k_1)\approx \hat\beta_1^{i\nu(k_0)}e^{\chi(k_0,k_0)}
t^{-\frac{i\nu(k_0)}{2}}z^{i\nu(k_0)},
\end{equation}
 we introduce $M^{par}(x,t,k)$ as follows:
\begin{equation}
M^{par}_{k_0}(x,t,k)=B(k)\hat G^{-1}(k)
\hat\Delta_0(k_0,t)
m^{\Gamma^{0}_{k_0}}(k_0,z(k))
\hat\Delta_0^{-1}(k_0,t)\hat G(k)B^{-1}(k).
\end{equation}
Here $\hat\Delta_0(k_0,t)=
\hat\beta_1^{i\nu(k_0)\sigma_3}
t^{-\frac{i\nu(k_0)}{2}\sigma_3}
e^{\chi(k_0,k_0)\sigma_3}$ and $m^{\Gamma^{0}_{k_0}}(k_0,z)$ (the ``simplified'' matrix $M^{par}_1(x,t,k(z))$) is given by (see Figure \ref{par_parab}; cf. (\ref{fsm-g-0}))
\begin{equation}
m^{\Gamma^{0}_{k_0}}(k_0,z)=m_0(k_0,z)
\begin{cases}
D_0(k_0,z),&z\in\Omega_{00},\\
\begin{pmatrix}
1& -\frac{r_2(k_0)}{1+r_1(k_0)r_2(k_0)}\\
0&1
\end{pmatrix}D_0^{-1}(k_0,z),&z\in\Omega_1,\\
\begin{pmatrix}
1& 0\\
\frac{r_1(k_0)}{1+r_1(k_0)r_2(k_0)}&1
\end{pmatrix}D_0^{-1}(k_0,z),&z\in\Omega_4,\\
\begin{pmatrix}
1& -r_1^{-1}(k_0)\\
0&1
\end{pmatrix}D_0^{-1}(k_0,z),&z\in\Omega_{01}\cup\Omega_{21},\\
\begin{pmatrix}
1& 0\\
r_2^{-1}(k_0)&1
\end{pmatrix}D_0^{-1}(k_0,z),&z\in\Omega_{02}\cup\Omega_{31},\\
\begin{pmatrix}
0& -r_1^{-1}(k_0)\\
r_1(k_0)&1
\end{pmatrix}D_0^{-1}(k_0,z),&z\in\Omega_{20},\\
\begin{pmatrix}
1& -r_2(k_0)\\
r_2^{-1}(k_0)&0
\end{pmatrix}D_0^{-1}(k_0,z),&z\in\Omega_{30},\\
\end{cases}
\end{equation}
where $D_0(k_0,z)$ is given by (\ref{fsD0}) and $m_0(k_0,z)$ solves the RH problem (\ref{fsm0}) with a constant jump matrix.

Arguing as in Appendix A (see also, e.g., \cite{RS, BLS20, L17}) one concludes that the main contribution to the asymptotics of $M^{err}(x,t,k)$ is given by the integral along the circle $S_{k_0}^{\varepsilon}$:
\begin{equation}\label{fsasMerrk0}
\lim\limits_{k\to\infty} k(M^{err}(x,t,k)-I)=-\frac{1}{2\pi i}
\int\limits_{|k-k_0|=\varepsilon}M^{mod}
\left(\left[M^{par}_{k_0}(x,t,k)\right]^{-1}
-I\right)\left[M^{mod}\right]^{-1}\,dk+\hat{R}(x,t),
\end{equation}
where $\hat{R}(x,t)$ is the same as in (\ref{fsasMerr}) and $\left[M^{par}_{k_0}(x,t,k)\right]^{-1}$ is inferred from the large-$z$ asymptotics of $m^{\Gamma_{k_0}^0}$, which is the same as (\ref{fsmgamma}) but with $\beta(k_0)$ and $\gamma(k_0)$ instead of $\beta(k_1)$ and $\gamma(k_1)$ respectively:
\begin{align}\label{fsMpark0as}
\nonumber
\left[M^{par}(x,t,k)\right]^{-1} &=B(k)\hat G^{-1}(k)\hat\Delta_0(k_0,t)
\left[m^{\Gamma_{k_0}^{0}}(x, \hat\beta_1^{-1}\sqrt{t}(k-k_0))\right]^{-1}
\hat\Delta_0^{-1}(k_0,t)\hat G(k)B^{-1}(k)\\
&= I + B(k)\hat G^{-1}(k)
\frac{\hat\beta_1\hat{B}(x,t)}{\sqrt{t}(k-k_0)}
\hat G(k)B^{-1}(k)
+\tilde{R}(x,t),
\end{align}
where $\tilde{R}(x,t)$ is the same as in (\ref{fsMparas}) and
(cf. (\ref{fsB}))
\begin{equation}
\hat{B}(x,t)=\begin{pmatrix}
0 & -i\beta(k_0)e^{2\chi(k_0,k_0)}
\left(\frac{t}{\hat\beta_1^2}\right)^{-i\nu(k_0)} \\
i\gamma(k_0)e^{-2\chi(k_0,k_0)}\left(\frac{t}{\hat\beta_1^2}\right)^{i\nu(k_0)} & 0
\end{pmatrix}.
\end{equation}
From (\ref{fsasMerrk0}) and (\ref{fsMpark0as}) we obtain
\begin{subequations}\label{kinftyk0}
\begin{align}
&\lim\limits_{k\to\infty}
\left[k\left(M^{err}(x,t,k)-I\right)\right]_{12}=
t^{-\frac{1}{2}-\Im\nu(k_0)}\tilde c_5(k_0,t)
+t^{-\frac{1}{2}+\Im\nu(k_0)}\tilde c_6(k_0,t)+R(x,t),\\
&\lim\limits_{k\to\infty}
\left[k\left(M^{err}(x,t,k)-I\right)\right]_{21}=
t^{-\frac{1}{2}-\Im\nu(k_0)}\tilde c_7(k_0,t)
+t^{-\frac{1}{2}+\Im\nu(k_0)}\tilde c_8(k_0,t)+R(x,t),
\end{align}
\end{subequations}
with $R(x,t)$, $j=1,2$ given by (\ref{fsR3}) and
\begin{subequations}
\label{fstck0}
\begin{align}
&\tilde c_5(k_0,t)=i\gamma(k_0)
e^{-2\chi(k_0,k_0)}\hat\beta_1^{-2i\nu(k_0)+1}
t^{i\Re\nu(k_0)}(Y_{12})^{2}(x,t,k_0),\\
&\tilde c_6(k_0,t)=i\beta(k_0)
e^{2\chi(k_0,k_0)}\hat\beta_1^{2i\nu(k_0)+1}
t^{-i\Re\nu(k_0)}(Y_{11})^{2}(x,t,k_0),\\
&\tilde c_7(k_0,t)=-i\gamma(k_0)
e^{-2\chi(k_0,k_0)}\hat\beta_1^{-2i\nu(k_0)+1}
t^{i\Re\nu(k_0)}(Y_{22})^{2}(x,t,k_0),\\
&\tilde c_8(k_0,t)=-i\beta(k_0)
e^{2\chi(k_0,k_0)}\hat\beta_1^{2i\nu(k_0)+1}
t^{-i\Re\nu(k_0)}(Y_{21})^{2}(x,t,k_0),
\end{align}
\end{subequations}
where the analytical (owing to the appropriate choice of $\hat G(k)$) function $Y(x,t,k)$ has the form
\begin{equation}\label{fsY}
Y(x,t,k)=M^{mod}(x,t,k)B(k)\hat G^{-1}(k).
\end{equation}
(iii) \textit{Local parametrixes $M^{par}_{\alpha}(x,t,k)$ and $M^{par}_{\overline{\alpha}}(x,t,k)$}

In the neighborhood of $k=\alpha$, the  parametrix $M_{\alpha}^{par}(x,t,k)$ satisfies the following jump condition on the contour $\Gamma_{\alpha}$ (see Figure \ref{par_alph})
\begin{equation}
(M_{\alpha}^{par})_+(x,t,k)=(M_{\alpha}^{par})_-(x,t,k)J_{\alpha}^{par}(x,t,k),\quad k\in\Gamma_{\alpha},
\quad k\text{ inside }S_{\alpha}^{\varepsilon},
\end{equation}
where
\begin{equation}
J_{\alpha}^{par}=
\begin{cases}
\begin{pmatrix}
1& 0\\
\frac{r_1(k)}{\delta^{2}(k,k_0)}G^{2}(k,k_0,\alpha)e^{2ith}&1
\end{pmatrix},\, k\in\gamma_{20}^\alpha;\quad
\begin{pmatrix}
0& -e^{-it\Omega-i\omega}\\
e^{it\Omega+i\omega}&0
\end{pmatrix},\,k\in\gamma_{21}^\alpha;\\
\begin{pmatrix}
1& \frac{\delta^{2}(k,k_0)}{r_1(k)}G^{-2}(k,k_0,\alpha)e^{-2ith}\\
0&1
\end{pmatrix},\, k\in\gamma_{22}^\alpha\cup\gamma_{23}^\alpha.
\end{cases}
\end{equation}
\begin{figure}[h]
	\begin{minipage}[h]{0.49\linewidth}
		\centering{\includegraphics[width=0.99\linewidth]{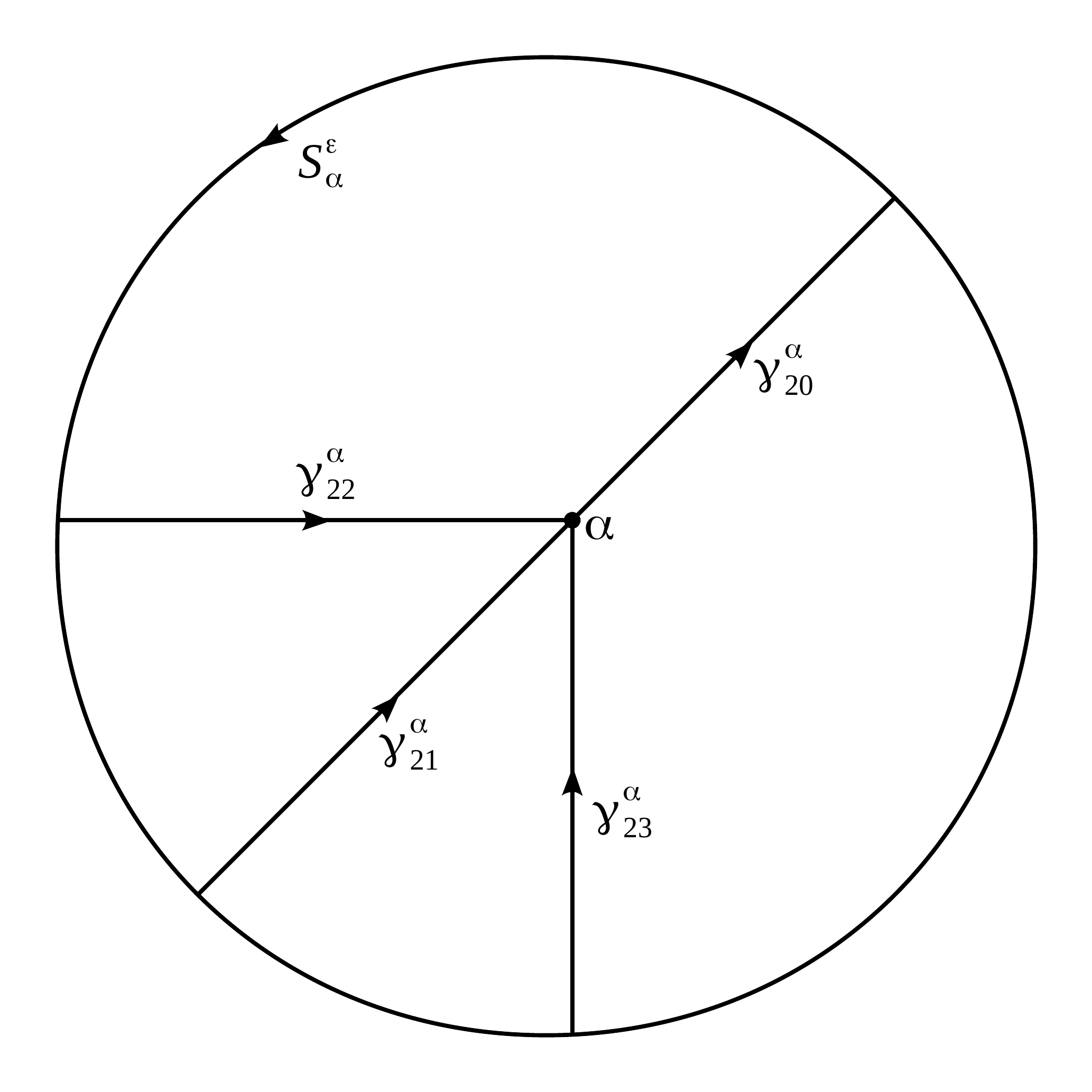}}
		\caption{Contour 
			$\Gamma_{\alpha}=\bigcup\limits_{j=0}^{3}\gamma_{2j}^{\alpha}$
			in $k$-plane. Here $\gamma_{2j}^{\alpha}=\hat\gamma_{2j}$ for $k$ inside $S_{\alpha}^{\varepsilon}$, $j=\overline{0,3}$.}
		\label{par_alph}
	\end{minipage}
	\hfill
	\begin{minipage}[h]{0.49\linewidth}
		\centering{\includegraphics[width=0.99\linewidth]{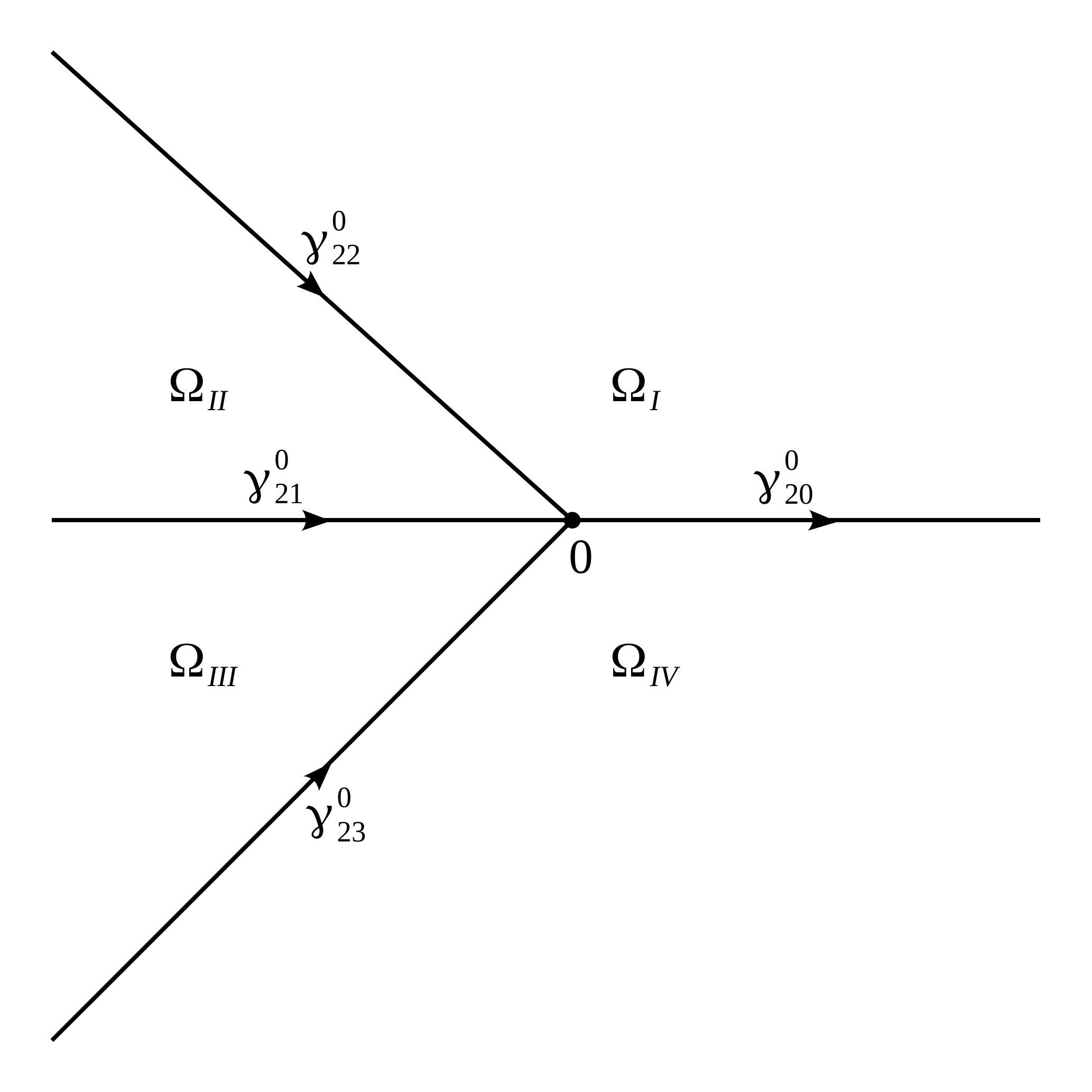}}
		\caption{Contour 
		$\Gamma_{\alpha}^0=\bigcup\limits_{j=0}^{3}\gamma_{2j}^0$ and domains $\Omega_{j}, j=I,\dots,IV$ in $z$-plane.}
		\label{par_Ai}
	\end{minipage}
\end{figure}
Introduce $\tilde{M}_{\alpha}^{par}(x,t,k)$ as follows:
\begin{equation}\label{fstMpardef}
\tilde{M}_{\alpha}^{par}=E^{-1}(x,t,k)
\left(\frac{\delta(k,k_0)}{\sqrt{r_1(k)}}G^{-1}(k,k_0,\alpha)e^{-ith(\alpha)}\right)^{-\sigma_3}
M^{par}_{\alpha}
\left(\frac{\delta(k,k_0)}{\sqrt{r_1(k)}}G^{-1}(k,k_0,\alpha)e^{-ith(k)}\right)^{\sigma_3},
\end{equation}
where the $2\times2$ matrix $E(x,t,k)$ will be determined below. Direct calculations show that $\tilde{M}_{\alpha}^{par}(x,t,k)$ satisfies the following jump condition across $\Gamma_{\alpha}$:
\begin{equation}\label{fstMparj}
(\tilde M_{\alpha}^{par})_+(x,t,k)=(\tilde M_{\alpha}^{par})_-(x,t,k)
\tilde J_{\alpha}^{par}(x,t,k),\quad k\in\Gamma_{\alpha},
\quad k\text{ inside }S_{\alpha}^{\varepsilon}
\end{equation}
where
\begin{equation}
\tilde J_{\alpha}^{par}=
\begin{cases}
\begin{pmatrix}
1& 0\\
1&1
\end{pmatrix},\, k\in\gamma_{20}^\alpha;\quad
\begin{pmatrix}
0& -1\\
1&0
\end{pmatrix},\,k\in\gamma_{21}^\alpha;\\
\begin{pmatrix}
1& 1\\
0&1
\end{pmatrix},\, k\in\gamma_{22}^\alpha\cup\gamma_{23}^\alpha.
\end{cases}
\end{equation}
The latter jump conditions are satisfied by a sectionally analytic matrix function that can be determined in terms of the Airy functions \cite{DKMVZ, BV07}. Indeed, the function (see Figure \ref{par_Ai})
\begin{equation}\label{fstMparz}
\tilde{M}^{par}_\alpha(z)=
\begin{cases}
\begin{pmatrix}
Ai(e^{4i\pi/3}z)& Ai(z)\\
e^{4i\pi/3}Ai^\prime(e^{4i\pi/3}z)& Ai^\prime(z)
\end{pmatrix}e^{i\pi\sigma_3/6},&z\in\Omega_\mathrm{I},\\
\begin{pmatrix}
Ai(e^{4i\pi/3}z)& Ai(z)\\
e^{4i\pi/3}Ai^\prime(e^{4i\pi/3}z)& Ai^\prime(z)
\end{pmatrix}e^{i\pi\sigma_3/6}
\begin{pmatrix}
1& -1\\
0& 1
\end{pmatrix}
,&z\in\Omega_\mathrm{II},\\
\begin{pmatrix}
-e^{4i\pi/3}Ai(e^{2i\pi/3}z)& Ai(z)\\
-Ai^\prime(e^{2i\pi/3}z)& Ai^\prime(z)
\end{pmatrix}e^{i\pi\sigma_3/6}
\begin{pmatrix}
1& 1\\
0& 1
\end{pmatrix},
&z\in\Omega_\mathrm{III},\\
\begin{pmatrix}
-e^{4i\pi/3}Ai(e^{2i\pi/3}z)& Ai(z)\\
-Ai^\prime(e^{2i\pi/3}z)& Ai^\prime(z)
\end{pmatrix}e^{i\pi\sigma_3/6},
&z\in\Omega_\mathrm{IV}.
\end{cases}
\end{equation}
satisfies the jump conditions (\ref{fstMparj}) on $\Gamma_{\alpha}^0$, which can be checked by using the identity \cite{ONIST}
$$
Ai(z)+e^{2i\pi/3}Ai(e^{2i\pi/3}z)+e^{4i\pi/3}Ai(e^{4i\pi/3}z)=0.
$$
The scaled variable $z=z(x,t,k)$ is chosen according to the asymptotics of the Airy functions:
\begin{subequations}\label{fsAias}
\begin{align}
&Ai(z)\sim\frac{1}{2\sqrt{\pi}}z^{-\frac{1}{4}}e^{-\frac{2}{3}z^{3/2}}
(1+O(z^{-\frac{3}{2}})),&&|\arg(z)|<\pi,\quad z\to\infty,\\
&Ai^\prime(z)\sim-\frac{1}{2\sqrt{\pi}}z^{\frac{1}{4}}e^{-\frac{2}{3}z^{3/2}}
(1+O(z^{-\frac{3}{2}})),&&|\arg(z)|<\pi,\quad z\to\infty.
\end{align}
\end{subequations}
In order to eliminate the exponentially growing term $e^{ith(k)}$ in (\ref{fstMpardef}), the asymptotics of $Ai(z)$ suggests to introduce the scaled variable $z(x,t,k)$ in the following way:
\begin{equation}
z(x,t,k)=(k-\alpha)\left(-\frac{3}{2}itH_{\alpha}(k)\right)^{\frac{2}{3}},
\end{equation}
where the analytic function $H_{\alpha}(k)$ is given by (\ref{fshalph}).
Choosing an appropriate branch in the definition of $\zeta$ and deforming the contours $\gamma_{2j}^{\alpha}$, $j=\overline{0,3}$ if necessary, we obtain that the biholomorphism $z(x,t,k)$ maps $\gamma_{2j}^{0}$ onto $\gamma_{2j}^{\alpha}$ for all $j=\overline{0,3}$.

From (\ref{fsAias}) and (\ref{fstMparz}) we have
\begin{equation}
\tilde{M}_{\alpha}^{par}(z)e^{-\frac{2}{3}z^{3/2}\sigma_3}=
\frac{e^{i\pi/12}}{2\sqrt{\pi}}
\begin{pmatrix}
z^{-\frac{1}{4}}&0\\
0&z^{\frac{1}{4}}
\end{pmatrix}
\left[
\begin{pmatrix}
1& 1\\
1& -1
\end{pmatrix}e^{i\pi\sigma_3/4}
+O(z^{-\frac{3}{2}})
\right].
\end{equation}
The latter asymptotic expansion suggests to define the (yet undetermined) function $E(x,t,k)$ as follows:
\begin{equation}
E(x,t,k)=\sqrt{\pi}e^{-i\pi/12}e^{-i\pi\sigma_3/4}
\begin{pmatrix}
1& 1\\
1& -1
\end{pmatrix}
\begin{pmatrix}
z^{\frac{1}{4}}(x,t,k)& 0\\
0& z^{-\frac{1}{4}}(x,t,k)
\end{pmatrix},
\end{equation}
which implies
\begin{equation}\label{fsMalph}
\tilde{M}^{par}_{\alpha}(x,t,k)\equiv\tilde{M}^{par}_{\alpha}(z(x,t,k))=I+O(t^{-1}),\quad t\to\infty.
\end{equation}
The treatment of the local parametrix at $k=\bar{\alpha}$ is similar.

Combining (\ref{fsasolMerrE}), (\ref{fsMmodask}), (\ref{kinftyk0}) and (\ref{fsMalph}) we obtain (\ref{fsasellw}).
\end{appendices}


\begin{thebibliography}{99}
	\bibitem{AMP}
	M.J. Ablowitz and Z.H. Musslimani, Integrable nonlocal nonlinear Schr\"odinger equation, Phys. Rev. Lett. 110 (2013) 064105.
	\bibitem{AM19}
	M.J. Ablowitz and Z.H. Musslimani, Integrable nonlocal asymptotic reductions of physically significant nonlinear equations, J. Phys. A: Math. Theor. 52 (2019) 15LT02.
	\bibitem{AM17}
	M.J. Ablowitz and Z.H. Musslimani, Integrable nonlocal nonlinear equations, Stud. Appl. Math. 139 (2017) 7--59.
	
	\bibitem{AKNS}
	M.J. Ablowitz, D.J. Kaup, A.C. Newell, and H. Segur, 
	The inverse scattering transform --- Fourier analysis for nonlinear problems, Stud. Appl. Math. 53 (1974) 249--315.
	\bibitem{ALM20}
	M.J. Ablowitz, X.-D. Luo and Z.H. Musslimani,
	Discrete nonlocal nonlinear Schr\"odinger systems: Integrability, inverse scattering and solitons, Nonlinearity 33 (2020) 3653--3707.
	\bibitem{ALM18}
	M.J. Ablowitz, X.-D. Luo and Z.H. Musslimani,
	Inverse scattering transform for the nonlocal nonlinear Schr\"odinger equation with nonzero boundary conditions, J. Math. Phys. 59(1) (2018), 011501.
	
	
	\bibitem{B16}
	C.M. Bender, PT symmetry in quantum physics: From a mathematical curiosity to optical experiments, Europhysics News 47(2) (2016) 17--20.
	
	\bibitem{BB}
	C.M. Bender and S. Boettcher, Real spectra in non-Hermitian Hamiltonians having P-T symmetry, Phys. Rev. Lett. 80 (1998), 5243.
	
	\bibitem{BF67}
	T.B. Benjamin and J.E. Feir, The disintegration of wavetrains in deep water. Part I, J. Fluid Mech. 27 (1967) 417--430.
	
	\bibitem{BM19}
	D. Bilman  and  P. Miller, A  robust  inverse  scattering  transform  for  the  focusing  nonlinear  Schr\"odinger equation, Comm. Pure Appl. Math. 72(8) (2019) 1722--1805.
	
	\bibitem{BK14}
	G. Biondini and G. Kovacic, Inverse scattering transform for the focusing nonlinear Schr\"odinger equation with nonzero boundary conditions, J. Math. Phys. 55 (2014), 031506.
	\bibitem{BLMT18}
	G. Biondini, S. Li, D. Mantzavinos, and S. Trillo, Universal Behavior of Modulationally Unstable Media, SIAM Review, 60(4) (2018) 888--908.
	\bibitem{BLM21}
	G. Biondini, S. Li and D. Mantzavinos, Long-Time Asymptotics for the Focusing Nonlinear Schr\"odinger Equation with Nonzero Boundary Conditions in the Presence of a Discrete Spectrum, Commun. Math. Phys. 382 (2021) 1495--1577.
	\bibitem{BM16}
	G. Biondini and D. Mantzavinos, Universal Nature of the Nonlinear Stage of Modulational Instability, Phys. Rev. Lett. 116 (2016) 043902.
	\bibitem{BM17}
	G. Biondini and D. Mantzavinos, Long-time asymptotics for the focusing nonlinear Schr\"odinger	equation with nonzero boundary conditions at infinity and asymptotic	stage of modulational instability, Comm. Pure Appl. Math. 70 (2017) 2300--2365.
	
	\bibitem{BKS11}
	A. Boutet de Monvel, V.P. Kotlyarov and D. Shepelsky, Focusing NLS Equation: Long-time dynamics of step-like initial data, Int. Math. Res. Not. 7 (2011) 1613--1653.
	\bibitem{BLS20}
	A. Boutet de Monvel, J. Lenells and D. Shepelsky, The focusing NLS equation with step-like oscillating background: the genus 3 sector, preprint 
	arXiv:2005.02822  
	\bibitem{BLS21-cimp}
	A. Boutet de Monvel, J. Lenells and D. Shepelsky, The Focusing NLS Equation with Step-Like Oscillating Background: Scenarios of Long-Time Asymptotics, Commun. Math. Phys. 383 (2021) 893--952.
	\bibitem{BS04}
	A. Boutet de Monvel and D. Shepelsky, Initial boundary value problem for the mKdV equation on a finite interval, Ann. Inst. Fourier 54(5) (2004) 1477--1495.
	\bibitem{BV07}
	R. Buckingham and S. Venakides, Long-time asymptotics of the nonlinear Schr\"odinger equation shock problem, Comm. Pure Appl. Math. 60 (2007), 1349--1414.
	
	\bibitem{DR09}
	M. Daniel and S. Rajasekar, Nonlinear Dynamics, Narosa, 2009.

	\bibitem{DZ}
	P.A. Deift and X. Zhou, A steepest descend method for oscillatory Riemann--Hilbert problems. Asymptotics for the MKdV equation, Ann. Math. 137(2), (1993) 295--368.
	\bibitem{DKMVZ}
	P. Deift, T. Kriecherbauer, K.T.-R. McLaughlin, S. Venakides, X. Zhou, Uniform asymptotics for polynomials orthogonal with respect to varying exponential weights and applications to universality questions in random matrix theory. Comm. Pure Appl. Math. 52(11) (1999) 1335--1425.
	\bibitem{DIZ}
	P.A. Deift, A.R. Its and X. Zhou, Long-time asymptotics for integrable nonlinear wave equations,
	in A.S. Fokas, V.E. Zakharov (Eds.), Important developments in Soliton Theory 1980-1990, New York: Springer, 1993, pp. 181--204.
	\bibitem{DVZ94}
	P.A. Deift, S. Venakides, and X. Zhou, The collisionless shock region for the long-time behavior of solutions of the KdV equation, Comm. Pure Appl. Math. 47 (2) (1994), 199--206.
	\bibitem{DVZ97}
	P.A. Deift, S. Venakides, and X. Zhou, New results in small dispersion KdV by an extension of the steepest descent method for Riemann–Hilbert problems, Int. Math. Res. Not. 6 (1997) 285--299.
	
	\bibitem{DDEG14}
	J.M. Dudley, F. Dias, M. Erkintalo and G. Genty, Instabilities, breathers and rogue waves in optics. Nat. Photonics 8 (2014) 755--764.
	
	\bibitem{ET20}
	G. El and A. Tovbis, Spectral theory of soliton and breather gases for the focusing nonlinear Schr\"dinger equation, Phys. Rev. E 101 (2020) 052207.
	\bibitem{FT}
	L.D. Faddeev and L.A. Takhtajan,
	Hamiltonian Methods in the Theory of Solitons, Springer Series in Soviet Mathematics. Springer-Verlag, Berlin, 1987.
	\bibitem{FIKN}
	A.S. Fokas, A.R. Its, A.A. Kapaev and V.Yu. Novokshenov, 
	Painleve Transcendents. The Riemann--Hilbert Approach, AMS, 2006.
	\bibitem{F16}
	A.S. Fokas, Integrable multidimensional versions of the nonlocal nonlinear Schr\"odinger equation, Nonlinearity 29 (2016) 319--324.
	
	\bibitem{GA}
	T. Gadzhimuradov and A. Agalarov, Towards a gauge-equivalent magnetic structure of the nonlocal nonlinear Schr\"odinger equation, Phys. Rev. A 93 (2016) 062124.
	
	\bibitem{G66}
	F.D. Gakhov, Boundary value problems, Dover, New York, 1966.
	
	\bibitem{GS18}
	P.G. Grinevich, P.M. Santini, The finite gap method and the an-
	alytic description of the exact rogue wave recurrence in the periodic NLS Cauchy problem. 1, Nonlinearity, 31(11) (2018) 5258--5308.
	\bibitem{HFX19}
	F.-J. He, E.-G. Fan and J. Xu, Long-Time Asymptotics for the Nonlocal MKdV Equation, Commun. Theor. Phys. 71(5) (2019) 475--488.
	\bibitem{ONIST}
	F.W.J. Olver, D.W. Lozier, R.F. Boisvert, and C.W. Clark (eds.), NIST handbook of mathematical functions, Cambridge University Press, Cambridge, 2010.
	\bibitem{I1}
	A.R. Its, Asymptotic behavior of the solutions to the nonlinear Schr\"odinger equation, and isomonodromic deformations of systems of linear differential equations, Doklady Akad. Nauk SSSR 261(1) (1981) 14--18.
	\bibitem{IU88}
	A.R. Its, A.F. Ustinov, Formulation of the scattering theory for the NLS equation with boundary conditions of the type of finite density in the soliton-free sector, Zap. Nauchn. Sem. LOMI 169 (1988) 60--67.
	
	\bibitem{KHB16}
	O. Kimmoun, H.C. Hsu, H. Branger, M.S. Li, Y.Y. Chen, C. Kharif,
	M. Onorato, E.J.R. Kelleher, B. Kibler, N. Akhmediev, A. Chabchoub, Modulation Instability and Phase-Shifted Fermi-Pasta-Ulam
	Recurrence, Scientific Reports 6 (2016) 28516.
	\bibitem{KPS09}
	C. Kharif, E. Pelinovsky, and A. Slunyaev, Rogue Waves in the Ocean,
	Springer-Verlag Berlin Heidelberg 2009.
	\bibitem{KYZ16}
	V.V. Konotop, J. Yang and D.A. Zezyulin, Nonlinear waves in PT-symmetric systems, Rev. Mod. Phys. 88	(2016) 035002.
	\bibitem{KSER19}
	A.E. Kraych, P. Suret, G. El and S. Randoux, Nonlinear Evolution of the Locally Induced Modulational Instability in Fiber Optics, Phys. Rev. Lett. 122 (2019) 054101.
	
	\bibitem{L17} 
	J. Lenells, The nonlinear steepest descent method for Riemann-Hilbert problems of low regularity, Indiana Univ. Math. 66 (2017) 1287--1332.
	
	\bibitem{Lou18}
	S.Y. Lou, Alice-Bob systems, $\hat{P}-\hat{T}-\hat{C}$ symmetry invariant and symmetry breaking soliton solutions. J. Math. Phys. 59 (2018) 083507.
	
	\bibitem{MM}
	K. T.-R. McLaughlin, P. D. Miller, The $\bar{\partial}$ steepest descent method and the asymptotic behavior of polynomials orthogonal on the unit circle with fixed and exponentially varying nonanalytic weights. \textit{Int. Math. Res. Pap.} 177 (2006) 48673.
	\bibitem{MM2}
	K. T.-R. McLaughlin, P. D. Miller. The $\bar{\partial}$ steepest descent method for orthogonal polynomials on the real line with varying weights, Int. Math. Res. Not. 2008 (2008) rnn075.
	
	\bibitem{MA16}
	M.-A. Miri, A. Al\'u, Nonlinearity-induced PT-symmetry without material gain, New J. Phys. 18 (2016) 065001.
	
	\bibitem{NMPZ84}
	S.P. Novikov, S.V. Manakov, L.P. Pitaevskii and V.E. Zakharov, Theory of Solitons, Plenum, New York, 1984.
	
	\bibitem{R21}
	M. Russo, Local and nonlocal solitons in a coupled real system of Landau-Lifshitz equations, Physica D 422 (2021) 132893.
	
	\bibitem{RS}
	Ya. Rybalko and D. Shepelsky, Long-time asymptotics for the integrable nonlocal nonlinear Schr\"odinger equation, J. Math. Phys. 60 (2019) 031504.
	\bibitem{RS20}
	Ya. Rybalko and D. Shepelsky, Defocusing nonlocal nonlinear Schr\"odinger equation with step-like boundary conditions: long-time behavior for shifted initial data, J. Math. Phys. Anal. Geom. 16(4) (2020) 418--453.
	\bibitem{RSs}
	Ya. Rybalko, D. Shepelsky, Long-time asymptotics for the integrable nonlocal nonlinear Schr\"odinger equation with step-like initial data,  J. Differential Equations 270 (2021) 694--724.
	
	\bibitem{San18}
	P.M. Santini, The periodic Cauchy problem for PT-symmetric NLS, I: the first
	appearance of rogue waves, regular behavior or blow up at finite times, J. Phys. A: Math. Theor. 51 (2018) 025201.
	
	\bibitem{YY}
	B. Yang and J. Yang, General rogue waves in the nonlocal PT-symmetric nonlinear Schr\"odinger equation, Lett. Math. Phys. 109 (2019) 945--973.
	\bibitem{Y18}
	J. Yang, Physically significant nonlocal nonlinear Schrödinger equation and its soliton solutions, Phys. Rev. E 98 (2018) 042202.
	
	\bibitem{ZG13}
	V.E. Zakharov and A.A. Gelash, Nonlinear stage of modulation instability, Phys. Rev. Lett. 111(5) (2013) 054101.
	\bibitem{ZO09}
	V.E. Zakharov, L.A. Ostrovsky, Modulation instability: The beginning, Physica D 238 (2009) 540--548.
	\bibitem{ZS73}
	V.E. Zakharov and A.B. Shabat, Interaction between solitons in a stable medium, Sov. Phys. JETP 37 (1973) 823--828.
	\bibitem{Zh89}
	X. Zhou, The Riemann-Hilbert problem and inverse scattering. SIAM J. Math. Anal. 20(4) (1989) 966--986.
	\bibitem{Zh89-sing}
	X. Zhou, Direct and inverse scattering transforms with arbitrary spectral singularities. Comm. Pure Appl. Math. 42(7) (1989) 895--938.
	\bibitem{Zh98}
	X. Zhou, $L^2$‐Sobolev space bijectivity of the scattering and inverse scattering transforms, Comm. Pure Appl. Math. 51(7) (1998) 697--731.
\end{thebibliography}
\end{document}